\documentclass[10pt,a4paper]{amsart} %options : draft,10pt,leqno ; classes : amsart, article

\usepackage[utf8]{inputenc}
\usepackage[T1]{fontenc}
\usepackage{amsmath, amssymb}
\usepackage{amsthm}
\usepackage{mathtools}   % provides \prescript 

\usepackage{bbm} %fonction indicatrice

\usepackage{hyperref}
\hypersetup{
    unicode=true,          % non-Latin characters in Acrobat’s bookmarks
    pdftitle={},    % title
    pdfauthor={},     % author
    pdfsubject={},   % subject of the document
    pdfkeywords={}{}, % list of keywords   
    pdfborder={0 0 0}
}

\usepackage{enumitem} % better enumerate
\setenumerate[0]{label=(\roman*)}   % default enumerate tag, need enumitem package

\numberwithin{equation}{section}

\newtheorem{theorem}{Theorem}[section]
\newtheorem{proposition}[theorem]{Proposition}
\newtheorem{lemma}[theorem]{Lemma}
\newtheorem{corollary}[theorem]{Corollary}

\theoremstyle{definition}
\newtheorem{definition}[theorem]{Definition}

\newtheorem{remark}{Remark}
\newtheorem*{notation}{Notation}

\mathcode`l="8000
\begingroup
\makeatletter
\lccode`\~=`\l
\DeclareMathSymbol{\lsb@l}{\mathalpha}{letters}{`l}
\lowercase{\gdef~{\ifnum\the\mathgroup=\m@ne \ell \else \lsb@l \fi}}
\endgroup
\DeclareMathOperator{\Grass}{Grass}

\DeclareMathOperator{\GL}{GL}
\DeclareMathOperator{\SL}{SL}
\DeclareMathOperator{\End}{End}
\DeclareMathOperator{\Mat}{\mathcal{M}}
\DeclareMathOperator{\Supp}{Supp}
\DeclareMathOperator{\diag}{diag}

\DeclareMathOperator{\dang}{d_\measuredangle}

\newcommand{\dd}{\,\mathrm{d}}

\newcommand{\pp}{\boxplus}
\newcommand{\ff}{*}
\newcommand{\mm}{\boxminus}

\newcommand{\R}{\mathbb{R}}
\renewcommand{\C}{\mathbb{C}}
\newcommand{\Z}{\mathbb{Z}}
\renewcommand{\H}{\mathbb{H}}
\newcommand{\N}{\mathbb{N}}
\newcommand{\Q}{\mathbb{Q}}

\newcommand{\g}{\mathfrak{g}}
\newcommand{\h}{\mathfrak{h}}

\newcommand{\sW}{\mathsf{W}}

\newcommand{\cG}{\mathcal{G}}
\newcommand{\cT}{\mathcal{T}}

\newcommand{\eps}{\varepsilon}

\newcommand{\abs}[1]{\lvert#1\rvert}    % valeur absolue
\newcommand{\norm}[1]{\lVert#1\rVert}   % norme
\DeclareMathOperator{\Nbd}{Nbd}

\newcommand\dash{\nobreakdash-\hspace{0pt}}

\renewcommand{\phi}{\varphi}
\DeclareMathOperator{\ev}{ev}
\DeclareMathOperator{\Span}{Span}
\DeclareMathOperator{\rank}{rank}

\DeclareMathOperator{\indic}{\mathbbm{1}}
\DeclareMathOperator{\Dirac}{\mathcal{D}}

\DeclareMathOperator{\Sym}{Sym}

\newcommand{\Tbb}{\mathbb{T}}
\newcommand{\Gbb}{\mathbb{G}}
\newcommand{\Hbb}{\mathbb{H}}
\newcommand{\Fbb}{\mathbb{F}}

\newcommand{\Abb}{\mathbb{A}}

\newcommand{\afr}{\mathfrak{a}}
\newcommand{\bfr}{\mathfrak{b}}
\newcommand{\gfr}{\mathfrak{g}}

\newcommand{\Ecal}{\mathcal{E}}
\newcommand{\Ncal}{\mathcal{N}}

\newcommand{\transp}[1]{\prescript{t}{}{\!#1}} % transposé

\begin{document}

\title{Linear random walks on the torus}

\author{Weikun He}
\thanks{W.H. is supported by ERC grant ErgComNum 682150.}
\address{Einstein Institute of Mathematics, The Hebrew University of Jerusalem, Jerusalem 91904, Israel.}
\email{weikun.he@mail.huji.ac.il}

\author{Nicolas de Saxcé}
\thanks{}
\address{CNRS -- Université Paris 13, LAGA, 93430 Villetaneuse, France.}
%LAGA, UMR 7539, CNRS, Université Paris 13 -- Sorbonne Paris-Cité,
%Université Paris 8, 93430 Villetaneuse, France.}
\email{desaxce@math.univ-paris13.fr}



\date{}

\begin{abstract}
We prove a quantitative equidistribution result for linear random walks on the torus, similar to a theorem of Bourgain, Furman, Lindenstrauss and Mozes, but without any proximality assumption.
\end{abstract}

\maketitle

\section{Introduction}

The goal of the present paper is to study the equidistribution of linear random walks on the torus.
We are given a probability measure $\mu$ on the group $\SL_d(\Z)$ of integer matrices with determinant one, and consider the associated random walk $(x_n)_{n\geq 0}$ on the torus $\Tbb^d=\R^d/\Z^d$, starting from a point $x_0$ in $\Tbb^d$, and moving at step $n$ following a random element $g_n$ with law $\mu$:
\[ x_n = g_n x_{n-1} = g_n\dots g_1 x_0.\]
We say that the measure $\mu$ on $\SL_d(\Z)$ has some finite exponential moment if there exists $\eps>0$ such that
\[ \int \norm{g}^\eps \dd\mu(g) < \infty,\]
where $\norm{\ }$ denotes an arbitrary norm on $\Mat_d(\R)$, the space of $d \times d$ matrices with real coefficients.
Our goal is to prove the following theorem.

\begin{theorem}[Equidistribution on the torus]
\label{thm:easy}
Let $d \geq 2$. Let $\mu$ be a probability measure on $\SL_d(\Z)$.
Denote by $\Gamma$ the subsemigroup generated by $\mu$, and by $\Gbb < \SL_d$ the Zariski closure of $\Gamma$.
Assume that:
\begin{enumerate}[label=(\alph*)]
\item \label{it:SA1} The measure $\mu$ has a finite exponential moment; 
\item \label{it:SA2} The only subspaces of $\R^d$ preserved by $\Gamma$ are $\{0\}$ and $\R^d$;
\item \label{it:SA3} The algebraic group $\Gbb$ is Zariski connected. % and semisimple.
%\comm{On peut se passer de l'hypothèse que $\Gbb$ est semisimple, qui découle de l'irréductibilité de l'action et du fait que $\Gbb$ est engendré par des éléments de $\SL_d(\Z)$.}
\end{enumerate} 
Then, for every irrational point $x$ in $\Tbb^d$, the sequence of measures $(\mu^{*n} * \delta_x)_{n \geq 1}$ converges to the Haar measure in the weak-$*$ topology.
\end{theorem}

With an additional proximality assumption, this theorem was proved a decade ago by Bourgain, Furman, Lindenstrauss and Mozes \cite{BFLM}, and we follow their approach to this problem, via a study of the Fourier coefficients of the law at time $n$ of the random walk on $\Tbb^d$.
Recently, Benoist and Quint~\cite{bq1,bq2,bq3} have developed another approach to study equidistribution of random walks on homogeneous spaces.
Their results are far more general, but imply in particular that, under the above assumptions, the Ces\`aro averages
\(\frac{1}{n}\sum_{k=1}^n \mu^{*k} * \delta_x\)
converge to the Haar measure on $\Tbb^d$.
Besides removing the Cesàro average in their result, one advantage of the Fourier analytic method used here is that it yields a quantitative statement, giving a speed of convergence of the random walk, in terms of the diophantine properties of the starting point $x$, see \cite[Theorem~A]{BFLM}.

\begin{theorem}[Quantitative equidistribution on the torus]
\label{thm:main}
Under the assumptions of Theorem~\ref{thm:easy}, let $\lambda_1$ denote the top Lyapunov exponent associated to $\mu$.\\
Given $\lambda \in {(0, \lambda_1)}$, there exists a constant $C = C(\mu, \lambda) > 0$ such that for every $x \in \Tbb^d$ and every $t \in {(0,1/2)}$, 
if for some $a \in \Z^d \setminus \{0\}$, 
\[\abs{\widehat{\mu^{*n}*\delta_x}(a)} \geq t \quad \text{and} \quad n \geq C \log\frac{\norm{a}}{t},\]
then there exists $q \in \Z_{> 0}$ and $x' \in (\frac{1}{q}\Z^d) / \Z^d$ such that
\[q \leq \Bigl(\frac{\norm{a}}{t}\Bigr)^C \quad \text{and} \quad d(x,x') \leq e^{-\lambda n} .\]
\end{theorem}

One of our motivations for removing any proximality assumption from this theorem was to generalize a theorem of Bourgain and Varjú on expansion in $\SL_d(\Z/q\Z)$, where $q$ is an arbitrary integer, to more general simple $\Q$-groups.
We briefly describe this application at the end of the paper, in \S\ref{ss:exp}.

\bigskip

The general strategy for the proof of Theorem~\ref{thm:main} is similar to the one in the work of Bourgain, Furman, Lindenstrauss and Mozes, and can be briefly sketched as follows:
\begin{enumerate}[label=(\arabic*)]
\item Start with the assumption
\[
\abs{\widehat{\mu^{*n}*\delta_x}(a)} = \abs{\widehat{\mu^{*m}*\nu_{n-m}}(a)} \geq t,
\]
where $\nu_{n-m}=\mu^{*(n-m)} * \delta_x$, and $t$ can be thought of as a fixed positive number.
By applying Hölder’s inequality, one can show that $\widehat{\nu_{n-m}}(a\cdot g)$ is large for $g$ from a set $A\subset\Mat_d(\Z)$ with large $\mu_1$-measure, where $\mu_1$ is an additive convolution of $\mu^{*m}$.
\item Using a Fourier decay property of $\mu_1$ and that $\mu_1(A)$ is large, one can show that the set $a\cdot A\subset\R^d$ is evenly distributed with respect to two scales $N$ and $M$ (each of which is an exponential function of $m$), in the sense that one needs to use roughly $\big(\frac{N}{M}\big)^d$ many balls to cover the set $a\cdot A\cap B(0,N)$.
The establishment of the Fourier decay property is the main difference between our approach and that of Bourgain, Furman, Lindenstrauss and Mozes.
\item On the other hand, $\widehat{\nu_{n-m}}(b)$ is large for $b\in a\cdot A\cap B(0,N)$. Using this together with the fact that $a\cdot A\cap B(0,N)$ is evenly distributed, an ingenious argument due to Bourgain, Furman, Lindenstrauss and Mozes \cite[Proposition~7.5]{BFLM} shows there is a $\frac{1}{M}$-separated set $X\subset\Tbb^d$ such that $\nu_{n-m}$ is concentrated over a neighborhood of $X$ of size $e^{-cm}$, for some $c>0$.
\item Using the geometry of the backwards random walk, and in particular the Lyapunov exponents, one can bootstrap concentration of $\nu_{n-m}$ near the set $X$. Then, a Diophantine property implies that $X$ must be $e^{-cm}$-close to a finite set of rational points with denominators at most $e^{Cm}$.
\item Using again the backwards random walk, one deduces from the above that $\nu_0=\delta_x$ must be $e^{-\lambda n}$ close to a rational point with denominator bounded by $e^{Cm}$. In this sketch, $m\asymp C\log\frac{\norm{a}}{t}$, where $C$ is a constant depending only on the measure $\mu$.
\end{enumerate}

\bigskip

In \cite{BFLM}, the proximality assumption is used at several important places, especially in the study of the large scale structure of Fourier coefficients \cite[Phase I]{BFLM}.
Let us mention the main ingredients we had to bring into our proof in order to overcome this issue.

\smallskip

One important tool in the proof of Bourgain, Furman, Lindenstrauss and Mozes \cite{BFLM} is a discretized projection theorem, due to Bourgain \cite[Theorem~5]{Bourgain2010}, giving information on the size of the projections of a set $A\subset\R^d$ to lines.
When the random walk is not proximal, one should no longer project the set to lines, but to subspaces whose dimension equals the proximality dimension of $\Gamma$.
One approach, of course, would be to generalize Bourgain's theorem to higher dimensions, and this has been worked out by the first author \cite{he_projection}.
But it turns out that the natural generalization of Bourgain's theorem, used with the general strategy of \cite{BFLM}, only allows to deal with some special cases \cite{he_schubert}.
Here we take a different route.
Instead of working in the space $\Z^d$ of Fourier coefficients, we place ourselves in the simple algebra $E\subset\Mat_d(\R)$ generated by the random walk.
This allows us to use the results of the first author on the discretized sum-product phenomenon in simple algebras \cite{He2016}.
Thus, instead of a projection theorem, we use a result on the Fourier decay of multiplicative convolutions in simple algebras, derived in Section~\ref{sc:sumprod}, and generalizing a theorem of Bourgain for the field of real numbers \cite[Theorem~6]{Bourgain2010}.

\smallskip

Then, in order to be able to apply this Fourier estimate to the law at time $n$ of the random walk, we have to check some non-concentration conditions.
For that, we use a result of Salehi Golsefidy and Varjú \cite{SGV} on expansion modulo prime numbers in semisimple groups, combined with a rescaling argument, proved with the theory of random walks on reductive groups.
In the end, we obtain some Fourier decay theorem for the law at time $n$ of a random walk on $\SL_d(\Z)$, Theorem~\ref{thm:decay}, which, we believe, bears its own interest and, we hope, will have other applications.

\smallskip

The rest of the proof, corresponding to \cite[Phase II]{BFLM}, follows more closely the strategy of \cite{BFLM}.
But since at several points we had to find an alternative proof to avoid the use of the proximality assumption, we chose to include the whole argument, rather than refer the reader to \cite{BFLM}.
We hope that this will make the proof easier to follow.

\section{Sum-product, \texorpdfstring{$L^2$}{L2}-flattening and Fourier decay}
\label{sc:sumprod}

The main objective of this section is to prove that in a simple real algebra, multiplicative convolutions of non-concentrated measures admit a polynomial Fourier decay. The precise statement is given in Theorem~\ref{thm:fourier} below.

\bigskip

From now on, $E$ will denote a finite-dimensional real associative simple algebra, endowed with a norm $\norm{\ }$.
%We say an algebra $E$ over $\R$ is normed if it comes with a norm (denoted by $\norm{\ }$) which makes the underlining linear space a normed vector space. We will only consider finite-dimensional algebras. Hence the choice of the norm only affects constants. Thus we may assume that the underlying linear space of $E$ is $\R^d$ and the norm is Euclidean, with $d = \dim E$.
Given a finite Borel measure $\mu$ on $E$ and an integer $s\geq 1$, we write 
\[\mu^{* s} = \underbrace{\mu * \dotsm * \mu}_\text{$s$ times}\]
for the $s$-fold multiplicative convolution of $\mu$ with itself.
In order to ensure the Fourier decay of some multiplicative convolution of the measure $\mu$, we need two assumptions:
First, $\mu$ should not be concentrated around a linear subspace of $E$, and second, $\mu$ should not give mass to elements of $E$ that are too singular.

To make these requirements more precise, let us set up some notation.
For $\rho > 0$ and a point $x$ in a metric space, let $B(x,\rho)$ denote the closed ball of radius $\rho$ and centered at $x$. 
When the ambient space is not clear from the context, we indicate it by adding a subscript to the notation, as in $B_E(x,\rho )$.
For a subset $W \subset E$, let $W^{(\rho)}$ denote the $\rho$-neighborhood of $W$,
\[
W^{(\rho)} = W + B(0,\rho).
\]
For $a \in E$ define $\det_E(a)$ to be the determinant of the endomorphism $E \to E$, $x \mapsto ax$. Note that since $E$ is simple this quantity is equal to the determinant of $E \to E$, $x \mapsto xa$. 
For $\rho > 0$, define $S_E(\rho)$, the set of badly invertible elements of $E$, as
\[S_E(\rho) = \{\, x \in E \mid \abs{\det\nolimits_E(x)} \leq \rho \,\}. \]

\begin{theorem}[Fourier decay of multiplicative convolutions]
\label{thm:fourier}
Let $E$ be a normed simple algebra over $\R$ of finite dimension. Given $\kappa > 0$, there exists $s = s(E,\kappa) \in \N$ and $\eps = \eps(E,\kappa) > 0$ such that for any parameter $\tau \in {(0, \eps \kappa)}$ the following holds for any scale $\delta > 0$ sufficiently small.
Let $\mu$ be a Borel probability measure on $E$. Assume that
\begin{enumerate}
\item $\mu \bigl(E \setminus B(0,\delta^{-\eps})\bigr) \leq \delta^{\tau}$;
\item for every $x \in E$, $\mu(x + S_E(\delta^{\eps})) \leq \delta^{\tau}$;
\item for every $\rho \geq \delta$ and every proper affine subspace $W \subset E$, $\mu(W^{(\rho)}) \leq \delta^{-\eps} \rho^\kappa$.
\end{enumerate}
Then for all $\xi \in E^*$ with $\norm{\xi} = \delta^{-1}$,
%$\delta^{-1+\eps} \leq \norm{\xi} \leq \delta^{-1-\eps}$,
\[
\abs{\widehat{\mu^{* s}}(\xi)} \leq \delta^{\eps \tau}.
\]
\end{theorem}

For $E = \R$, this is due to Bourgain~\cite[Lemma 8.43]{Bourgain2010}. Li proved in~\cite{Li2018} a similar statement for the semisimple algebra $\R \oplus \dotsb \oplus \R$. 
%\marginpar{\small Ce n'est pas bien d'annoncer les choses dont on ne possède pas la démonstration !?}
While a more general statement should hold for any semisimple algebra, we do not pursue in this direction and focus in the present paper only on simple algebras.

\subsection{A sum-product theorem.}
The goal of this paragraph is a sum-product estimate in simple algebras, which is going to be used to prove an $L^2$-flattening lemma in the next paragraph.

\begin{notation}
For a subset $A$ of $E$ or of any metric space, we denote by $\Ncal(A,\delta)$ the least number of balls of radius $\delta$ that is needed to cover $A$.
\end{notation}

For $K > 1$, define the set of well invertible elements of $E$ as
\[G_{\!E}(K) = \{x \in E^\times \mid \norm{x},\norm{x^{-1}} \leq K\}.\]
Note that if $x \in G_{\!E}(K)$, the left, or right, multiplication by $x$ as a map from $E$ to itself is $O(K)$-bi-Lipschitz.
Moreover, there exists a constant $C \geq 1$
% depending only on the choice of the norm on $E$
such that for all $\rho \in {(0,1)}$, 
\[
G_{\!E}(\rho^{-1}) \subset B_E(0,\rho^{-1})\setminus S_E \bigl(\frac{\rho^{\dim E}}{C} \bigr)\]
and conversely,
\[B_E(0,\rho^{-1}) \setminus S_E(\rho) \subset G_{\!E}\bigl(C \rho^{-\dim E }\bigr).\]

\begin{proposition}[Sum-product estimate in simple algebras]
\label{pr:sumprod}
Let $E$ be a normed finite-dimensional simple algebra over $\R$ of dimension $d \geq 2$. Given $\kappa > 0$, there exists $\eps = \eps(E,\kappa) > 0$ such that the following holds for every $\delta > 0$ sufficiently small. 
Let $A$ be a subset of $E$, $\mu$ a probability measure on $E$, and $B$ a subset of $E \times E$. Assume
\begin{enumerate}
\item $A \subset B(0,\delta^{-\eps})$;
\item $\forall \rho \geq \delta$, $\Ncal(A,\rho) \geq \delta^{\eps}\rho^{-\kappa}$;
\item $\Ncal(A,\delta) \leq \delta^{-(d-\kappa)}$;
\item \label{it:spmu1} $\mu$ is supported on $G_{\!E}(\delta^{-\eps})$;
%\item \label{it:spmu2} $\forall \rho \geq \delta$, $\forall x \in E$, $\mu(B(x,\rho)) \leq \delta^{-\eps}\rho^\kappa$;
\item \label{it:spmu3} for every  proper linear subspace $W<E$, $\forall \rho \geq \delta$, $\mu(W^{(\rho)}) \leq \delta^{-\eps}\rho^\kappa$;
\item \label{it:spB} $\mu \otimes \mu (B) \geq \delta^\eps$.
\end{enumerate}
Then for every $a_\star, b_\star \in G_{\!E}(\delta^{-\eps})$, there exists $(a,b) \in B \cup \{(a_\star,b_\star)\}$ such that
\[ \Ncal(a_\star^{-1}aA+Ab_\star^{-1}b,\delta) \geq \delta^{-\eps} \Ncal(A,\delta).\]
\end{proposition}

The idea of the proof is to consider the action of $E \times E$ on $E$ by left and right multiplication and to apply a sum-product theorem \cite[Theorem 3]{He2016} for irreducible linear actions due to the first author. For the reader's convenience, let us recall the statement of the latter. 

\begin{theorem}[Sum-product theorem for irreducible linear actions]
\label{thm:sumprod_he}
Given a positive integer $d$ and a real number $\kappa > 0$ there exists $\eps = \eps(d,\kappa) > 0$ such that the following holds for $\delta > 0$ sufficiently small. 
Let $X$ be a subset of the Euclidean space $\R^d$ and $\Phi \subset \End(\R^d)$ a subset of linear endomorphisms. Assume
\begin{enumerate}
\item \label{it:spX1} $X \subset B_{\R^d}(0,\delta^{-\eps})$;
\item for all $\rho \geq \delta$, $\Ncal(X,\rho) \geq \delta^{\eps}\rho^{-\kappa}$;
\item $\Ncal(X,\delta) \leq \delta^{-(d-\kappa)}$;
\item \label{it:spPhi1} $\Phi \subset B_{\End(\R^d)}(0,\delta^{-\eps})$;
\item \label{it:spPhi2} for all $\rho \geq \delta$, $\Ncal(\Phi,\rho) \geq \delta^{\eps}\rho^{-\kappa}$;
\item \label{it:spPhi3} for every proper linear subspace $W \subset \R^d$, there is $\phi \in \Phi$ and $w \in B_W(0,1)$ such that $d(\phi w,W) \geq \delta^\eps$.
\end{enumerate}
Then 
\[\Ncal(X+X,\delta)+\max_{\phi \in \Phi}\Ncal(X + \phi X,\delta) \geq \delta^{-\eps} \Ncal(X,\delta).\]
\end{theorem}
Here, of course, $\phi X$ denote the set $\{\phi x \mid x \in X\}$

\begin{proof}[Proof of Proposition~\ref{pr:sumprod}]
%Changing the norm on $E$ only affects the constants hence we may assume that the norm on $E$ is Euclidean and identify $E$ with $\R^d$ where $d = \dim(E)$. Note that for all $x, y \in E$, 
%\[\norm{xy} \ll \norm{x} \norm{y}.\]
In this proof the implied constants in the Vinogradov or Landau notation may depend on $E$.
We may assume without loss of generality that $B \subset \Supp(\mu) \times \Supp(\mu)$.
This implies that for all $(a,b) \in B$, 
\[\norm{a},\norm{a^{-1}},\norm{b},\norm{b^{-1}} \leq \delta^{-\eps},\]
and consequently the multiplication on the left or right by $a$ or $b$ is a $\delta^{-O(\eps)}$-bi-Lipschitz endomorphism of $E$.

For $(a,b)$ in $B$, define $\phi(a,b) \in \End(E)$ by
\[\forall x \in E,\; \phi(a,b)x = - a^{-1} a_\star x b_\star^{-1} b.\]
We would like to apply the previous theorem to $X = A$ and $\Phi = \{\phi(a,b) \in \End(E) \mid a,b \in B\}$.
We claim that the assumptions of Theorem~\ref{thm:sumprod_he} hold with $O(\eps)$ in the place of $\eps$. 
Hence there is $\eps_1 > 0$ such that when $\eps > 0$ is small enough, we have either 
\[\Ncal(A + A,\delta) \geq \delta^{-\eps_1} \Ncal(A,\delta)\]
in which case we are done, or there exists $(a,b) \in B$ such that
\[\Ncal(A + \phi(a,b) A ,\delta) \geq \delta^{-\eps_1} \Ncal(A,\delta).\]
In the latter case we conclude by multiplying the set above by $a_\star^{-1}a$ on the left,
\[\Ncal(a_\star^{-1} a A - A b_\star^{-1} b,\delta) \geq \delta^{O(\eps)} \Ncal(A + \phi(a,b)A,\delta) \geq \delta^{-\eps_1 + O(\eps)} \Ncal(A,\delta).\]
It remains to check the assumptions in Theorem~\ref{thm:sumprod_he}.
Items \ref{it:spX1}--\ref{it:spPhi1} are immediate.
To check the remaining assumptions, write, for $b \in E$,
\[B_1(b) = \{a \in E \mid (a,b) \in B\}.\]
By assumption~\ref{it:spB} of the proposition we are trying to prove, we can pick $b_0 \in E$ such that 
\[\mu(B_1(b_0)) \geq \delta^\eps.\]
From the inequalities
\[ \norm{a - a'} \ll \norm{a} \norm{a'} \norm{a'^{-1} - a^{-1}}\]
and 
\[\norm{a'^{-1} - a^{-1}} = \norm{\bigl(\phi(a,b_0) - \phi(a',b_0)\bigr)(a_\star^{-1}b_0^{-1}b_\star)} \leq \norm{\phi(a,b_0) - \phi(a',b_0)} \norm{a_\star^{-1}b_0^{-1}b_\star},\]
we see that the map $a \mapsto \phi(a, b_0)$ is $\delta^{-O(\eps)}$-bi-Lipschitz on $B_1(b_0)$.
Thus item~\ref{it:spPhi2} follows from assumption~\ref{it:spmu3} of Proposition~\ref{pr:sumprod}.

Finally, assume for contradiction that item~\ref{it:spPhi3} fails with $\delta^{C\eps}$ in the place of $\delta^\eps$. Namely, there is a linear subspace $W_0 \subset E$ of intermediate dimension $0 < k <d$ such that $\forall (a,b) \in B$, $d(W_0, \phi(a,b)W_0) \leq \delta^{C\eps}$, where $d$ denotes the distance on the the Grassmannian $\Grass(k,E)$ of $k$-dimensional subspaces in $E$ defined by
\[d(W,W') = \min_{w \in B_W(0,1)} d(w,W') = \min_{w' \in B_{W'}(0,1)} d(w',W).\]
In particular, for $a, a' \in B_1(b_0)$, we have
\[
d(W_0, \phi(a,b_0) W_0) \leq \delta^{C\eps} \text{ and } d(W_0, \phi(a',b_0)W_0) \leq \delta^{C\eps}.
\]
Multiplying the second inequality on the left by $a^{-1}a'$, we obtain
\[d(a^{-1}a'W_0, \phi(a,b_0) W_0) \leq \delta^{(C-O(1))\eps}.\]
By the triangle inequality,
\[d(W_0, a^{-1} a' W_0) \leq \delta^{(C-O(1))\eps},\]
which means 
\begin{equation}
\label{eq:dgwW0}
\forall g \in a^{-1}B_1(b_0),\, \forall w \in W_0,\; d(gw,W_0) \leq \norm{gw}d(gW_0,W_0) \leq \delta^{(C-O(1))\eps}\norm{w}.
\end{equation}
Observe that the assumption~\ref{it:spmu3} of Proposition~\ref{pr:sumprod} implies that the subset $B_1(b_0)$ is $\delta^{O(\eps)}$-away from linear subspaces, 
that is, for any proper linear subspace $V \subset E$, there is $a \in B_1(b_0)$ such that $d(a,V) \geq \delta^{O(\eps)}$.
Hence so is the subset $a^{-1}B_1(b_0)$. 
By \cite[Lemma 8]{He2016}, every $x \in B(0,1)$ can be expressed as a linear combination of $d$ elements from $a^{-1}B_1(b_0)$ with coefficients bounded by $\delta^{-O(\eps)}$.
Thus \eqref{eq:dgwW0} implies that
\[\forall x \in B_E(0,1),\, \forall w \in B_{W_0}(0,1),\; d(xw,W_0) \leq \delta^{(C-O(1))\eps}.\]
We can apply the same argument to right multiplication and obtain similarly,
\[\forall x \in B_E(0,1),\, \forall w \in B_{W_0}(0,1),\; d(wx,W_0) \leq \delta^{(C-O(1))\eps}.\]
Consider the map $f \colon \Grass(k,E) \to \R$ defined by
\[f(W) = \iint_{B_E(0,1) \times B_{W}(0,1)} \big(d(xw,W)  + d(wx,W)\big) \dd x \dd w.\]
On the one hand, from the above, $f(W_0) \leq \delta^{(C-O(1))\eps}$. On the other hand, $f$ is continuous and defined on a compact set. It never vanishes for the reason that a zero of $f$ must be a two-sided ideal of $E$ contradicting the simplicity of $E$. Hence $f$ has a positive minimum on $\Grass(k,E)$. We obtain a contradiction if $C$ is chosen large enough, proving our claim regarding item~\ref{it:spPhi3}.
\end{proof}

\subsection{\texorpdfstring{$L^2$}{L2}-flattening.}
The aim of this subsection is to prove a sum-product $L^2$-flattening lemma for simple algebras.

We shall consider both additive and multiplicative convolutions between measures or functions on $E$.
To avoid confusion, we shall use the usual symbol $\ff$ to denote multiplicative convolution and the symbol $\pp$ to denote additive convolution.
In the same fashion, for finite Borel measures $\mu$ and $\nu$ on E, we define $\mu \mm \nu$ to be the push forward measure of $\mu \otimes \nu$ by the map $(x,y) \mapsto x-y$.

For a Borel set $A \subset E$, denote by $\abs{A}$ the Lebesgue measure of $A$. 
For $\delta > 0$, define $P_\delta = \abs{B(0,\delta)}^{-1}\indic_{B(0,\delta)}$. 
For absolutely continuous measures such as $\mu \pp P_\delta$, by abuse of notation, we write $\mu \pp P_\delta$ to denote both the measure and the density function.
In particular, $\norm{\mu \pp P_\delta}_2$ denotes the $L^2$-norm of the density function with respect to the Lebesgue measure.
For $x\in E$, we write $\Dirac_x$ to denote the Dirac measure at the point $x$.

\begin{proposition}[\texorpdfstring{$L^2$}{L2}-flattening]
\label{pr:apla}
Let $E$ be a normed finite-dimensional simple algebra over $\R$ of dimension $d \geq 2$. Given $\kappa > 0$, there exists $\eps = \eps(E,\kappa) > 0$ such that the following holds for $\delta > 0$ sufficiently small. 
Let $\mu$ be a Borel probability measure on $E$. Assume that
\begin{enumerate}
\item $\mu$ is supported on $G_{\!E}(\delta^{-\eps})$;
\item $\delta^{-\kappa} \leq \norm{\mu \pp P_\delta}_2^2 \leq \delta^{-d + \kappa}$;
%\item $\forall\rho\geq\delta$, $\mu(N_\rho) \leq \delta^{-\eps}\rho^{\kappa}$;
\item \label{it:flatnc2} for every proper linear subspace $W<E$, $\forall \rho \geq \delta$, $\mu(W^{(\rho)}) \leq \delta^{-\eps}\rho^\kappa$.
\end{enumerate}
Then,
\[ \norm{(\mu \ff \mu \mm \mu \ff \mu) \pp P_\delta}_2 \leq \delta^{\eps}\norm{\mu \pp P_\delta}_2.\]
\end{proposition}

\begin{remark}
If $E = \R$ the same holds if condition~\ref{it:flatnc2} is replaced by 
\begin{equation}
\label{eq:flatnc1} 
\forall \rho \geq \delta,\, \forall x \in E,\; \mu(B(x,\rho)) \leq \delta^{-\eps}\rho^\kappa.
\end{equation}
Note also that when $\dim(E) \geq 2$, property \eqref{eq:flatnc1} is implied by condition~\ref{it:flatnc2}.
\end{remark}

\begin{remark}
\label{rk:mu1half}
We shall apply this proposition to measures that are not probability measures. It is clear that by making $\eps$ slightly smaller, the same statement holds for measures $\mu$ with total mass $\mu(E) \in {[\delta^{\eps}, \delta^{-\eps}]}$.
\end{remark}

For the proof of the proposition, we need basic tools from additive combinatorics, in the context of discretized sets.
We briefly recall the definitions and results we shall use, and refer the reader to Tao \cite[Section 6]{Tao2008} for more detail on the subject.

Let $\delta > 0$ be a small positive real number.
We will refer to it as the scale.
By a $\delta$-discretized set we mean a union of balls of radius $\delta$.
Note that if $A$ is a $\delta$-discretized set then $\Ncal(A,\delta) \asymp_E \delta^{-d} \abs{A}$. 
The following discretized version of Ruzsa's triangle inequality can be deduced easily from \cite[Lemma 2.6]{TaoVu}.

\begin{lemma}[Ruzsa's triangle inequality]
Let $A$, $B$, $C$ be bounded subsets of $E$.
\[\Ncal(B,\delta) \Ncal(A - C,\delta) \ll_E  \Ncal(A - B,\delta) \Ncal(B - C, \delta) .\]
\end{lemma}

\begin{definition}[Additive energy]
The \emph{additive energy} between two subsets $A$, $B \subset E$ at scale $\delta$ is
\[
\Ecal_\delta(A,B) = \Ncal\bigl(\{\, (a,a',b,b') \in A \times A \times B \times B \mid \norm{a + b - a'-b'} \leq \delta \,\} ,\delta\bigr).
\]
If $A$ and $B$ are $\delta$-discretized sets, then % it is not difficult to see (cf. \cite[Page 579]{Tao2008}) that
\[
 \Ecal_\delta(A,B) \asymp_E \delta^{-3d} \norm{\indic_A \pp \indic_B}_2^2.
\]
\end{definition}

The additive energy $\Ecal_\delta(A,B)$ quantifies the amount of additive relations between the sets $A$ and $B$.
One very important instance of this is the following theorem, taken from \cite[Theorem~6.10]{Tao2008}, which is a non-commutative version of the Balog-Szemerédi-Gowers lemma, 

\begin{theorem}[Balog-Szemerédi-Gowers lemma]
Let $K\geq 1$ be a parameter. 
Let $A,B$ be two subsets of $E$. If 
\[\Ecal_\delta(A,B) \geq \frac{1}{K}\Ncal(A,\delta)^{3/2} \Ncal(B,\delta)^{3/2}\]
then there are subsets $A'\subset A$ and $B' \subset B$ satisfying $\Ncal(A',\delta) \geq K^{-O(1)} \Ncal(A,\delta)$ and $\Ncal(B',\delta) \geq K^{-O(1)} \Ncal(B,\delta)$
and 
\[\Ncal(A' + B',\delta) \ll_E K^{O(1)} \Ncal(A',\delta) \Ncal(B',\delta).\]
\end{theorem}

\smallskip

\begin{proof}[Proof of Proposition~\ref{pr:apla}]
In this proof, the implied constants in the Landau or Vinogradov notation depend on the algebra structure of $E$ as well as the choice of norm on it. We use the following rough comparison notation : for positive quantities $f$ and $g$, we write $f \lesssim g$ if $f \leq \delta^{-O(\eps)} g$ and $f \sim g$ for $f \lesssim g$ and $g \lesssim f$.
For instance, if $a \in G_{\!E}(\delta^{-\eps})$ then $\abs{\det_E(a)} \sim 1$.

To simplify notation, we shall also use the shorthand $\mu_\delta = \mu \pp P_\delta$.
Now assume for a contradiction that the conclusion of the proposition does not hold, namely
\begin{equation}
\label{eq:counterFlat0}
\norm{(\mu \ff \mu \mm \mu \ff \mu) \pp P_\delta}_2 \geq \delta^{\eps}\norm{\mu_\delta}_2.
\end{equation}

\noindent\underline{Step 0:} Compare the $L^2$-norms of $(\mu \ff \mu \mm \mu \ff \mu) \pp P_\delta$ and $\mu \ff \mu_\delta \mm \mu_\delta \ff \mu$.\\
For $x \in E$, write
\begin{align*}
&(\mu \ff \mu \mm \mu \ff \mu \pp P_\delta)(x)\\
=& \abs{B(0,\delta)}^{-1} \mu^{\otimes 4}\{(a,b,c,d) \mid ab - cd \in B(x, \delta) \}\\
\leq& \abs{B(0,\delta)}^{-1} (\mu^{\otimes 4} \otimes P_\delta^{\otimes 2})\{(a,b,c,d,y,z) \mid a(b+y) - (c+z)d \in B(x, \delta^{1-2\eps})\}\\
\lesssim& \abs{B(0,\delta^{1-2\eps})}^{-1} (\mu^{\otimes 4} \otimes P_\delta^{\otimes 2}) \{(a,b,c,d,y,z) \mid a(b+y) - (c+z)d \in B(x, \delta^{1-2\eps})\}\\
=& (\mu \ff \mu_\delta \mm \mu_\delta \ff \mu \pp P_{\delta^{1-2\eps}})(x).
\end{align*}
Above at the sign $\leq$, we used the assumption that $\Supp(\mu) \subset B(0,\delta^{-\eps})$.
Therefore, by Young's inequality,
\[
\norm{\mu \ff \mu \mm \mu \ff \mu \pp P_\delta}_2 \lesssim \norm{\mu \ff \mu_\delta \mm \mu_\delta \ff \mu \pp P_{\delta^{1-2\eps}}}_2 \leq \norm{\mu \ff \mu_\delta \mm \mu_\delta \ff \mu}_2.
\]
To conclude step 0, we deduce from the above and \eqref{eq:counterFlat0} that
\begin{equation}
\label{eq:counterFlat}
\norm{\mu \ff \mu_\delta \mm \mu_\delta \ff \mu}_2 \gtrsim \norm{\mu_\delta}_2.
\end{equation}

\smallskip

\noindent\underline{Step 1:} Discretize the measure $\mu$ using dyadic level sets.\\
%By \cite[Lemma A.4]{BISG} (or \cite[Lemma 4.4]{LindenstraussSaxce}),
It is easy to check that there exist $\delta$-discretized sets $A_i \subset B(0,\delta^{-\eps})$, $i \geq 0$ such that $A_i$ is empty for $i \gg \log\frac{1}{\delta}$, and
\begin{equation}
\label{eq:mullAi}
\mu_\delta
\ll \sum_{i \geq 0} 2^i \indic_{A_i}
\ll \mu_{3\delta} + 1.
\end{equation}

\smallskip

\noindent\underline{Step 2:} Pick a popular level in order to transform \eqref{eq:counterFlat} into a lower bound on the additive energy between two $\delta$-discretized sets.\\
We have
\[\mu \ff \mu_\delta \mm \mu_\delta \ff \mu = \iint_{E \times E} (\Dirac_{\!a} \ff \mu_\delta) \mm (\mu_\delta \ff \Dirac_b) \dd \mu(a) \dd \mu(b).\]
From the left inequality in \eqref{eq:mullAi},
\[\mu \ff \mu_\delta \mm \mu_\delta \ff \mu \ll \sum_{i,j \geq 0} 2^{i + j} \iint (\Dirac_{\!a} \ff \indic_{A_i}) \mm (\indic_{A_j} \ff \Dirac_b) \dd \mu(a) \dd \mu(b).\]
Observe that $\Dirac_{\!a} \ff \indic_{A_i} = \abs{\det_E(a)}^{-1}\indic_{aA_i}$ and $\indic_{A_j} \ff \Dirac_b = \abs{\det_E(b)}^{-1}\indic_{A_jb}$. Hence
\[\mu \ff \mu_\delta \mm \mu_\delta \ff \mu \ll \sum_{i,j \geq 0} 2^{i + j} \iint  \frac{\indic_{aA_i} \mm \indic_{A_jb}}{\abs{\det_E(a) \det_E(b)}} \dd \mu(a) \dd \mu(b).\]
By \eqref{eq:counterFlat}, the triangle inequality and the assumption that $\mu$ is supported on $G_{\!E}(\delta^{-\eps})$,
\[\sum_{i,j \geq 0} 2^{i + j} \iint  \norm{\indic_{aA_i} \mm \indic_{A_jb}}_2 \dd \mu(a) \dd \mu(b) \gtrsim \norm{\mu_\delta}_2,\]
There are at most $O\bigl( (\log \frac{1}{\delta})^2 \bigr) \lesssim 1$ terms in this sum.
Hence by the pigeonhole principle, there exist $i \geq 0$ and $j \geq 0$ such that
\begin{equation}
\label{eq:defiandj}
2^{i + j} \iint \norm{\indic_{aA_i} \mm \indic_{A_jb}}_2 \dd \mu(a) \dd \mu(b) \gtrsim \norm{\mu_\delta}_2.
\end{equation}

From now on we fix such $i$ and $j$.
By the right-hand inequality in \eqref{eq:mullAi}, we find % and \cite[Lemma A.5]{BISG}
\[2^i \abs{A_i}^{1/2} = \norm{2^i \indic_{A_i}}_2 \ll \norm{\mu_{3\delta}}_2 +1
 \ll \norm{\mu_\delta}_2,\]
so that for all $a,b \in G_{\!E}(\delta^{-O(\eps)})$,
\begin{equation}
\label{eq:aAi2leqmu}
%2^i \abs{\det\nolimits_E(a)}^{1/2} \abs{A_i}^{1/2} = 
\norm{2^i \indic_{aA_i}}_2 \lesssim \norm{\mu_\delta}_2 \text{ and } \norm{2^j \indic_{A_j b}}_1 \lesssim 1.
\end{equation}
Hence by Young's inequality,
\[\forall a,b \in G_{\!E}(\delta^{-O(\eps)}),\; 2^{i+j}\norm{\indic_{aA_i} \mm \indic_{A_jb}}_2  \lesssim \norm{\mu_\delta}_2.\]
This combined with \eqref{eq:defiandj} implies
%(use an inverse Chebyshev’s inequality)
that the set 
\[B_0 = \bigl\{\, (a,b) \in G_{\!E}(\delta^{-O(\eps)}) \times  G_{\!E}(\delta^{-O(\eps)}) \mid  2^{i+j}\norm{\indic_{aA_i} \mm \indic_{A_jb}}_2  \geq \delta^{O(\eps)} \norm{\mu_\delta}_2 \,\bigr\}\]
has measure $\mu \otimes \mu (B_0) \gtrsim 1$. 
For $c = (a,b) \in B_0$, using \eqref{eq:aAi2leqmu}, we find
\[\norm{\indic_{a A_i} \mm \indic_{A_j b}}_2 \gtrsim \abs{a A_i}^{1/2} \abs{A_j b},\]
and switching the role of $a A_i$ and $A_j b$,
\[\norm{\indic_{a A_i} \mm \indic_{A_j b}}_2 \gtrsim \abs{a A_i} \abs{A_j b}^{1/2}.\]
Hence,
\[\norm{\indic_{a A_i} \mm \indic_{A_j b}}^2_2 \gtrsim \abs{a A_i}^{3/2} \abs{A_j b}^{3/2}.\]
Note that $a A_i$ and $A_j b$ are $\delta^{1 - O(\eps)}$-discretized sets. 
Hence the last inequality translates to
\begin{equation}
\label{eq:energyAiAj}
\Ecal_\delta(a A_i, -A_j b) \gtrsim \Ncal(a A_i, \delta)^{3/2} \Ncal(A_j b, \delta)^{3/2}.
\end{equation}

\smallskip

\noindent\underline{Step 3:} Apply an argument of Bourgain~\cite[Proof of Theorem C]{Bourgain2009} sometimes known as the additive-multiplicative Balog-Szemerédi-Gowers theorem.\\
For subsets $A, A' \subset E$ we write $A \approx A'$ if 
\[\Ncal(A - A', \delta) \lesssim \Ncal(A,\delta)^{1/2} \Ncal(A',\delta)^{1/2}.\]
Ruzsa's triangle inequality recalled above can be summarized as : the relation $\approx$ is transitive\footnote{Strictly speaking, $\approx$ is not relation, because it involves an implicit constant in the $\lesssim$ notation.}, i.e. $A \approx A'$ and $A' \approx A''$ implies $A' \approx A''$.

By the Balog-Szemerédi-Gowers lemma applied to \eqref{eq:energyAiAj}, for every $c =(a,b) \in B_0$, there exists $A_c \subset A_i$ and $A'_c \subset A_j$ such that $\Ncal(A_c,\delta) \gtrsim \Ncal(A_i,\delta)$ and $\Ncal(A'_c,\delta) \gtrsim \Ncal(A_j,\delta)$ and
\begin{equation}
\label{eq:aAcApcb}
a A_c \approx A'_c b.
\end{equation}
%Note that this implies in particular that $\Ncal(A_i,\delta) \sim \Ncal(A_j,\delta)$.

By taking $\delta$-neighborhoods if necessary, we may assume that $A_c$ and $A'_c$ are $\delta$-discretized sets. Write $X = A_i \times A_j \subset \R^{2d}$ and $X_c = A_c \times A'_c \subset X$. From the Cauchy-Schwarz inequality applied to the function $x \mapsto \int_{B_0} \indic_{X_c}(x) \dd \mu^{\otimes 2}(c)$, we infer that
\[\iint_{B_0 \times B_0} \abs{X_c \cap X_d} \dd \mu^{\otimes 2}(c) \dd \mu^{\otimes 2}(d) \gtrsim \abs{X}.\]
By the pigeonhole principle, there exists $c_\star \in B_0$  and $B_1 \subset B_0$ such that
\[\mu^{\otimes 2}(B_1) \gtrsim  \mu^{\otimes 2}(B_0) \gtrsim 1\]
and for all $c \in B_1$, $\abs{X_{c_\star} \cap X_c} \gtrsim \abs{X}$. 
Abbreviate $A_{c_\star}$ as $A_\star$ and $A'_{c_\star}$ as $A'_\star$. 
We then have, for every $c \in B_1$, 
\begin{equation}
\label{eq:AstarCap}
\Ncal(A_\star \cap A_c,\delta) \gtrsim \Ncal(A_i,\delta)
\text{ and }
\Ncal(A'_\star \cap A'_c,\delta) \gtrsim \Ncal(A_j,\delta).
\end{equation}
For $c = (a,b) \in B_1$, by the Ruzsa calculus and \eqref{eq:aAcApcb},  
\[a A_c \approx A'_c b \approx a A_c.\]
Since $a \in G_{\!E}(\delta^{-\eps})$, this implies
\[A_c \approx A_c.\]
Using \eqref{eq:AstarCap} and the definition of the symbol $\approx$, we get $A_\star \cap A_c \approx A_c$ and for the same reason $A_\star \cap A_c \approx A_\star$.
Hence
\[aA_\star \approx a(A_\star \cap A_c) \approx a A_c \approx A'_c b \approx (A'_\star \cap A'_c) b \approx A'_\star b.\]
On the other hand, writing $c_\star = (a_\star, b_\star)$, we have $a_\star A_\star \approx A'_\star b_\star$. Hence 
$a_\star A_\star b_\star^{-1}b \approx A'_\star b$ and then $a_\star A_\star b_\star^{-1}b \approx a A_\star$ and finally
\begin{equation}
\label{eq:countersumprod}
\Ncal(a_\star^{-1}a A_\star - A_\star b_\star^{-1} b ,\delta) \lesssim \Ncal(A_\star,\delta), \quad \forall (a,b) \in B_1 \cup \{(a_\star,b_\star)\}.
\end{equation}

\smallskip

\noindent\underline{Step 4:} Apply the sum-product estimate Proposition~\ref{pr:sumprod}.\\ % to get a contradiction.\\
We claim that the assumptions of Proposition~\ref{pr:sumprod} are satisfied by the set $A_\star$, the set $B_1$ and the measure $\mu$ for the parameters $\kappa/2$ in the place of $\kappa$ and $O(\eps)$ in the place of $\eps$. 
Indeed, using Young's inequality and remembering \eqref{eq:aAi2leqmu}, we obtain
\[
\norm{\mu_\delta}_2
\lesssim 2^{i+j} \norm{\indic_{a_\star A_i} \mm \indic_{A_j b_\star}}
\leq 2^i\abs{a_\star A_i}^{1/2} 2^j\abs{A_j b_\star}
\lesssim \norm{\mu_\delta}_2.
\]
Hence $2^i \abs{A_i}^{1/2} \sim \norm{\mu_\delta}_2$ and $2^j \abs{A_j} \sim 1$.
Inversing the roles of $A_i$ and $A_j$, we get also $2^i \abs{A_i} \sim 1$.
Thus, 
\begin{equation}
\label{eq:2iAimu2}
\abs{A_i} \sim \norm{\mu_\delta}_2^{-2}\text{ and }2^i \sim \norm{\mu_\delta}_2^2.
\end{equation}
Hence
\[\abs{A_\star} \lesssim \abs{A_i} \lesssim \norm{\mu_\delta}_2^{-2} \leq \delta^{\kappa},\]
which implies, as $A_\star$ is $\delta$-discretized,
\[\Ncal(A_\star,\delta) \lesssim \delta^{-d + \kappa}.\]

Moreover, let $\rho \geq \delta$ and $x \in E$.
Since $\mu_{3\delta} = \mu \pp P_{3\delta}$, inequality \eqref{eq:flatnc1} implies
\[\mu_{3\delta}\bigl(B(x,\rho)\bigr) \lesssim \rho^\kappa.\]
By \eqref{eq:2iAimu2} and \eqref{eq:mullAi},
\[\frac{\abs{A_i \cap B(x,\rho)}}{\abs{A_i}} \lesssim 2^i \abs{A_i \cap B(x,\rho)} \ll \mu_{3\delta}\bigl(B(x,\rho)\bigr).\]
But $A_\star \subset A_i$ and $\abs{A_\star} \gtrsim \abs{A_i}$, hence
\[\frac{\abs{A_\star \cap B(x,\rho)}}{\abs{A_\star}} \lesssim \frac{\abs{A_i \cap B(x,\rho)}}{\abs{A_i}} \lesssim \rho^{\kappa}.\]
It follows that for all $\rho \geq \delta$,
\[\Ncal(A_\star,\rho) \gtrsim \rho^{-\kappa}.\]
The verification of the other assumptions in Proposition~\ref{pr:sumprod} are straightforward, so we can apply Proposition~\ref{pr:sumprod}, which leads to a contradiction to \eqref{eq:countersumprod} when $\eps > 0$ is chosen small enough.
\end{proof}

\subsection{Fourier decay for multiplicative convolutions}
The goal here is to prove Theorem~\ref{thm:fourier} using iteratively the $L^2$-flattening lemma proved above.

\smallskip

%On a finite-dimensional semisimple algebra $E$, there is a map $\tr_E \colon E \to \R$ for which the bilinear form define by $(x,y) \mapsto \tr_E(xy)$ is symmetric and non-degenerate.
Let $E$ be any finite-dimensional real algebra.
The Fourier transform of a finite Borel measure $\mu$ on $E$ is the function on  the dual $E^*$ given by
\[
\forall \xi \in E^*,\; \hat\mu(\xi) = \int_E e(\xi x)\dd \mu (x).
\] 
where $e(t) = e^{2\pi i t}$ for $t \in \R$, and we simply write $E^*\times E\to \R;\ (\xi,x)\mapsto \xi x$ for the duality pairing.
The product on $E$ yields a natural right action of $E$ on $E^*$ given by
\[ \forall \xi\in E^*,\ x\in E,\ y\in E,\quad (\xi y)(x) = \xi(yx),\]
and for finite Borel measures $\mu$ and $\nu$ on $E$, the Fourier transform of their multiplicative convolution is given by
\begin{equation}
\label{eq:hatmuffnu}
\widehat{\mu \ff \nu} (\xi) = \int \hat\nu(\xi y) \dd\mu(y).
\end{equation}
The idea of the proof of Theorem~\ref{thm:fourier} is to iterate Proposition~\ref{pr:apla} to get a measure with small $L^2$-norm, and then to get the desired Fourier decay by convolving one more time.
Two technical issues arise. First, after each iteration, the measure we obtain does not necessarily satisfy the non-concentration property required by Proposition~\ref{pr:apla}. To settle this, at each step, we truncate the measure to restrict the support on well-invertible elements.
Second, the measure we obtain in the end of the iteration is not an additive convolution of a multiplicative convolution of $\mu$ but some measure obtained from $\mu$ through successive multiplicative and additive convolutions. To conclude we need to clarify relation between the Fourier transforms of these measures. This is settled in Lemma~\ref{lm:ordre}.

\begin{lemma}
\label{lm:iterNC}
Let $E$ be a finite-dimensional normed algebra over $\R$ % Given $\kappa > 0$, there exists $s = s(E,\kappa) \in \N$ and $\eps = \eps(E,\kappa) > 0$ such that for any parameter $\tau \in {(0, \eps \kappa)}$ the following holds for any scale $\delta > 0$ sufficiently small.
and $\mu$ a Borel probability measure on $E$ such that for some $\tau,\eps>0$,
\begin{enumerate}
\item $\mu \bigl(E \setminus B(0,\delta^{-\eps})\bigr) \leq \delta^{\tau}$;
\item for every $x \in E$, $\mu(x + S_E(\delta^{\eps})) \leq \delta^{\tau}$;
\item for every $\rho \geq \delta$ and every proper affine subspace $W \subset E$, $\mu(W^{(\rho)}) \leq \delta^{-\eps} \rho^\kappa$.
\end{enumerate}

Set $\mu_1 = \mu_{\mid B(0,\delta^{-\eps})}$ and define recursively for integer $k \geq 1$,
\[
\eta_k = \mu_{k\mid E \setminus S_E(\delta^{2^k\eps})}
\]
and
\[
\mu_{k+1} = \eta_{k} \ff \eta_{k} \mm \eta_{k} \ff \eta_{k}.
\] 
Then we have for $k \geq 1$,
\begin{equation}
\label{eq:iterNC0}
\mu_k(E) \geq 1 - O_k(\delta^\tau) 
\end{equation}
\begin{equation}
\label{eq:iterNC1}
\Supp(\mu_k) \subset B(0, \delta^{-O_k(\eps)})
\end{equation}
\begin{equation}
\label{eq:iterNC2}
\forall x \in E,\; \mu_k(x + S_E(\delta^{2^k\eps})) \leq \delta^\tau
\end{equation}
\begin{equation}
\label{eq:iterNC3}
\forall \rho \geq \delta,\, \forall W \subset E \text{ proper affine subspace},\; \mu_k(W^{(\rho)}) \leq \delta^{-O_k(\eps)} \rho^\kappa.
\end{equation}
As a consequence, the same holds for $\eta_k$ in the place of $\mu_k$.
\end{lemma}
\begin{proof}
The proof goes by induction on $k$.\\
The result is clear for $k=1$, by assumption on $\mu$.
Assume \eqref{eq:iterNC0}--\eqref{eq:iterNC3} true for some $k \geq 1$, so that the same holds for $\eta_k$.
Then  \eqref{eq:iterNC0} and \eqref{eq:iterNC1} for $k+1$ follow immediately.

Let us prove \eqref{eq:iterNC2} for $k+1$. Let $x \in E$. Since $\mu_{k+1} = \eta_{k} \ff \eta_{k} \mm \eta_{k} \ff \eta_{k}$,
\[
\mu_{k+1}\bigl(x + S_E(\delta^{2^{k+1}\eps})\bigr) = \iint \eta_k\{ y \in E \mid  \abs{\det\nolimits_E( y z - w - x)} \leq  \delta^{2^{k+1}\eps} \} \dd\eta_k(z) \dd (\eta_k \ff \eta_k)(w)
\]
Note that for $z \in \Supp(\eta_k)$, by definition, $\abs{\det_E(z)} \geq \delta^{2^k\eps}$. Hence $\abs{\det\nolimits_E( y z - w - x)} \leq  \delta^{2^{k+1}\eps}$ implies $y - (w+x) z^{-1} \in S_E(\delta^{2^k\eps})$. Therefore, by induction hypothesis \eqref{eq:iterNC2}
\[
\mu_{k+1}\bigl(x + S_E(\delta^{2^{k+1}\eps})\bigr) \leq \max_{z \in \Supp(\eta_k),\, w\in E} \eta_k\bigl((w+x) z^{-1} + S_E(\delta^{2^k\eps})\bigr) \leq \delta^\tau.\]

Finally, let us prove \eqref{eq:iterNC3} for $k+1$. Let $\rho \geq \delta$ and let $W$ be a proper affine subspace of $E$. We have as above
\[
\mu_{k+1}\bigl( W^{(\rho)} \bigr) \leq  \max_{z \in \Supp(\eta_k),\, w\in E} \eta_k\bigl((w + W^{(\rho)}) z^{-1}\bigr).
\]
For all $z \in \Supp(\eta_k)$, we have $\abs{\det_E(z)} \geq \delta^{O_k(\eps)}$ and by the induction hypothesis $\norm{z} \leq \delta^{-O_k(\eps)}$. Hence $\norm{z^{-1}} \leq \delta^{-O_k(\eps)}$. Thus 
\[(w + W^{(\rho)}) z^{-1} = wz^{-1} + Wz^{-1} + B(0,\rho) z^{-1} \subset wz^{-1} + Wz^{-1} + B(0,\delta^{-O_k(\eps)}\rho),\]
which is nothing but the $(\delta^{-O_k(\eps)}\rho)$-neighborhood of another proper affine subspace.
Hence by induction hypothesis~\eqref{eq:iterNC3},
\[
\mu_{k+1}\bigl( W^{(\rho)} \bigr) \leq \delta^{-O_k(\eps)}\rho^\kappa.
\]
This finishes the proof of the induction step and that of the lemma.
\end{proof}

%Observe that for $\rho > 0$,
%\[B(0,\rho^{-1}) \setminus S_E(\rho) \subset G_E(\rho^{-O_E(1)}).\]

This is a slightly more general form of \cite[Theorem 7]{Bourgain2010}. The proof is essentially the same.

\begin{lemma}
\label{lm:L2Fourier}
Let $E$ be normed algebra over $\R$ of dimension $d$. The following holds for any parameters $\kappa > 0$ and $\eps > 0$ and any scale $\delta > 0$ small enough. Let $\mu$ and $\nu$ be Borel probability measures on $E$. Assume 
\begin{enumerate}
\item $\Supp(\nu) \subset B(0,\delta^{-\eps})$,
\item $\norm{\mu \pp P_\delta}_2^2 \leq \delta^{-\kappa}$,
\item \label{it:L2Fourier3} for every proper affine subspace $W \subset E$, $\nu(W^{(\delta)}) \leq \delta^{2\kappa}$.
\end{enumerate}
Then for $\xi \in E^*$ with $\delta^{-1+\eps} \leq \norm{\xi} \leq \delta^{-1-\eps}$,
\[
\abs{\widehat{\mu \ff \nu}(\xi)} \leq \delta^{\frac{\kappa}{d + 3} - O(\eps)}
\]
\end{lemma}

\begin{proof}
Changing $\delta$ to $\delta^{1 + O(\eps)}$ only weakens the assumptions by little. Namely, $\eps$ becomes $O(\eps)$ and $\kappa$ becomes $\kappa - O(\eps)$.
Therefore, we may assume without loss of generality that $\norm{\xi} = \delta^{-1 + 2\eps}$.

First, by continuity, there is a constant $c > 0$ depending only on $E$ such that 
\[\forall \zeta \in B_{E^*}(0,c),\quad \abs{\widehat{P_1}(\zeta)} \geq \frac{1}{2}.\]
Hence 
\[\forall \zeta \in B_{E^*}(0,c\delta^{-1}), \quad \abs{\widehat{P_\delta}(\zeta)} = \abs{\widehat{P_1}(\delta \zeta)}\geq \frac{1}{2}.\]
Then, by the Plancherel formula
\begin{equation}
\label{eq:PlancheEmu} 
\int_{B_{E^*}(0,c\delta^{-1})}\abs{\hat{\mu}(\zeta)}^2 \dd \eta \ll \int_{E^*} \abs{\hat{\mu}(\zeta)}^2\abs{\widehat{P_\delta}(\zeta)}^2 \dd \eta = \norm{\mu \pp P_\delta}_2^2 \leq \delta^{-\kappa}.
\end{equation}

For $R \geq 1$ and $t  \in{(0,1)}$, define
\[H_{R,t} = \{\, \zeta \in B_{E^*}(0,R) \mid \abs{\hat{\mu}(\zeta)} \geq t \,\}.\]
There is a constant $C \geq 1$ depending only on $E$ such that $\hat{\mu}$ is $C$-Lipschitz. Hence
\[H_{R,t} + B_{E^*}\bigl(0,\frac{t}{2C}\bigr) \subset H_{R+1,\frac{t}{2}}.\]
Choosing $R = \delta^{-1 + \eps}$ and using \eqref{eq:PlancheEmu} we can bound the Lebesgue measure :
\[\bigl\lvert{H_{R,t} + B_{E^*}\bigl(0,\frac{t}{2C}\bigr)}\bigr\rvert \ll t^{-2} \delta^{-\kappa}.\]
It follows that
\begin{equation}
\label{eq:NcovHRt}
\Ncal(H_{R,t},t) \ll t^{-d-2} \delta^{-\kappa}.
\end{equation}

Now, let $E$ act on $E^*$ by 
\[\forall x,y \in E,\, \zeta \in E^*,\quad (y\cdot \zeta)(x) = \zeta(xy).\]
Then
\[\widehat{\mu \ff \nu}(\xi) = \int_{E} \hat\mu(y \cdot \xi)\dd \nu(y).\]
Note that for any $y \in \Supp(\nu)$, $\norm{y\cdot \xi} \leq \delta^{-\eps} \norm{\xi} = R$.
Thus, we can bound 
\begin{equation}
\label{eq:HxiRt}
\abs{\widehat{\mu \ff \nu}(\xi)} \leq t + \nu (H^\xi_{R,t})
\end{equation}
where
\[H^\xi_{R,t} = \{\, y \in E \mid y\cdot \xi \in H_{t,R}\,\}.\]
Covering $H_{R,t}$ with balls of radius $t$, we get a covering of $H^\xi_{R,t}$ by sets of the form
\[ \{\, y \in E \mid y\cdot \xi \in B_{E^*}(\zeta,t)\,\},\quad \zeta \in E^*.\]
Each of these sets is contained in
\[ \{\, y \in E \mid \abs{ \xi(y) - \zeta(1)} \leq t \,\},\]
which is a $\norm{\xi}^{-1}t$-neighborhood of an affine hyperplane. 
Assumption \ref{it:L2Fourier3} and equation \eqref{eq:NcovHRt} then lead to
\[\nu(H^\xi_{R,t}) \leq \Ncal(H_{R,t},t) t^{2\kappa} \norm{\xi}^{-2\kappa} \leq \delta^{-O(\eps) + \kappa} t^{-d-2}. \]
With the choice $t = \delta^{\frac{\kappa}{d+3}}$, \eqref{eq:HxiRt} proves the lemma.
\end{proof}

For a finite Borel measure $\nu$ on $E$ and an integer $l \geq 1$, we denote by 
\[\nu^{\pp l} = \underbrace{\nu \pp \dotsb \pp \nu}_{\text{$l$ times}}\]
the $l$-fold additive convolution of $\nu$.

\begin{lemma}
\label{lm:ordre}
Let $E$ be a finite-dimensional real algebra.
Let $l \geq 1$ be an integer, $\nu$ a Borel probability measure on $E$, and set 
\[\mu = \nu^{\pp l} \mm \nu^{\pp l}.\] 
%\[\mu = \bigl(\underbrace{\nu \pp \dotsb \pp \nu}_{\text{$l$ times}}\bigr) \mm \bigl(\underbrace{\nu \pp \dotsb \pp \nu}_{\text{$l$ times}}\bigr).\] 
Then for every integer $m \geq 1$, for every $\xi \in E^*$, the Fourier coefficient $\widehat{\mu^{* m}}(\xi)$ is real and 
\begin{equation}
\label{eq:ordre}
\abs{\widehat{\nu^{* m}}(\xi)}^{(2l)^m} \leq \widehat{\mu^{* m}}(\xi).
\end{equation}
The same inequality also holds for a finite Borel measure with total mass $\nu(E) \leq 1$.
\end{lemma}

\begin{proof}
%Observe that $\mu$ is symmetric with respect to $0$. Hence so is $\mu^{* k}$. Hence $\widehat{\mu^{* k}}(\xi)$ is real and nonnegative. 

We proceed by induction on $m$.
For $m=1$,
\[\hat{\mu}(\xi) = \abs{\hat\nu(\xi)}^{2l}.\]
Assume then that \eqref{eq:ordre} is true for some $m\geq 1$. 
By \eqref{eq:hatmuffnu}, the Hölder inequality and the induction hypothesis,
\begin{align*}
\abs{\widehat{\nu^{* (m+1)}}(\xi)}^{(2l)^m} &= \Bigl\lvert\int \widehat{\nu^{* m}}(\xi y)\dd \nu(y) \Bigr\rvert^{(2l)^m}\\
&\leq \int \abs{\widehat{\nu^{* m}}(\xi y)}^{(2l)^m} \dd \nu(y)\\
&\leq \int \widehat{\mu^{* m}}(\xi y) \dd \nu(y) \\
&= \iint e(\xi yx) \dd \mu^{* m}(x) \dd \nu(y) 
\end{align*}

Taking the $2l$-th power and using again the Hölder inequality and \eqref{eq:hatmuffnu}, we obtain
\begin{align*}
&\abs{\widehat{\nu^{* (m+1)}}(\xi)}^{(2l)^{m+1}}\\
\leq& \Bigl\lvert \iint e(\xi yx) \dd \nu(y) \dd \mu^{* m}(x) \Bigr\rvert^{2l}\\
\leq&  \int \Bigl\lvert \int e(\xi yx)) \dd \nu(y) \Bigl\lvert^{2l} \dd \mu^{* m}(x) \\
=&  \iiint e\bigl(\xi (y_1 + \dotsb + y_l - y_{l+1} - \dotsb - y_{2l})x \bigr) \dd \nu(y_1) \dotsm \dd \nu(y_{2l}) \dd \mu^{* m}(x) \\
=& \iint e(\xi z x) \dd \mu(z)  \dd \mu^{* m}(x)\\
=& \widehat{\mu^{* (m+1)}}(\xi) 
\end{align*}
This proves the induction step and finishes the proof of \eqref{eq:ordre}.
If $\nu$ is a finite Borel measure with $\nu(E)\leq 1$, we may apply \eqref{eq:ordre} to the probability measure $\nu(E)^{-1} \nu$, which yields
\[\Bigl( \frac{\abs{\widehat{\nu^{* m}}(\xi)}}{\nu(E)^m} \Bigr)^{(2l)^m} \leq \frac{\widehat{\mu^{* m}}(\xi)}{\nu(E)^{2lm}}.\]
\end{proof}

\begin{proof}[Proof of Theorem~\ref{thm:fourier}]
Let $d = \dim(E)$. Let $\eps_1 > 0$ be the constant given by Proposition~\ref{pr:apla} applied to the parameter $\kappa/2$. Define $s_{\max} = \lceil \frac{d}{\eps_1}\rceil$. 
First, remark that by the non-concentration assumption, $\norm{\mu \pp P_\delta}_\infty \leq \delta^{-d + \kappa - \eps}$. Hence
\[ \norm{\mu \pp P_\delta}_2^2 \leq \norm{\mu \pp P_\delta}_1 \norm{\mu \pp P_\delta}_\infty \leq \delta^{-d + \kappa/2}\]
if we choose $\eps \leq \kappa/2$.
For $k = 1, \dotsc, s_{\max}$, let $\mu_k$ and $\eta_k$ be defined as in Lemma~\ref{lm:iterNC}.
Since $k \leq s_{\max}$ is bounded, the implied constants in the $O_k(\eps)$ notations in Lemma~\ref{lm:iterNC} can be chosen uniformly over $k$. 
Thus, when $\eps > 0$ is sufficiently small, Lemma~\ref{lm:iterNC} allows us to apply Proposition~\ref{pr:apla} and Remark~\ref{rk:mu1half} after it to the measure $\eta_k$ for each $k = 1, \dotsc, s_{\max}$. 
Thus, either $\norm{\eta_k \pp P_\delta}_2^2 \leq \delta^{-\kappa/2}$ or
\[
\norm{\eta_{k+1} \pp P_\delta}_2^2 \leq \norm{\mu_{k+1} \pp P_\delta}_2^2 \leq \delta^{\eps_1} \norm{\eta_k \pp P_\delta}_2^2.
\]
We deduce that there exists $s \in \{1, \dotsc , s_{\max}\}$ such that
\[
\norm{\eta_{s} \pp P_\delta}_2^2 \leq \delta^{-\kappa/2}.
\]
Remembering \eqref{eq:iterNC3}, we apply Lemma~\ref{lm:L2Fourier} to obtain, for all $\xi$ with $\norm{\xi} = \delta^{-1}$,
%$\delta^{-1+\eps} \leq \norm{\xi} \leq \delta^{-1-\eps}$, 
\[ 
\abs{\widehat{\eta_s * \eta_s} (\xi)} \leq \delta^{\kappa/(2d+6)-O(\eps)} \leq \delta^\tau. 
\]
Here we assumed $\eps$ sufficiently small compared to $\kappa/d$.

For $k = 1, \dotsc, s$, apply Lemma~\ref{lm:ordre} with $l = 1$ and $m = 2^k$ to $\mu_{s-k+1} = \eta_{s-k}^{* 2} \mm \eta_{s-k}^{* 2}$. We obtain
\[
\bigl\lvert \widehat{\eta_{s-k}^{* 2^{k+1}}} (\xi) \bigr\rvert^{2^{2^k}} \leq \bigl\lvert \widehat{\mu_{s-k+1}^{* 2^k}} (\xi) \bigr\rvert.
\]
Moreover, by \eqref{eq:iterNC2}, $\mu_{s-k+1}$ differs from $\eta_{s-k+1}$ by a measure of total mass at most $\delta^\tau$. Hence 
\[
\bigl\lvert \widehat{\mu_{s-k+1}^{* 2^k}} (\xi) \bigr\rvert \leq \bigl\lvert \widehat{\eta_{s-k+1}^{* 2^k}} (\xi) \bigr\rvert + O_k(\delta^\tau).
\]
From the above, we deduce using a simple recurrence that for all $k = 1, \dotsc, s$,
\[
\bigl\lvert \widehat{\mu_{s-k+1}^{* 2^k}} (\xi) \bigr\rvert \ll_k \delta^{{\tau}/{O_k(1)}}.
\]
In particular,
\[
\bigl\lvert \widehat{\mu_1^{* 2^s}} (\xi) \bigr\rvert \ll_s \delta^{\eps\tau},
\]
which allows to conclude since $\mu$ differs from $\mu_1$ by a measure of total mass at most $\delta^\tau$.
\end{proof}

\section{Non-concentration in subvarieties}
\label{sc:noncon}

Our goal here is to prove that the law of a large random matrix product satisfies some regularity conditions. 

\smallskip

Throughout this section, unless otherwise stated, $\mu$ denotes a probability measure on $\SL_d(\Z)$. %having a finite exponential moment.
As in the introduction, $\Gamma$ denotes the subsemigroup generated by $\Supp(\mu)$, $\Gbb < \SL_d$ is the Zariski closure of $\Gamma$, and $G = \Gbb(\R)$ its group of $\R$-rational points.
We also let $E$ denote the subalgebra of $\Mat_{d}(\R)$ generated by $G$, and fix a norm on the space of all polynomial functions on $E$.
%Recall the assumptions from Theorem~\ref{thm:main}.
%\begin{enumerate}[label=(\alph*)]
%\item The measure $\mu$ has a finite exponential moment. 
%\item The action $G$ on $\R^d$ is irreducible.
%\item The algebraic group $\Gbb$ is Zariski connected.% and semisimple.
%\end{enumerate}
We shall prove two non-concentration statements for the distribution of a random matrix product.
The first one shows that the law $\mu^{*n}$ at time $n$ of the random matrix product is not concentrated near affine subspaces of the algebra $E$.

\begin{proposition}[Non-concentration on affine subspaces]
\label{pr:NCaffine}
Let $\mu$ be a probability measure on $\SL_d(\Z)$, for some $d\geq 2$.
Denote by $\Gamma$ the subsemigroup generated by $\mu$, and by $\Gbb$ the Zariski closure of $\Gamma$ in $\SL_d$, and by $E$ the subalgebra of $\Mat_d(\R)$ generated by $\Gamma$.
Assume that: 
\begin{enumerate}[label=(\alph*)]
\item  The measure $\mu$ has a finite exponential moment; 
\item  The action of $\Gamma$ on $\R^d$ is irreducible;
\item  The algebraic group $\Gbb$ is Zariski connected. % and semisimple.
\end{enumerate} 
There exists $\kappa > 0$ such that for every proper affine subspace $W \subset E$ ,
\[ \forall n \geq 1,\; \forall \rho \geq e^{-n},\quad \mu^{*n}(\{\, g \in \Gamma  \mid  d(g, W ) \leq \rho \norm{g} \,\}) \ll_\mu \rho^\kappa.\]
\end{proposition}

The second result concerns general subvarieties of the algebra $E$, with the caveat that we have to replace $\mu^{*n}$ by an additive convolution power of itself, to avoid some obstructions.
It is also worth noting that the quantification of the non-concentration is slightly weaker than in the case of affine subspaces.
%This cannot be avoided.

\begin{proposition}[Non-concentration on subvarieties]
\label{pr:NCdet}
%\marginpar{\small On a besoin de cette proposition seulement pour $k = \dim(E)$.}
Let $\mu$ be a probability measure on $\SL_d(\Z)$, for some $d\geq 2$.
Let $\Gamma$ denote the subsemigroup generated by $\mu$, $\Gbb$ the Zariski closure of $\Gamma$ in $\SL_d$, $E$ the subalgebra of $\Mat_d(\R)$ generated by $\Gamma$, and $\lambda_1$ the top Lyapunov exponent of $\mu$.
Assume that: 
\begin{enumerate}[label=(\alph*)]
\item  The measure $\mu$ has a finite exponential moment; 
\item  The action of $\Gamma$ on $\R^d$ is irreducible;
\item  The algebraic group $\Gbb$ is Zariski connected. % and semisimple.
\end{enumerate} 
Given an integer $D \geq 1$ and given $\omega > 0$, there exists $c > 0$ and $n_0 \in\N$
%depending also on $\mu$
such that the following holds. 
Let $f \colon E \to \R$ be a polynomial function of degree $D$. Writing $f_D$ for its degree $D$ homogeneous part, we have
\[
\forall k \geq \dim(E),\; \forall n \geq n_0,\quad (\mu^{\ff n})^{\pp k} \bigl(\bigl\{ x \in E \mid \abs{f(x)} \leq e^{(D \lambda_1 - \omega) n} \norm{f_D} \bigr\}\bigr) \leq e^{-cn}.
\]
\end{proposition}

These two propositions will allow us to apply the results of the previous section to the law $\mu^{*n}$ at time $n$ of an irreducible random walk on $\SL_d(\Z)$.
This will yield Theorem~\ref{thm:decay} below.

\smallskip

Non-concentration estimates for subvarieties can sometimes be obtained using the theory of random matrix products, combined with linearization techniques, as is done in \cite{aoun}.
But this approach relies on the proximality assumption which we want to avoid here.
Moreover, the statements proven in \cite{aoun} are not uniform for subvarieties of bounded degree, which is crucial for our application.

The argument developed in this section relies on the spectral gap property modulo primes for finitely generated subgroups of $\H(\Z)$, where $\H$ is a semisimple $\Q$-subgroup of $\SL_n$, a still rather recent result obtained by Salehi Golsefidy and Varjú \cite{SGV} after several important works in this direction, starting with Helfgott \cite{helfgott_sl2}, followed by Bourgain and Gamburd \cite{bourgaingamburd_sl2p}, Breuillard, Green and Tao \cite{bgt_expansionlinear}, and Pyber-Szabó \cite{pyberszabo}.

It is not difficult to check, see e.g. \cite[Lemma~8.5]{bq1}, that if $\Gamma < \SL_d(\Z)$ acts strongly irreducibly on $\R^d$, then its Zariski closure is semisimple.
In particular, the assumptions of Propositions~\ref{pr:NCaffine} and \ref{pr:NCdet} imply that $\Gbb$ is semisimple.
This justifies why we focus on semisimple groups in the discussion of the spectral gap property that we give below.

\subsection{Prelude : Expansion in semisimple groups}
%We now explain the spectral gap result of Salehi Golsefidy and Varjú, and its implications for our problem.
Since elements in $\Gamma$ have integer coefficients, we may consider, for a given prime number $p$, the image $\Gamma_p$ of $\Gamma$ under the projection modulo $p$.

On the space $l^2(\Gamma_p)$ of square-integrable functions on $\Gamma_p$, we shall consider the convolution operator
\[ \begin{array}{lrcl}
T_{\mu} \colon & l^2(\Gamma_p) & \to & l^2(\Gamma_p))\\
& f & \mapsto & \mu * f
\end{array}
\]
Let $l_0^2(\Gamma_p) \subset l^2(\Gamma_p)$ denote the subspace of functions on $\Gamma_p$ having zero mean.

The theorem we shall need is the following; up to some minor modifications, it appears in Salehi Golsefidy and Varjú \cite[Theorem~1]{SGV}.

\begin{theorem}[Spectral gap theorem]
\label{thm:specgap}
Let $d\geq 2$ and let $\mu$ be a probability measure on $\SL_d(\Z)$ such that the Zariski closed subgroup $\Gbb$ generated by $\mu$ is semisimple.
Then there exists a constant $c>0$ and an integer $k$ such that for every prime number $p$,
\[
\norm{T^k_{\mu}}_{l_0^2(\Gamma_p)} \leq 1 - c.
\]
\end{theorem}

\begin{remark}
Another way to state the above theorem is to say that the spectral radius of the operator $T_\mu$ restricted to $l^2_0(\Gamma_p)$ is bounded above by $1-c$, for every prime number $p$.
\end{remark}
%
%
%The following statement says that the same thing holds for measures $\mu$ that are not necessarily uniform measures on a finite set, if we are willing to take a small power of the Markov operator.
%\begin{corollary}
%\label{cr:specgap}
%Under the assumptions \ref{it:SA1} and \ref{it:SA3}.
%There is an integer $k \geq 1$ and $c > 0$ such that for every prime number $p$ sufficiently large,
%\end{corollary}

For completeness, we now explain how to derive the above theorem from \cite[Theorem~1]{SGV}.
The argument uses the following two lemmata.

\begin{lemma}
\label{lm:Sfinite}
Let $\Gamma$ be a subsemigroup of $\SL_d(\R)$ whose Zariski closure $\Gbb$ is semisimple. 
There exists a finite subset $S \subset \Gamma$ such that the semigroup generated by $S$ is Zariski dense in $\Gbb$.
\end{lemma}
%This lemma and its proof remain valid when $\R$ is replaced by any field of characteristic zero.
\begin{proof}
For any finite subset $S \subset \Gamma$, denote by $Z_e(S)$ the identity component of the Zariski closure of the semigroup generated by $S$.
Let $S_0 \subset \Gamma$ be a finite subset such that $\Hbb := Z_e(S_0)$ has maximal dimension among these subgroups. Then $\Hbb$ is also maximal for the order of inclusion, because $Z_e(S_0) \subset Z_e(S_0 \cup S)$ are both irreducible subvarieties and have the same dimension.

In particular, for any $\gamma \in \Gamma$, $\gamma\Hbb\gamma^{-1} = Z_e(\gamma S_0 \gamma^{-1}) \subset \Hbb$. Hence $\Hbb$ is a normal subgroup in $\Gbb$, since $\Gamma$ is Zariski dense.
By \cite[Theorem 6.8 and Corollary 14.11]{Borel}, the quotient $\Gbb/\Hbb$ is
%(can be given a structure of)
a semisimple linear algebraic group such that the projection $\pi \colon \Gbb \to \Gbb/\Hbb$ is a morphism of algebraic groups. %defined over $\R$.

Moreover, for any $\gamma \in \Gamma,$ there is $k \geq 1$ such that $\gamma^k \in Z_e(\{\gamma\}) \subset \Hbb$. Hence the image $\pi(\Gamma)$ of $\Gamma$ in $\Gbb/\Hbb$ is a torsion group.
By the Jordan-Schur theorem~\cite[Theorem~8.31]{raghunathan}, $\pi(\Gamma)$ is virtually abelian.
%\marginpar{\small Find the reference.}
But $\pi(\Gamma)$ is Zariski dense in $\Gbb/\Hbb$. Hence the Zariski closure of any of its subgroups of finite index contains the identity component of $\Gbb/\Hbb$.
Therefore, the identity component of $\Gbb/\Hbb$ is both semisimple and abelian, hence trivial.
It follows that $\Gbb/\Hbb$ is finite. Thus, by adding a finite number of elements to $S_0$, we can make sure that $S_0$ generates a Zariski dense subsemigroup in $\Gbb$.
\end{proof}

\begin{lemma}
\label{lm:S-kSk}
Let $\Gbb$ be a connected semisimple algebraic group defined over $\Q$. Let $S \subset \Gbb(\Q)$ be a finite subset which generates a Zariski dense subgroup.
%Let $S$ be a finite subset of $\SL_d(\Z)$. Assume that the Zariski closure $\Gbb$ of the subsemigroup generated by $S$ is semisimple. 
Then there exists $k \geq 1$ such that the symmetric set $S^{-k}S^k$ also generates a Zariski dense subgroup in $\Gbb$.  
\end{lemma}
\begin{proof}
By a result of Nori \cite[Theorem~5.2]{nori}, the group $\Gamma$ generated by $S$ is dense in some open subgroup $\Omega$ of $\Gbb(\Q_p)$, for some prime number $p$.
In other words, the union $\bigcup_{k\geq 1}S^k$ is dense in $\Omega$.
Since the Lie algebra $\g$ of $\Omega$ satisfies $[\g,\g]=\g$, the derived subgroup $[\Omega,\Omega]$ contains an open subgroup $\Omega'$; then, the increasing sequence of subsets $(S^kS^{-k})_{k\geq 1}$ gets arbitrarily dense in $\Omega'$.
%Il y a une petite tricherie ici. On pourrait avoir besoin de plusieurs commutateurs pour engendrer $\Omega'$, et c'est donc plutôt $(S^kS^{-k})^l$, pour certains $k$ et $l$, qui se densifie.
%Il faudrait donc modifier l'énoncé en changent $S^kS^{-k}$ en $(S^kS^{-k})^l$.
%À moins que l'application $(x,y)\mapsto [x,y]$ soit ouverte en $0$ dans $\g$, ce qui semble vrai, car $\g$ est semi-simple. Voir D'andrea-Maffei pour cet énoncé dans un group de Lie réel semisimple compact.
%De toute façon, cette modification est sans conséquence sur la suite.

On the other hand, there exists a neighborhood $U$ of the identity in $\Gbb(\Q_p)$ and $\delta>0$, such that if $\Hbb$ is an algebraic subgroup of $\Gbb$, then $\Hbb(\Q_p)$ is not $\delta$-dense in $U$.
Indeed, the Lie algebra $\h$ of $\Hbb(\Q_p)$ satisfies $\dim_{\Q_p}\h < \dim_{\Q_p}\g$, and we can distinguish two cases.
If the normalizer $N(\h)$ of $\h$ in $\Gbb(\Q_p)$ does not contain an open subgroup, then we conclude using \cite[Lemma~2.2]{saxce_producttheorem}, which is still valid over $\Q_p$.
Otherwise, $\h$ is an ideal in $\g$, so that $\Hbb$ is a sum of simple factors of $\Gbb$; there are only finitely many such groups.
\end{proof}

\begin{proof}[Proof of Theorem~\ref{thm:specgap}]
Let $\check\mu$ denote the image measure of $\mu$ by the map $g \mapsto g^{-1}$.
By lemmata \ref{lm:Sfinite} and \ref{lm:S-kSk} above, there is $k \geq 1$ such that the support of $\check\mu^{*k} * \mu^{*k}$ contains a finite symmetric subset $S$ which generates a Zariski dense subgroup in $\Gbb$. 
By \cite[Theorem~1]{SGV} applied to $S$, there is $c > 0$ such that for any prime number $p$,
\[\norm{T_{\mu_S}}_{l_0^2(\Gamma_p)} \leq 1 - c,\]
where $\mu_S$ denotes the normalized counting measure on $S$.
Then, we can write 
\[\check\mu^{*k} * \mu^{*k} = \alpha \mu_S + (1- \alpha) \mu'\]
where $\alpha > 0$ and $\mu'$ is some probability measure on $\SL_d(\Z)$.
Thus, for any prime number $p$ sufficiently large, using the fact that $T_\mu^* = T_{\check\mu}$,
\begin{align*}
\norm{T_\mu^k}_{l_0^2(\Gamma_p)}^2 & \leq  \norm{(T_\mu^*)^k T_\mu^k}_{l_0^2(\Gamma_p)} \\
&= \norm{T_{\check\mu^{*k} * \mu^{*k}}}_{l_0^2(\Gamma_p)}\\
& \leq \alpha \norm{T_{\mu_S}}_{l_0^2(\Gamma_p)} + (1- \alpha) \norm{T_{\mu'}}_{l_0^2(\Gamma_p)} \\
&\leq 1 - \alpha c.
\end{align*}
\end{proof}

One can also interpret Theorem~\ref{thm:specgap} as a statement on the speed of equidistribution of the random walks associated to $\mu$ on the Cayley graphs $\Gamma_p$.
This is explained for instance in \cite[\S 3.1]{HLW} for the case of simple random walks on a family of expander graphs.  
%It is well known that a spectral gap implies a rapid mixing property for random walks.
Corollary~\ref{cr:gapEquid} below is a statement of this kind.

Since elements in $\Gamma$ have integer coefficients, the group $\Gbb$ is defined over $\Q$, so we may choose a set of defining polynomials with coefficients in $\Z$.
Given a prime number $p$, this allows to consider the variety $\Gbb_p$ defined over $\Fbb_p$ by the reduction modulo $p$ of the polynomials defining $\Gbb$.
Naturally, the group $\Gamma_p$ can be viewed as a subgroup of $\Gbb_p(\Fbb_p)$; by Nori's approximation theorem \cite[Theorem~5.1 and Remark 3.6]{nori}, we know that the index $[\Gbb_p(\Fbb_p):\Gamma_p]$ is uniformly bounded.
For a prime number $p$, we denote by $\pi_p \colon \Mat_d(\Z) \to \Mat_d(\Fbb_p)$ the reduction modulo $p$. 
The detailed proof of the result below is left to the reader.

\begin{corollary}
\label{cr:gapEquid}
Let $d\geq 2$ and let $\mu$ be a probability measure on $\SL_d(\Z)$ such that the Zariski closed subgroup $\Gbb$ generated by $\mu$ is semisimple.
There exists $C \geq 0$ such that for every prime number $p$ sufficiently large, and $n \geq C \log p$ , for all $a \in \Mat_d(\Fbb_p)$,
\[
\mu^{*n}(\{g \in \Mat_d(\Z) \mid \pi_p(g) = a\}) \ll_d \frac{1}{\abs{\Gbb_p(\Fbb_p)}}.
\]
\end{corollary}

\subsection{Escaping from subvarieties: a consequence of the spectral gap}

The aim of this subsection is to establish the following proposition, using the results of the previous paragraph.
\begin{proposition}
\label{pr:gapEscape}
Let $\mu$ be a probability measure on $\SL_d(\Z)$ such that the Zariski closed subgroup $\Gbb$ generated by $\Supp\mu$ is semisimple and connected.

There exists a constant $c > 0$ depending on $\mu$ such that 
for all polynomials $f \in \Q[\Mat_d]$, of degree at most $D$ and not vanishing on $\Gbb$, we have 
\[\forall n \geq 1,\quad \mu^{*n}(\{g \in \Gamma \mid f(g)=0\}) \ll_{\mu,D} e^{-cn}.\]
\end{proposition}

Indeed, this is a general version of \cite[Corollary 1.1]{BourgainGamburd_modpnII} which is stated for the group $\SL_d$ and for the simple random walk on the Cayley graph.
The proof is essentially the same.
But since it will be important for us to know that the upper bound depends only on the degree of $f$, we provide a detailed argument.
%Here we include the details nevertheless since we are dealing with random walks which are not simple random walk on Cayley graphs. \marginpar{\small ou bien "for the reader's convenience"? } 

%In this article, we adhere to the convention that a variety is identified with their point over a fixed algebraically closed field (e.g. $\C$ if the variety is defined over $\Q$.). Thus, by irreducible we always mean geometrically irreducible. 
% all relative notions (dimension, irreducibility, etc.) are with respect to the structure over this algebraically closed field.
%However, in the next proof we will briefly use the Zariski topology of a spectrum of an algebra over an algebraically non-closed filed, $\Q$.

We need the following lemma, which is a consequence of the Lang-Weil inequality.
It will be important for us to have an estimate which is uniform for subvarieties of bounded complexity.
\begin{lemma}
\label{lm:LWcplxity}
Given a geometrically irreducible subvariety $V \subset \Abb^d$ defined over $\Q$ and an integer $D \geq 1$, there exists $p_0 = p_0(V,D)$ such that the following holds. Let $f \in \Q[X_1,\dotsc,X_d]$ be a polynomial of degree at most $D$. Assume that $f$ does not vanish on $V$. Then for every prime number $p \geq p_0$,
\[\abs{\{\pi_p(x) \in \Fbb_p^d \mid x \in V \cap \Z^d,\, f(x) = 0\}} \ll_{V,D} p^{-1} \abs{V_p(\Fbb_p)}\]
where $V_p$ denotes the reduction modulo $p$ of $V$.
\end{lemma}

Note that to speak about reductions modulo $p$ of a variety over $\Q$, we need to choose a model over $\Z$. But since $V$ is a subvariety of an affine space, among such choices, there is a canonical one. See the first paragraph in the proof below.

\begin{proof}
We abbreviate $\Q[X_1,\dotsc,X_d]$ as $\Q[X]$ and $\Z[X_1,\dotsc,X_d]$ as $\Z[X]$. 
Let $I \subset \Q[X]$ be the ideal of all polynomials with coefficients in $\Q$ and vanishing on $V$. 
Let $I_\Z = I \cap \Z[X]$. % and $I_{\Fbb_p} = \pi_p(I_\Z) \subset \Fbb_p[X]$  for a prime number $p$.
For any prime number $p$, let $V_p$ be the variety over $\Fbb_p$ defined by the ideal $ \pi_p(I_\Z) \subset \Fbb_p[X]$.
By the Bertini-Noether theorem \cite[Proposition 10.4.2]{FriedMoshe},
%\cite[Chapter VIII, \S 8, Prop. 7]{Lang_DG}\marginpar{Lang traite le cas où $V$ est une hypersurface, nous avons besoin du cas général}, 
 $V_p$ is geometrically irreducible for $p \geq p_0$, where $p_0$ is a constant depending only on the embedding $V \subset \Abb^d$.

On the one hand applying the Lang-Weil inequality \cite[Theorem 1]{LangWeil} to the irreducible variety $V_p$, we obtain  
\begin{equation}
\label{eq:LWlower}
\abs{V_p(\Fbb_p)} \geq \frac{1}{2} p^{\dim(V_p)}.
\end{equation}
On the other hand, one can prove, using Gröbner bases \cite[Chapter 2]{CoxLittleOShea}, that given an integer $D \geq 1$ there is $p_0 = p_0(V,D)$ such that for all $p\geq p_0$ and all $f \in \Q[X] \setminus I$ of degree at most $D$ there is $h \in (\Q f \oplus I) \cap \Z[X]$ of degree at most $O_{V,D}(1)$ and such that $\pi_p(h) \notin \pi_p(I_\Z)$.
For such $h$, we have
\[\{x \in V \cap \Z^d \mid  f(x) = 0\} \subset \{x \in V \cap \Z^d \mid  h(x) = 0\}\]
and hence
\[\abs{\{\pi_p(x) \in \Fbb_p^d \mid x \in V \cap \Z^d,\, f(x) = 0\}} \leq \abs{ \{ x \in  V_p(\Fbb_p) \mid \pi_p(h) (x) = 0 \} }.\]
The right-hand side is the number of $\Fbb_p$-points in the subvariety $V_p \cap \{\pi_p(h) = 0\}$.
This subvariety has dimension at most $\dim(V_p) - 1$ since $V_p$ is  irreducible and $\pi_p(h) \notin \pi_p(I_\Z)$. Thus, applying a version of the Schwarz-Zippel estimate, like \cite[Lemma 1]{LangWeil}, and using the fact that the complexity controls the degree, we get
\begin{equation}
%\label{eq:LWupper}
\abs{(V_p \cap \{\pi_p(h) = 0\})(\Fbb_p)} \ll_{V,D} p^{\dim(V_p) - 1}.
\end{equation}
Together with \eqref{eq:LWlower}, this proves the desired inequality.
\end{proof}

Now we are ready to prove Proposition~\ref{pr:gapEscape}.
\begin{proof}[Proof of Proposition~\ref{pr:gapEscape}]
By Lemma~\ref{lm:LWcplxity},
\begin{equation*}
%\label{eq:dimf=0}
\abs{\{\pi_p(g) \in \Mat_d(\Fbb_p) \mid g \in \Gamma,\, f(g) = 0\}} \ll_{\Gbb,D} p^{-1} \abs{\Gbb_p(\Fbb_p)}
\end{equation*}
for every prime number $p \geq p_0(\Gbb,D)$.
Combined with Corollary \ref{cr:gapEquid}, this yields
\[\mu^{*n}(\{g \in \Gamma \mid f(g) = 0\}) \ll_{\Gbb,D} p^{-1}.\]
for all $n \geq C \log p$, where $C$ is a constant depending only on $\mu$. We conclude by choosing $p$ to be a prime number such that $p \asymp e^{n / C}$ and $p \geq p_0$.
%By \cite[Chapter I, \S 1.2, Proposition]{Borel}, the identity component of $\Gbb$ is defined over $\Q$. It follows that all the connected components of $\Gbb$ are defined over $\Q$ because $\Gbb(\Q)$ acts transitively on the set of connected components of $\Gbb$. 
%We get \eqref{eq:dimf=0} by applying the lemma below to each of the connected components and summing up the inequality thus obtained.
\end{proof}

\subsection{Large deviation estimates for random matrix products.}
In this subsection, $\mu$ is a Borel probability measure on $\SL_d(\R)$, not necessarily supported on matrices with integer coefficients. 
By $\Gamma$ we denote the closure of the subsemigroup generated by $\Supp(\mu)$ and by $G$ the group of $\R$-points of the Zariski closure of $\Gamma$. 

Let us first recall the large deviation estimates for random matrix products.
This result is originally due to Le Page~\cite{LePage}, and the version below is taken from Bougerol~\cite[Theorem V.6.2]{BougerolLacroix}.
For $g \in \GL_d(\R)$, denote by $\sigma_1(g) \geq \dots \geq \sigma_d(g) > 0$ the singular values of $g$ ordered decreasingly.
Recall that the Lyapunov exponents $(\lambda_k)_{1\leq k \leq d}$ of a measure $\mu$ on $\GL_d(\R)$ is defined as the unique real numbers satisfying
\[\lambda_1 + \dots + \lambda_k = \lim_{n\to +\infty} \frac{1}{n} \int \log \norm{\wedge^k g} \dd \mu^{*n}(g),\quad k = 1, \dotsc, d.\]

\begin{theorem}[Large deviation estimates]%see also {\cite[Theorem 3.4]{Breuillard}} and {\cite[Lemma 14.11]{BenoistQuint}}
\label{thm:LargeD}
Let $\mu$ be a Borel probability measure on $\GL_d(\R)$ having a finite exponential moment. 
Let $\lambda_{k}$ denote the $k$-th Lyapunov exponent associated to $\mu$.
Assume that the group $G$ generated by $\mu$ acts strongly irreducibly on $\R^d$.
For any $\omega > 0$, there is $c > 0$, $n_0 > 0$ such that  the following holds.
\begin{enumerate}
\item \label{it:LargeDn} For all $n \geq n_0$,
\[\mu^{*n} \bigl(\bigl\{ g \in \Gamma \mid \abs{ \frac{1}{n}\log \norm{g} - \lambda_1} \geq \omega \bigr\}\bigr) \leq e^{-cn}.\]
\item \label{it:LargeDsv} For all $k = 1, \dotsc, d$ and all $n \geq n_0$,
\[\mu^{*n} \bigl(\bigl\{ g \in \Gamma \mid \abs{ \frac{1}{n}\log \sigma_k(g) - \lambda_{k}} \geq \omega \bigr\}\bigr) \leq e^{-cn}.\]
\item \label{it:LargeDnc}  For all $n \geq n_0$, For all $v \in \R^d \setminus \{0\}$,
\[\mu^{*n} \bigl(\bigl\{ g \in \Gamma \mid \bigl\lvert \frac{1}{n}\log \frac{\norm{gv}}{\norm{v}} - \lambda_1 \bigr\rvert \geq \omega \bigr\}\bigr) \leq e^{-cn}.\]
%\item \label{it:LargeDmc} If the action of $\Gamma$ on $\R^d$ is strongly irreducible and proximal then for all nonzero vectors $x \in \R^d$ and all nonzero linear forms $f \in (\R^d)^*$,
%\[\mu^{*n} \bigl\{g \in \Gamma \mid \abs{f(gx)}\leq e^{-\omega l}\norm{gx}\,\norm{f}\bigr\} \leq e^{-cl}.\]
\end{enumerate}
\end{theorem}
%Item~\ref{it:LargeDn} is, of course, a special case of Item~\ref{it:LargeDsv} since $\sigma_1(g) = \norm{g}$.
%\comm{Si je me souviens bien, les deux premiers points ne requièrent pas d'hypothèse d'irréductibilité. Quant au troisième, il doit rester valable si l'on suppose seulement que l'action de $\Gamma$ est irréductible, mais cela est peut-être un peu plus subtil à montrer.}

Recall that the \emph{proximal dimension} of the group $G$ is the integer
\[ r = \min\{\, \rank g\ ;\ g\in\overline{\R G}\setminus\{0\} \,\},\]
and that $G$ is said to be \emph{proximal} if $r=1$.
For $g \in G$, we write its Cartan decomposition $g = k\diag(\sigma_1(g),\dotsc,\sigma_d(g))l$, where $k$ and $l$ are orthogonal matrices and $\sigma_1(g) \geq \dots \geq \sigma_d(g)$ are the singular values of $g$.
We also define
\[V^+_g = k \Span(e_1,\dotsc,e_r) \; \text{ and } \; W^-_g = l^{-1}\Span(e_{r+1},\dotsc, e_d)\]
where $(e_1,\dotsc,e_d)$ is the standard basis of $\R^d$.
One important consequence of the above theorem is the following regularity statement for the law at time $n$ of the random walk.
It is due to Guivarc'h \cite{guivarch}, and also appears in \cite[Lemma 4.5]{BFLM}.

\begin{lemma}\label{prox}
Let $\mu$ be a Borel probability measure on $\GL_d(\R)$ having a finite exponential moment. 
Assume that the group $G$ generated by $\mu$ is proximal and acts strongly irreducibly on $\R^d$.
Then, there exists $\kappa>0$ such that for every $n\in\N$,
\begin{equation}
\forall \rho \geq e^{-n},\quad \mu^{*n} \bigl(\bigl\{ g \in \Gamma \mid \dang(V^+_{\transp{g}}, v^\perp) \leq \rho \bigr\}\bigr) \ll_\mu \rho^\kappa.
\end{equation}
\end{lemma}

%\comm{Ce lemme est faux sans l'hypothèse de proximalité, si on remplace $v$ par un sous-espace arbitraire $P$ de dimension $r$. En effet, Weikun m'a montré un exemple où l'ensemble limite de $G$ dans $\Grass(r,d)$ est inclus dans un pinceau. Il semble que ce que permet de faire l'astuce de Bougerol et l'algèbre linéaire d'Emmanuel, c'est de trouver, pour chaque $v$, un sous-espace $P$ de dimension $r$ contenant $v$ et pour lequel l'estimée ci-dessus est valable.}

We shall use this lemma to derive the following proposition, which will be useful in the proof of Proposition~\ref{pr:NCaffine}.

\begin{proposition}
\label{pr:LDrhokappa}
Let $\mu$ be a Borel probability measure on $\GL_d(\R)$ having a finite exponential moment. 
Assume that $G$ acts strongly irreducibly on $\R^d$.
Then there exists $\kappa > 0$ such that 
for all $n \geq 1$ and all $\rho \geq e^{-n}$, for every $v\in\R^d\setminus\{0\}$,
%\marginpar{\small Ce choix de $e^{-n}$ est bien sûr arbitraire, on peut le remplacer par $e^{-10n}$ par exemple.}
\[
\mu^{*n} \bigl(\bigl\{ g \in \Gamma \mid \norm{gv}\leq \rho \norm{g} \norm{v} \bigr\}\bigr) \ll_\mu \rho^\kappa.
\]
\end{proposition}
\begin{proof}
%\comm{C'est un peu dommage de se ramener au cas proximal dans cette démonstration. D'autant que cela fait utiliser deux fois l'astuce de Bougerol: une première fois pour démontrer le théorème~\ref{thm:LargeD}, et une seconde ci-dessous. À étudier...}
By \cite[Lemma~14.2 (i)]{BenoistQuint},
\begin{equation}
\label{eq:normgv}
\forall g \in \GL_d(\R), \forall v \in \R^d\setminus\{0\},\quad \norm{gv} \geq \dang(\R v , W^-_g) \norm{g} \norm{v},
\end{equation}
where $\dang$ is defined, for any two subspaces $V,W \subset \R^d$ of $\R^d$ with respective orthonormal bases $(v_1,\dots,v_s)$ and $(w_1,\dots,w_t)$, by the formula
\[
\dang(V,W) = \norm{v_1 \wedge \dotsm \wedge v_s \wedge w_1 \wedge \dotsm \wedge w_t}.
\]
If $G$ is proximal, then Lemma~\ref{prox} applied to the transposed random walk, which is also proximal, shows that there exists $\kappa > 0$ such that 
\begin{equation}
\label{eq:V+rhokappa}
\forall \rho \geq e^{-n},\quad \mu^{*n} \bigl(\bigl\{ g \in \Gamma \mid \dang(V^+_{\transp{g}}, v^\perp) \leq \rho \bigr\}\bigr) \ll_\mu \rho^\kappa.
\end{equation}
Noting that $(W^-_g)^\perp = V^+_{\transp{g}}$ and that $\dang(V,W)=\dang(V^\perp,W^\perp)$ if $\dim(V)+\dim(W)=d$, we see that the result follows from \eqref{eq:normgv}.

We now use Bougerol's trick~\cite[Proof of Theorem V.6.2]{BougerolLacroix} to reduce the proposition to the proximal case.
The argument is based on the following lemma, taken from unpublished notes of Emmanuel Breuillard \cite[Lemma~3.2 and Lemma~3.3]{Breuillard}, generalizing \cite[Lemma 4.36]{BenoistQuint}.
Recall that $r$ denotes the proximality dimension of $G$.

\begin{lemma}
We have a decomposition of $\wedge^r \R^d$ into $G$-invariant subspaces
\[\wedge^r \R^d = \Lambda_+ \oplus \Lambda_0\]
such that the action of $G$ on $\Lambda_+$ is strongly irreducible and proximal and moreover,
\begin{equation}
\label{eq:gLambda+}
\forall g \in G,\quad \norm{(\wedge^r g)_{\mid \Lambda_+}} \gg_G \norm{g}^r \geq \norm{\wedge^r g}.
\end{equation}
Moreover, if $\pi_+ \colon \wedge^r \R^d \to \Lambda_+$ denotes the projection with respect to this decomposition, then for every $v \in \R^d$, we can find a subspace $P \subset \R^d$ of dimension $r$ and containing $v$ such that
\begin{equation}
\label{eq:pickPfromv}
\pi_+(\mathbf{v}_{\! P}) \gg_G \norm{\mathbf{v}_{\! P}},
\end{equation}
where $\mathbf{v}_{\! P} \in \wedge^r \R^d$ is the wedge product of the elements of a basis of $P$.
\end{lemma}

Now, observe that for every $g$ in $\GL_d(\R)$, 
\[\dang(\R v , W^-_g) \geq \dang(P,W^-_g).\]
Hence, recalling \eqref{eq:normgv}, it suffices to prove, for some $\kappa > 0$ and $n_0 \geq 1$,
\begin{equation}
\label{eq:W-rhokappa}
\forall n \geq n_0,\, \forall \rho \geq e^{-n},\quad \mu^{*n} \bigl(\bigl\{ g \in \Gamma \mid \dang(P, W^-_g) \leq \rho \bigr\}\bigr) \ll_\mu \rho^\kappa.
\end{equation}
By \cite[Lemma~14.2 (i)]{BenoistQuint} applied in $\wedge^r\R^d$ (see also \cite[Lemma 4.2]{Breuillard}) we find
\[
\forall g \in \GL_d(\R),\quad  \frac{\norm{(\wedge^r g) \mathbf{v}_{\! P}}}{\norm{\wedge^r g}\, \norm{\mathbf{v}_{\! P}}} \leq \dang(P, W^-_g) + \frac{\sigma_{r+1}(g)}{\sigma_r(g)}.
\]
Combined with \eqref{eq:gLambda+} and \eqref{eq:pickPfromv}, this yields
\[
\forall g \in G,\quad \frac{\norm{(\wedge^r g)_{\mid \Lambda_+} \pi_+(\mathbf{v}_{\! P})}}{\norm{(\wedge^r g)_{\mid \Lambda_+}}\, \norm{\pi_+(\mathbf{v}_{\! P})}} \ll_G \dang(P, W^-_g) + \frac{\sigma_{r+1}(g)}{\sigma_r(g)}.
\]
By a result of Guivarc'h-Raugi~\cite{GuivarchRaugi}, $\lambda_r > \lambda_{r+1}$. Applying Theorem~\ref{thm:LargeD}\ref{it:LargeDsv} to $k = r$ and $r + 1$ and $\omega_0 = (\lambda_r - \lambda_{r + 1})/3 > 0$, we get $c > 0$ and $n_0 \geq 1$ such that
\[
\forall n \geq n_0,\quad \mu^{*n} \bigl(\bigl\{ g \in \Gamma \mid \frac{\sigma_{r+1}(g)}{\sigma_r(g)} \leq e^{-\omega_0 n} \bigr\}\bigr) \geq 1 - e^{-cn}.
\]
Note that $e^{-cn} \leq \rho^c$.
The desired estimate \eqref{eq:W-rhokappa} then follows by the proximal case applied to the induced random walk on $\Lambda_+$ and to the vector $\pi_+(\mathbf{v}_{\! P}) \in \Lambda_+$.
\end{proof}

\subsection{Escaping a small neighborhood of a subvariety}
For an integer $D \geq 0$, and a regular function $f \in \R[\Gbb]$ on $\Gbb$, we say that $f$ has degree at most $D$ if it can be represented by a polynomial on $\Mat_d$ of degree at most $D$. Denote by $\R[\Gbb]_{\leq D}$ the finite-dimensional subspace consisting of regular functions of degree at most $D$. We fix a norm on $\R[\Gbb]_{\leq D}$.

\begin{lemma}
\label{lm:mundef<e}
Let $\mu$ be a Borel probability measure on $\SL_d(\Z)$ having a finite exponential moment. 
Assume that the Zariski closed subgroup $\Gbb$ generated by $\Supp\mu$ is semisimple and connected.

Given an integer $D \geq 1$, there exist constants $C> 0$, $c > 0$ and $n_0 \geq 1$ depending on $\mu$ and $D$ such that
\[\forall f \in \R[\Gbb]_{\leq D},\; \forall n \geq n_0, \quad \mu^{*n} \bigl(\bigl\{ g \in \Gamma \mid \abs{f(g)} < e^{-C n} \norm{f} \bigr\}\bigl) \leq e^{-c n}.\]
\end{lemma}
\begin{proof}
By Theorem~\ref{thm:LargeD}\ref{it:LargeDn}, there is $c > 0$ such that for $n$ large enough
\[\mu^{*n} \bigl( \bigl\{ g \in \Gamma \mid \norm{g} \geq e^{2 \lambda_1 n}  \bigr\} \bigr) \leq e^{-cn},\]
where $\lambda_1$ is the top Lyapunov exponent associated to the random walk defined by $\mu$ on $\R^d$.
Thus, we are left to bound from above the $\mu^{*n}$\dash{}measure of the set
\[A_{f,C} = \bigl\{ g \in \Gamma \mid \abs{f(g)} \leq e^{-C n} \norm{f} \text{ and } \norm{g} \leq e^{2 \lambda_1 n} \bigr\}.\]

Let us abbreviate $V = \R[\Gbb]_{\leq D}$. Let $V_\Q$ denote the set of functions $f \in V$ which are $\Q$-rational, i.e. represented by polynomials on $\Mat_d$ with coefficients in $\Q$.
Note that $V_\Q$ defines a $\Q$\dash{}structure on $V$.

We claim that if $C$ is chosen large enough then for each $f \in V$, $A_{f,C}$ must be contained in some subvariety $\{f_0 = 0\}$ where $f_0 \in V_\Q \setminus \{0\}$. Then Proposition~\ref{pr:gapEscape} applied to $f_0$ allows to conclude. 

For every $g \in \Gamma$ let $\ev_{\!g} : V \to \R$ be the evaluation map
\[\forall v\in V, \quad \ev_{\!g}(v) = v(g).\]
Since the matrices $g \in A_{f,C}$ have integer coefficients, the intersection 
\[W = \bigcap_{g \in A_{f,C}} \ker(\ev_{\!g})\]
is a subspace of $V$ defined over $\Q$, i.e. $W = \R \otimes_\Q (W \cap V_\Q)$.
We want to show that $W \cap V_\Q$ contains nonzero element. 
Assume for a contradiction that $W = \{0\}$.

Write $s = \dim(V)$. 
We can choose $g_1 ,\dotsc, g_s \in A_{f,C}$ such that 
\[\{0\} = \bigcap_{i = 1}^s \ker(\ev_{\!g_i}).\]
Fix a basis $(v_1, \dotsc, v_s)$ of $V$ in which each element is represented by a polynomial on $\Mat_d$ with coefficients in $\Z$.
Thus, the map $\Phi \colon V \to \R^s$ defined by
\[\forall f \in V,\quad \Phi(v) = \bigl(\ev_{\!g_i}(v)\bigr)_{1 \leq i \leq s}\]
is invertible and has integer coefficients when expressed in the basis $(v_1,\dotsc,v_s)$ and the standard basis of $\R^s$. 
Thus, in these bases, the determinant of $\Phi$ satisfies $\abs{\det(\Phi)} \geq 1$.
Moreover,
\[\norm{\Phi} \ll \max_{1 \leq i,j \leq s} \abs{\ev_{\!g_i}(v_j)} \ll_{\Gbb,D} \max_{1 \leq i \leq s} \norm{g_i}^D \leq e^{2D\lambda_1 n}.\]
It follows that
\[\norm{\Phi^{-1}} \ll \frac{\norm{\Phi}^{s-1}}{\abs{\det(\Phi)}} \leq e^{2Ds\lambda_1 n}.\]
By the definition of $A_{f,C}$, we have
\[\norm{\Phi(f)} = \max_{1 \leq i \leq s} \abs{f(g_i)} \leq e^{-Cn} \norm{f}.\]
Thus,
\[\norm{f} \leq \norm{\Phi^{-1}}\norm{\Phi(f)} \ll_{\Gbb,D} e^{(2Ds\lambda_1 - C)n} \norm{f}.\]
We get a contradiction if $C$ is chosen to be larger than $2Ds\lambda_1 + O_{\Gbb,D}(1)$.
\end{proof}

The following is a variant and an easy consequence of the previous lemma.
\begin{lemma}
\label{lm:mundef<eg}
Let $\mu$ be a Borel probability measure on $\SL_d(\Z)$ having a finite exponential moment. 
Assume that the Zariski closed subgroup $\Gbb$ generated by $\Supp\mu$ is semisimple and connected.

Given an integer $D$, there exist constants $C> 0$, $c > 0$ and $n_0 \geq 1$ depending on $\mu$ and $D$ such that
\[\forall f \in \R[\Gbb]_{\leq D},\; \forall n \geq n_0, \quad \mu^{*n} \bigl(\bigl\{ g \in \Gamma \mid \abs{f(g)} < e^{-C n} \norm{g} \norm{f} \bigr\}\bigl) \leq e^{-c n}.\]
\end{lemma}

\begin{proof}
For any $n \geq 1$ and any $C > 0$, we have 
\begin{multline*}
\mu^{*n} \bigl(\bigl\{ g \in \Gamma \mid \abs{f(g)} < e^{-(C + 2\lambda_1) n} \norm{g} \norm{f} \bigr\}\bigl)\\
\leq \mu^{*n} \bigl(\bigl\{ g \in \Gamma \mid \abs{f(g)} < e^{-C n}  \norm{f} \bigr\}\bigl) + \mu^{*n} \bigl(\bigl\{ g \in \Gamma \mid \norm{g} \geq e^{2 \lambda_1 n}  \bigr\}\bigl).
\end{multline*}
We conclude by using Lemma~\ref{lm:mundef<e} for the first term and Theorem~\ref{thm:LargeD}\ref{it:LargeDn} for the second.
\end{proof}

\subsection{Non-concentration near affine subspaces}
We now want to prove Proposition~\ref{pr:NCaffine}.
Of course, if we are to show that the random walk does not concentrate near any proper affine subspace in the algebra $E$, we should first check that the group $G$ is not trapped in any proper affine subspace.

\begin{lemma}
\label{lm:notaffine}
Let $G$ be a subgroup of $\GL_d(\R)$ acting irreducibly on $\R^d$, and $E$ the associative subalgebra generated by $G$ in $\Mat_d(\R)$.
Then $G$ is not contained in any proper affine subspace of $E$.
\end{lemma}
\begin{proof}
Equivalently, we have to show that the linear span $W = \Span(G - 1)$ of $G - 1$ is $E$.
For this, it suffices to prove that $1 \in W$. 
Firstly, $W$ is closed under multiplication.
Indeed, any product between two elements of $W$ is a linear combination of elements of the form $(g-1)(h-1)$ with $g, h \in G$ and we have
\[(g-1)(h-1) = (g h - 1) - (g - 1) - (h - 1) \in W .\]
Secondly, any subspace of $\R^d$ preserved by $W$ is preserved by $G$, hence the only subspaces preserved by $W$ are $\R^d$ and $\{0\}$.
We conclude by using the following algebraic lemma.
\end{proof}

Momentarily, in the next lemma and its proof, algebras are not assumed to be unital. Thus, an subalgebra is a linear subspace that is closed under multiplication. Accordingly, a left (resp. right) ideal, is a subspace preserved under multiplication on the left (resp. right) by all elements of the algebra.  
\begin{lemma}
If a nonzero subalgebra of $\Mat_d(\R)$ does not preserve any proper nontrivial subspaces of $\R^d$, then it contains the multiplicative identity of $\Mat_d(\R)$.
\end{lemma}

\begin{proof}
Let $W \subset \Mat_d(\R)$ be such a subalgebra. We first show that the only nilpotent right ideal of $W$ is the zero ideal. Indeed, let $I$ be a  nonzero nilpotent right ideal of $W$. Let $k$ be the maximal number such that $I^k \neq 0$. Let $f_0 \in I^k$. Since $I^{k+1} = 0$, we have
\[f_0(\R^d) \subset \bigcap_{f \in I} \ker(f).\]
The intersection on the right-hand side is preserved by $W$ and not equal to $\R^d$, because $I$ is a nonzero right ideal.
Then $f_0$ must be zero, which contradicts $I^k \neq 0$.

Thus, $W$ is an algebra without radical \cite[Chapter XVI, \S 116]{VanderWaerden}. By \cite[Chapter XVI, \S 117]{VanderWaerden}\footnote{More precisely, we apply the theorem to algebras, or rings with operator domain $\R$, using the terminology of Van der Waerden, see \cite[Chapter XVI, \S 115]{VanderWaerden}.}, $W$ has a multiplicative identity $1_W$. 
Its image $1_W(\R^d)$ is preserved by $W$ and nonzero since $W$ is nonzero. Hence $1_W(\R^d) = \R^d$ and $1_W \in \GL_d(\R)$. Then $1_W^2 = 1_W$ forces $1_W$ to be the identity of $\Mat_d(\R)$.
\end{proof}

Now we are ready to prove Proposition~\ref{pr:NCaffine}.

\begin{proof}[Proof of Proposition~\ref{pr:NCaffine}]
By Lemma~\ref{lm:notaffine}, $\R[\Gbb]_{\leq 1}$ is isomorphic to $\R \oplus E^*$, the space of affine mappings from $E$ to $\R$. 
On $\R \oplus E^*$, let $G$ act by
\[\forall g \in G,\, \forall f \in \R \oplus E^*, \, \forall x \in E, \;  (g \cdot f)(x) = f(xg).\]
%Then $\R[\Gbb]_{\leq 1}$ and $\R\oplus E^*$ are isomorphic as $G$-modules.
By Lemma~\ref{lm:mundef<eg}, there exist $C_1 \geq 1$ and $c_1 > 0$ such that 
\begin{equation}
\label{eq:mu*mdef}
\forall f \in \R \oplus E^*,\; \forall m \geq 1, \quad \mu^{*m} \bigl(\bigl\{ g \in \Gamma \mid \abs{f(g)} < e^{-C_1 m} \norm{g} \norm{f} \bigr\}\bigl) \ll_\mu e^{-c_1 m}.
\end{equation}

Given a proper affine subspace $W \subset E$, there exists $f \in \R \oplus E^*$ such that its linear part $f_1 \in E^*$ has norm $\norm{f_1} = 1$ and 
\[
\forall g \in E,\; d(g, W ) = \abs{f(g)}.
\]
Let $\rho \geq e^{-n}$.
Pick $m$ such that $e^{-C_1 m} \asymp \rho^{1/2}$.
Using the relation $\mu^{*n} = \mu^{*m} * \mu^{*(n - m)}$, we have 
\[
\mu^{*n}(\{ g \in \Gamma  \mid  \abs{f(g)} \leq \rho \norm{g} \}) = \int_{\Gamma} \mu^{*m}(\{ g \in \Gamma  \mid  \abs{(h \cdot f)(g)} \leq \rho \norm{gh} \}) \dd \mu^{*(n - m)}(h)
\]
We distinguish two cases according to whether $\rho \norm{h} \leq e^{-C_1 m} \norm{h \cdot f}$. If this is the case, then $\rho \norm{gh} \leq e^{-C_1 m} \norm{g} \norm{h \cdot f}$ and then by \eqref{eq:mu*mdef},
\[
 \mu^{*m}\bigl( \bigl\{ g \in \Gamma  \mid  \abs{(h \cdot f)(g)} \leq \rho \norm{gh} \bigr\}\bigr) \ll_\mu e^{-c_1 m} \ll \rho^{c_1/2C_1} .
\]
Otherwise, $\norm{h \cdot f} \leq  \rho e^{C_1 m} \norm{h} \leq \rho^{1/2} \norm{h}$ by the choice of $m$.

Thus,
\[
\mu^{*n}(\{ g \in \Gamma  \mid  \abs{f(g)} \leq \rho \norm{g} \}) \ll_\mu \rho^{c_1/2C_1} + \mu^{*(n - m)}\bigl( \bigl\{ h \in \Gamma \mid \norm{h \cdot f} \leq \rho^{1/2} \norm{h} \bigr\}\bigr).
\]
Now observe that $E^*$ is a submodule of the the semisimple $G$-module $\Mat_d(\R)^*$, which is isomorphic to the sum of $d$ copies of the simple $G$-module $\R^d$. %the given representation of $G$. 
It follows that $E^*$ is isomorphic to the sum of $\frac{\dim(E)}{d}$ copies of $\R^d$.
For each $i= 1, \dotsc, \frac{\dim(E)}{d}$, let $\pi_i \colon \R\oplus E^* \to \R^d$ denote the projection to the $i$-th $\R^d$-factor. Remembering $\norm{f_1} = 1$, we obtain
$\norm{\pi_i(f)} \gg_G 1$ for some $i$.
Then $\norm{h \cdot f} \leq \rho^{1/2} \norm{h}$ implies $\norm{h\pi_i(f)} \ll_G \rho^{1/2} \norm{h}\norm{\pi_i(f)}$. 
Hence
\begin{multline*}
\mu^{*(n - m)}\bigl( \bigl\{ h \in \Gamma \mid \norm{h \cdot f} \leq \rho^{1/2} \norm{h} \bigr\}\bigr) \\
\leq \mu^{*(n - m)}\bigl( \bigl\{ h \in \Gamma \mid \norm{h\pi_i(f)} \leq C_2 \rho^{1/2} \norm{h}\norm{\pi_i(f)} \bigr\}\bigr)
\end{multline*}
where $C_2$ is a constant depending only on $G$.

By our choice of $m$, we have $C_2 \rho^{1/2} \geq e^{-(n-m)}$. Hence, by Proposition~\ref{pr:LDrhokappa},
\[
 \mu^{*(n - m)}\bigl( \bigl\{ h \in \Gamma \mid \norm{h\pi_i(f)} \leq C_2 \rho^{1/2} \norm{h}\norm{\pi_i(f)} \bigr\}\bigr) \ll_\mu \rho^{\kappa_2/2}
\]
where $\kappa_2 > 0$ is the constant given by Proposition~\ref{pr:LDrhokappa} which depends only on $\mu$.
This proves the desired estimate with $\kappa = \min\{\frac{c_1}{2C_1}, \frac{\kappa_2}{2}\}$.
\end{proof}

\subsection{Escaping a larger neighborhood of a subvariety}
The rest of this section is devoted to the proof of Proposition~\ref{pr:NCdet}.
The idea is to generalize what we did above for affine subspaces.
This time, the variety that we want to avoid is defined by a general polynomial map $f$ on the algebra $E$, so that we shall have to consider  the representation  $\rho \colon G \to \GL(\R[\Gbb])$ defined by
\[\forall g \in G,\, \forall f \in \R[\Gbb],\, \forall x \in \Gbb, \quad (\rho(g)f)(x) = f(xg).\]
We refer to finite-dimensional subrepresentations of this representation as $G$-modules.
For a $G$-module $M$, we denote by $\lambda_1(\mu,M)$ the top Lyapunov exponent associated to the random walk on $M$ defined by $\mu$:
\[\lambda_1(\mu,M) = \lim \frac{1}{n} \int_G \log \norm{\rho(g)_{\mid M}} \dd \mu^{*n}(g)\]
where $\norm{\ }$ denotes some operator norm.

For a real number $\lambda \geq 0$, define $M_\lambda$ to be the sum of submodules $M$ of $\R[\Gbb]_{\leq D}$ such that $\lambda_1(\mu,M) \geq \lambda$.
Let $p_\lambda \colon \R[\Gbb]_{\leq D} \to M_\lambda$ be an epimorphism of $G$-modules onto $M_\lambda$.
This exists because $G$ is semisimple.
Remark that $M_\lambda$ is a sum of isotypical components in $\R[\Gbb]_{\leq D}$ so that $p_\lambda$ is uniquely defined.

\begin{proposition}
\label{pr:LDPpolynome}
Let $\mu$ be a probability measure on $\SL_d(\Z)$, for some $d\geq 2$.
Let $\Gamma$ denote the subsemigroup generated by $\mu$ and $\Gbb$ the Zariski closure of $\Gamma$ in $\SL_d$.
Assume that 
\begin{enumerate}[label=(\alph*)]
\item  the measure $\mu$ has a finite exponential moment,
\item  the algebraic group $\Gbb$ is Zariski connected and semisimple,
\end{enumerate} 
and let the notation be as above.

Given $D$, $\lambda \geq 0$ and $\omega > 0$, there is $c > 0$ and $n_0 \geq 1$ (depending also on $\mu$) such that
\[\forall f \in \R[\Gbb]_{\leq D},\; \forall n \geq n_0,\quad \mu^{*n} \bigl(  \bigl\{ g \in \Gamma \mid \abs{f(g)} \leq e^{(\lambda - \omega)n} \norm{p_\lambda(f)}\bigr\}  \bigr)  \leq e^{-c n}.\]
\end{proposition}
%Since $\rho$ is an algebraic representation, the assumption that $\mu$ has a finite exponential moment implies that $\rho(\mu)$ restricted any finite-dimensional submodule also has a finite exponential moment.
\begin{proof}
Note that, for any $0 < m  < n$, we have
\begin{multline}
\label{eq:n=n-m+m}
\mu^{*n}\bigl\{ g \in \Gamma \mid \abs{f(g)} \leq  e^{(\lambda - \omega)n} \norm{ p_\lambda(f)} \bigr\} \\
= \int_\Gamma \mu^{*m} \bigl\{ g \in \Gamma \mid  \abs{ (\rho(h)f)(g)} \leq e^{(\lambda - \omega) n} \norm{p_\lambda(f)} \bigr\} \dd \mu^{*(n-m)}(h)
\end{multline}

For any $f \in \R[\Gbb]_{\leq D}$, there is a simple $G$-module $M$ contained in $M_\lambda$ such that
\[\norm{p_M(f)} \gg_{\Gbb,D,\lambda} \norm{p_\lambda(f)}\]
where $p_M \colon M_\lambda \to M$ is a projection of $G$-modules.
By definition of $M_\lambda$, the top Lyapunov exponent in $M$ satisfies $\lambda_1(\mu,M) \geq \lambda$.
Note that since $\Gbb$ is Zariski connected, the $G$-action on the simple module $M$ is strongly irreducible. Hence, by the large deviation estimate Theorem~\ref{thm:LargeD}\ref{it:LargeDnc}, there is $c > 0$ and $n_0 \geq 1$ such that for all $m > 0$ satisfying $n - m \geq n_0$,
\[ \mu^{*(n-m)} \bigl\{ h \in \Gamma \mid \norm{\rho(h) p_M(f)} \leq e^{(\lambda-\frac{\omega}{2}) (n-m)} \norm{p_M(f)}\bigr\}
\leq e^{-c(n-m)},\]
and hence
\begin{equation}
\label{eq:LDPrhohf}
 \mu^{*(n-m)} \bigl\{ h \in \Gamma \mid \norm{\rho(h) f} \leq e^{ (\lambda-\frac{\omega}{2}) (n-m)} \norm{p_\lambda(f)}\bigr\}
\leq e^{-c(n-m)}.
\end{equation}

Applying Lemma~\ref{lm:mundef<e} to the function $\rho(h)f$, for $h \in \Gamma$, we obtain $ \forall m \geq m_0$
\begin{equation}
\label{eq:mumdef}
\mu^{*m} \bigl\{ g \in \Gamma \mid \abs{(\rho(h)f)(g)} < e^{-C m} \norm{\rho(h) f} \bigr\} \leq e^{-c m},
\end{equation}
for some $C > 0$, $c > 0$ and $m_0 \geq 1$.
Setting $m = \bigl \lfloor \frac{\omega n}{2(C + \lambda)} \bigr \rfloor$, so that
\[(\lambda - \omega)n - (\lambda - \frac{\omega}{2})(n- m) \leq - C m,\]
the desired inequality follows from \eqref{eq:n=n-m+m}, \eqref{eq:LDPrhohf} and \eqref{eq:mumdef}.
\end{proof}

\subsection{Criterion to have nonzero component in modules of maximal Lyapunov exponent}
In order to use Proposition~\ref{pr:LDPpolynome}, we need to be able to say when a regular function has a nonzero component in a simple submodule of large Lyapunov exponent.

\begin{lemma}
\label{lm:criterionL}
Let $\mu$ be a probability measure on $\SL_d(\R)$, $d\geq 2$, with some finite exponential moment.
Assume that the group $\Gamma$ generated by $\mu$ is non-compact and acts irreducibly on $\R^d$, and that its Zariski closure $\Gbb$ is connected.

Let $f \in \R[\Mat_d]$ be a polynomial of degree $D \geq 1$ whose degree $D$ homogeneous part does not vanish on the algebra $E$ generated by $G$.
The following holds for every integer $k \geq \dim(E)$.
Consider the polynomial $\bar{F} \in \R[\Mat_d^k]$ defined by
\[\bar{F}(x_1,\dotsc, x_k) = f(x_1 + \dotsb + x_k)\]
and let $F \in \R[\Gbb^k]$ be the restriction of $\bar{F}$ to $\Gbb^k$. 
Then
\[p(F) \neq 0\]
where $p \colon \R[\Gbb^k]_{\leq D} \to \R[\Gbb^k]_{\leq D}$ is the projection to the sum of all simple $G^k$-submodules $M$ of $\R[\Gbb^k]_{\leq D}$ having $\lambda_1(\mu^{\otimes k}, M) \geq D \lambda_1(\mu,\R^d)$.
\end{lemma}

We shall use the theory of the highest weight as well as the theory of random walks on semisimple groups. 
So let us fix some notation and recall briefly the needed results.
%\marginpar{\small Donner une référence pour les notions basiques.}
Let $\gfr$ denote the Lie algebra of $G$.
Let $K$ be a maximal compact subgroup of $G$.
Inside the orthogonal complement, with respect to the Killing form, of the Lie algebra of $K$, we choose a Cartan subspace $\afr$ of $\gfr$.
Every algebraic representation of $G$ is diagonalizable for $\afr$. That is, for every $G$-module $M$, we have
\[M = \bigoplus_{\chi \in \afr^*} M_\chi\]
where for each $\chi \in \afr^*$, $M_\chi$ is the corresponding weight space
\[M_\chi = \{v \in M \mid \forall a \in \afr,\, \exp(a)\cdot v = e^{\chi(a)} v \}.\]
The linear forms $\chi \in \afr^*$ for which $M_\chi \neq \{0\}$ are called the
%restricted
 weights of $M$. 
Denote by $\Sigma(M)$ the sets of weights of $M$.

The set of of nontrivial weights of the adjoint representation of $G$ is the set of restricted roots. We denote it by $\Sigma$. It forms a root system.
We fix a set $\Sigma_+$ of positive roots and denote by $\afr^+$ the associated Weyl chamber:
\[\afr^+ = \{a \in \afr \mid \forall \alpha \in \Sigma_+,\, \alpha(a) \geq 0\}.\]
We also write $\afr^{++}$ to denote the interior of the Weyl chamber:
\[\afr^{++} = \{a \in \afr \mid \forall \alpha \in \Sigma_+,\, \alpha(a) > 0\}.\]

Let $g \in G$. The Cartan projection $\kappa(g)$ of $g$ is the $\afr^+$-part in its Cartan decomposition, that is, the unique element in $\afr^+$ such that $g \in K \exp(\kappa(g)) K$. The law of large numbers for a semisimple group, \cite[Theorem 10.9]{BenoistQuint}, says that there is an element $\vec\lambda(\mu)$ in $\afr^{++}$, called the Lyapunov vector associated to $\mu$, such that
\[ \vec\lambda(\mu) = \lim_{n \to +\infty} \frac{1}{n} \int_G \kappa(g) \dd \mu^{*n}(g).\]

If $M$ is a simple $G$-module, then $M$ has a highest weight, denoted by $\chi_M \in \Sigma(M)$, so that for any weight $\chi \in \Sigma(M)$, $\chi_M - \chi$ is a sum of positive roots. By \cite[Corollary 10.12]{BenoistQuint}, we have
\[
\lambda_1(\mu,M) = \chi_M(\vec\lambda(\mu)).
\]
%\comm{Si $M$ n'est pas simple, alors il peut ne pas y avoir de plus haut poids au sens ci-dessus. C'est le cas pour $G=\SL_3$, par exemple, si $M$ est la somme de la représentation standard et de sa représentation duale: la différence des deux plus hauts poids n'est pas égale à une racine.
%L'existence d'un plus haut poids découle du théorème de Poincaré-Birkhoff-Witt, et de l'irréductiblité, je pense.}

\smallskip
Now let us recall the definition of the limit set of the group $G$ in $\Mat_d(\R)$.
We write $\R G$ for the set of all elements $\Mat_d(\R)$ of the form $\lambda g$ with $\lambda \in \R$, and $g \in G$. Let $\overline{\R G}$ denote the closure of $\R G$ in $\Mat_d(\R)$ for the norm topology. 
Let $r_G$ denote the proximal dimension of $G$, defined by
\begin{equation}
\label{eq:defproxdim}
r_G = \min \bigl\{ \rank(\pi) \mid \pi \in \overline{\R G} \setminus \{0\} \bigr\}.
\end{equation}
The limit set of $G$ in $\Mat_d(\R)$ is defined to be
\[\Pi_G = \bigl\{\pi \in \overline{\R G} \mid \rank(\pi)= r_G \bigr\}.\]

\begin{lemma}
\label{lm:limitset}
Let $\Gbb < \SL_d$ be connected semisimple $\R$-group. 
Assume that $G = \Gbb(\R)$ acts irreducibly on $\R^d$.
Let $\pi_0 \in \Mat_d(\R)$ be the spectral projector to the weight space associated to the highest weight.
Then 
\[
\Pi_G = \R^{*} K \pi_0 K = \R^{*} G \pi_0 G.
\]
If moreover $G$ is not compact then, writing $E = \Span_\R(G)$, the sum-set
\[\underbrace{G\pi_0 G + \dotsb + G\pi_0 G}_{\text{$\dim(E)$ times}}\]
contains an open subset of $E$.
\end{lemma}
%\comm{Le cas compact est un peu particulier : le plus haut poids est nul, et $\pi_0=Id$. D'ailleurs, la deuxième partie du lemme est encore vraie dans le cas compact.}
\begin{proof}
Let $\chi_0 = \chi_{\R^d} \in \Sigma(\R^d)$ denote the highest weight of the simple $G$\dash{}module $\R^d$.
For a weight $\chi \in \Sigma(\R^d)$, let $\pi_\chi$ be the spectral projector to the associated weight space.
Let $a \in \afr^{++}$ be any element. We have 
\[\norm{\exp(n a)}^{-1}  \exp(na) = \sum_{\chi \in \Sigma(\R^d)} e^{n (\chi(a)- \chi_0(a))}\pi_\chi.\]
Now by definition of the highest weight, $\chi(a)- \chi_0(a) < 0$ for $\chi \neq \chi_0$. It follows that
\[\pi_0 = \lim_{n \to +\infty} \norm{\exp(n a)}^{-1} \exp(n a)
\in \overline{\R G}.\]

Let $\pi \in \overline{\R G}$ be another nonzero element. There exists sequences $(\lambda_n) \in \R^\N$ and $(g_n) \in G^\N$ such that
\[ \pi = \lim_{n \to +\infty} \lambda_n g_n.\]
Let $g_n = k_n \exp(a_n) l_n \in K \exp(\afr^+) K$ be the Cartan decomposition of $g_n$. By compactness of $K$, replacing $(g_n)$ by a subsequence if necessary, we may assume that $k_n$ converges to $k \in K$ and $l_n$ converges to $l$. Then
\[k^{-1} \pi l^{-1} = \lim_{n \to +\infty} \lambda_n \exp(a_n)\]
Observe that 
\[\exp(a_n) = \sum_{\chi \in \Sigma(\R^d)} e^{\chi(a_n)} \pi_\chi.\]
Hence 
\[\lambda_n \exp(a_n) = \lambda_n e^{\chi_0(a_n)} (\pi_0 + \sum_{\chi \in \Sigma(\R^d) \setminus \{\chi_0\}} e^{\chi(a_n) - \chi_0(a_n)} \pi_\chi).\]
Note that $e^{\chi(a_n) - \chi_0(a_n)} \leq 1$ for all $\chi \in \Sigma(\R^d)$ and $n \geq 1$.
We deduce that $\lambda_n e^{\chi_0(a_n)}$ converges to $\lambda \neq 0$, for otherwise $\pi$ would be zero. Moreover, 
\[\rank(\pi)  = \rank(k^{-1} \pi l^{-1}) \geq \rank(\pi_0).\]
Equality holds if and only if $\lim_{n \to +\infty} e^{\chi(a_n) - \chi_0(a_n)} = 0$ for all $\chi \in \Sigma(\R^d) \setminus \{\chi_0\}$, which in turn is equivalent to
\[k^{-1} \pi l^{-1} = \lambda \pi_0.\]
Therefore, $r_G = \rank(\pi_0)$ and $\Pi_G \subset \R^* K\pi_0 K$.
We conclude by noticing that $\Pi_G$ is invariant under multiplication by $G$ on both sides : $\R^* K\pi_0 K \subset \R^* G\pi_0 G \subset \Pi_G$.

For the last assertion, assume that $G$ is not compact.
Then, $\chi_0 \neq 0$ and therefore $\chi_0(\afr) = \R$.
For $a \in \afr$, $\exp(a) \pi_0 = e^{\chi_0(a)}\pi_0$,
so that $\R^*_+ \pi_0 \subset G \pi_0$ and hence
\[
\R^*_+ G \pi_0 G \subset G \pi_0 G.
\]
Since the action of $G$ on $\R^d$ is irreducible, $E$ is an simple algebra over $\R$, by a version of Wedderburn's theorem \cite[2. page 194]{VanderWaerden}.
Observe that $\Span_\R(G\pi_0 G)$ is a nontrivial two-sided ideal of $E$, hence $\Span_\R(G\pi_0 G) = E$. Therefore we can pick $\dim(E)$ elements $(\pi_1, \dotsc, \pi_{\dim(E)}) $ from $G\pi_0 G$ making a basis of $E$. We conclude that
\[\underbrace{G\pi_0 G + \dotsb + G\pi_0 G}_{\text{$\dim(E)$ times}} \supset \R^*_+ \pi_1 + \dotsb + \R^*_+ \pi_{\dim(E)}\]
contains an open subset of $E$.
\end{proof}

\begin{proof}[Proof of Lemma~\ref{lm:criterionL}]
The Lie algebra of $G^k$ is $\gfr \oplus \dots \oplus \gfr$, in which we choose $\bfr = \afr \oplus \dots \oplus \afr$ to be the Cartan subspace. Then the associated restricted root system is the direct sum $\Sigma \sqcup \dots \sqcup \Sigma \subset \bfr^*$. We choose $\Sigma_+ \sqcup \dots \sqcup \Sigma_+$ as the set of positive roots so that $\bfr^+ = \afr^+ \times \dots \times \afr^+$ is the corresponding Weyl chamber and
\[\vec\lambda(\mu^{\otimes k}) = (\vec\lambda(\mu), \dotsc, \vec\lambda(\mu)) \in \bfr^+\]
is the Lyapunov vector associated to the random walk defined by $\mu^{\otimes k}$. 

For any algebraic representation $\pi$ of $G^k$, we denote by $\Sigma(G^k,\pi)$ the set of weights of $\pi$ with respect to $\bfr$.

Let $\sigma \colon G \to \GL(\R^d)$ denote the standard representation of $G$ and, for $i = 1,\dotsc,k$, let $\sigma_i \colon G^k \to G \to \GL(\R^d)$ denote the representation of $G^k$ obtained by composing the projection $G^k \to G$ to the $i$-th factor with $\sigma$. 
Note that for each $i$, there is a natural bijection between $\Sigma(G^k,\sigma_i) \to \Sigma(\sigma)$, $\chi \mapsto \tilde{\chi}$ such that the weight $\chi$ is the composition of the $i$-th projection with $\tilde{\chi} \in \Sigma(\sigma)$.

Let $G^k$ act on $\R[\Mat_d^k]_{\leq D}$ by right translation. Let $\rho \colon G^k \to \GL(\R[\Mat_d^k]_{\leq D})$ denote the corresponding representation. Then, $\rho$ is equivalent to
\[ \bigoplus_{j = 0}^D \Sym^j \bigl( \underbrace{\sigma_1 \oplus \dotsb \oplus \sigma_1}_{\text{$d$ times}} \oplus \dotsb \oplus \underbrace{\sigma_k \oplus \dotsb \oplus \sigma_k}_{\text{$d$ times}} \bigr).\]
It follows that any weight $\chi \in \Sigma(G^k,\rho)$ in $\rho$ is the sum of at most $D$ elements from $\bigcup_{i=1}^k \Sigma(G^k,\sigma_i)$.
In particular, 
\begin{align*}
\lambda_1(\mu^{\otimes k}, \R[\Mat_d^k]_{\leq D}) &= \max_{\chi \in \Sigma(G^k,\rho)} \chi( \vec\lambda(\mu^{\otimes k})) \\
&\leq D \max_{i} \max_{\chi \in \Sigma(G^k,\sigma_i)} \chi( \vec\lambda(\mu^{\otimes k}))\\
& = D \max_{i} \max_{\chi \in \Sigma(G^k,\sigma_i)} \tilde\chi( \vec\lambda(\mu))\\
&\leq D \lambda_1(\mu,\R^d).
\end{align*}
Since $G$ is not compact, $\lambda_1(\mu,\R^d)$ is positive by a result of Furstenberg~\cite{Furstenberg1963},
%(using a more modern language, this is a consequence of the Lyapunov vector being in the interior of the Weyl chamber)
and it follows that
\[
\lambda_1(\mu^{\otimes k}, \R[\Mat_d^k]_{< D}) < D \lambda_1(\mu,\R^d).
\]
The $G^k$-module $\R[\Gbb^k]_{< D}$ is a quotient of $\R[\Mat_d^k]_{< D}$, whence
\begin{equation}
\label{eq:l1of<Dpart}
\lambda_1(\mu^{\otimes k}, \R[\Gbb^k]_{< D}) < D \lambda_1(\mu,\R^d).
\end{equation}
Let $f_D\in \R[\Mat_d]$ denote the degree $D$ homogeneous part of $f$. Then $\bar{F}_D \in \R[\Mat_d^k]$ defined by
\[
\forall x_1, \dotsc, x_k \in \Mat_d,\;
\bar{F}_D( x_1, \dotsc, x_k ) = f_D(x_1 + \dots + x_k)
\]
is the degree $D$ homogeneous part of $\bar{F}$. Let $F_D \in \R[\Gbb^k]$ be the restriction of $\bar{F}_D$ to $\Gbb^k$. 
By \eqref{eq:l1of<Dpart} and the fact that $F - F_D \in \R[\Gbb^k]_{< D}$, we get
\[p(F) = p(F_D).\]
We may therefore assume that $f$ is homogeneous of degree $D$.

By Lemma~\ref{lm:limitset}, $f$ does not vanish on $G\pi_0 G + \dotsb + G\pi_0 G$ where% (see the proof of Lemma~\ref{lm:limitset})
\[\pi_0 = \lim_{n \to +\infty} \frac{\exp{na}}{\norm{\exp{na}}} \in \Mat_d(\R)\]
for some $a \in \afr^{++}$. 
Fix $g = (g_1, \dotsc, g_k) \in G^k$ and $h = (h_1,\dotsc, h_k) \in G^k$ such that
\[
f(g_1\pi_0 h_1 + \dotsb + g_k \pi_0 h_k) \neq 0.
\]
Writing $b = (a,\dotsc,a) \in \bfr$, we have, by the homogeneity of $\bar{F}$,
\[\lim_{n \to +\infty} \frac{F(g \exp(nb) h)}{\norm{\exp(na)}^D}
= \lim_{n \to +\infty} \bar{F} \Bigl( \frac{g \exp(nb) h}{\norm{\exp(na)}} \Bigr)
= f(g_1\pi_0 h_1 + \dotsb + g_k \pi_0 h_k) \neq 0,
\]
whence 
\[\abs{F(g \exp(nb) h)} \gg \norm{\exp(na)}^D.\]
In the rest of the proof, the implied constants in the Vinogradov notation are independent of $n$ but might depend on other quantities, like $g$ or $h$.
On the one hand,
\[\norm{\exp(na)} \gg e^{n \chi_\sigma(a)},\]
and on the other hand,
\[\abs{F(g \exp(nb) h)} = \bigl\lvert{\bigl(\rho(\exp(nb))\rho(h)F\bigr)(g))}\bigr\rvert \ll \norm{\rho(\exp(nb))\rho(h)F}.\]
Decompose $ \R[\Gbb^k]_{\leq D} = \bigoplus_j M_j$ into simple $G^k$-modules and decompose $\rho(h)F = \sum_j F_j$ accordingly.
Denote by $\chi_{M_j}$ the highest weight of $M_j$. Then
\[\norm{\rho(\exp(nb))\rho(h)F} \leq \sum_j e^{n\chi_{M_j}(b)} \norm{F_j}.\]
From the previous inequalities, there must exist $j$ such that $F_j \neq 0$ and
\[e^{D n \chi_\sigma(a)} \ll  e^{n\chi_{M_j}(b)} \norm{F_j}.\]
Hence $D n \chi_\sigma(a) \leq n \chi_{M_j}(b) + O(1)$ and necessarily
\[D \chi_\sigma(a) \leq \chi_{M_j}(b).\]
We have seen that $\chi_{M_j}$ is the sum of $D$ elements from $\bigcup_{i=1}^k \Sigma(G^k,\sigma_i)$:
\[\chi_{M_j} = \chi_1 + \dotsb + \chi_D.\]
Then 
\[\chi_{M_j}(b) = \chi_1(b) + \dotsb + \chi_D(b) = \tilde{\chi}_1(a) + \dotsb + \tilde{\chi}_D(a).\]
Thus we have simultaneously 
\[D \chi_\sigma(a) \leq \tilde{\chi}_1(a) + \dotsb + \tilde{\chi}_D(a)\]
 and 
\[\tilde{\chi}_1 + \dotsb + \tilde{\chi}_D \leq D \chi_\sigma \]
for the order over the set of weights.
Since $a \in \afr^{++}$, , this forces
\[\tilde{\chi}_1 = \dots = \tilde{\chi}_D = \chi_\sigma.\]
Therefore
\[\lambda_1(\mu^{\otimes k},M_j) = \chi_{M_j}(\vec\lambda(\mu^{\otimes k})) = (\tilde{\chi}_1 + \dotsb + \tilde{\chi}_D) (\vec\lambda(\mu)) = D \chi_\sigma (\vec\lambda(\mu)) = D \lambda_1(\mu,\R^d).\]
We conclude that $p(\rho(h)F) \neq 0$ and hence
\[p(F) = \rho(h)^{-1} p(\rho(h)F) \neq 0.\qedhere\]
%\comm{Cette démonstration est trop fastidieuse ; à reprendre ?}
\end{proof}

We can now easily deduce Proposition~\ref{pr:NCdet} from Proposition~\ref{pr:LDPpolynome} and Lemma~\ref{lm:criterionL}.

\begin{proof}[Proof of Proposition~\ref{pr:NCdet}]
Note that under our assumptions, $G$ cannot be compact.
%For otherwise, $\Gamma \subset G \cap \SL_d(\Z)$ would be a finite group and hence equal to its Zariski closure $\Gbb$. But $\Gbb$ is assumed to be Zariski connected. It follows that $\Gbb$ is trivial, this contradicts the assumption that $G$ acts irreducibly on $\R^d$.
Let $f \colon E \to \R$ be a polynomial map of degree $D$, and denote by $f_D$ its degree $D$ homogeneous part.
Define $F \in \R[\Gbb^{k}]$ by
\[\forall g_1, \dotsc, g_k \in G,\; F(g_1, \dotsc, g_k) = f(g_1 + \dotsb + g_k).\]
Let $p \colon \R[\Gbb^k]_{\leq D}\to\R[\Gbb^k]_{\leq D}$ be the projection to the sum of all simple submodules $M$ such that $\lambda_1(\mu^{\otimes k},M)=D\lambda_1(\mu,\R^d)$.
By Lemma~\ref{lm:criterionL}, $\norm{p(F)} = 0$ implies $\norm{f_D} = 0$, and since these two expressions define seminorms on the space of polynomial maps on $E$, it follows that
\begin{equation}\label{fdpf}
\norm{f_D} \ll \norm{p(F)}.
\end{equation}
By Proposition~\ref{pr:LDPpolynome} applied to the random walk on $G^k$ associated to the measure $\mu^{\otimes k}$, with $\lambda=D\lambda_1(\mu,\R^d)$ we get that for every $\omega>0$, there exists $c>0$ and $n_0\in\N$ such that
\[ \forall n \geq n_0,\quad
(\mu^{*n})^{\otimes k} \bigl(  \bigl\{ g \in G^k \mid \abs{F(g)} \leq e^{(\lambda - \omega)n} \norm{p(F)}\bigr\}  \bigr)  \leq e^{-c n}.\]
Together with \eqref{fdpf}, this proves what we want.
\end{proof}

\subsection{Fourier decay for random walks}
The relevant object here is the measure $\tilde{\mu}_n$, obtained from $\mu_n = \mu^{*n}$ after rescaling by a factor $e^{-\lambda_1 n}$.
This rescaling shrinks $\mu_n$ to a ball of subexponential size around $0$.
An important consequence of the results of this section and the previous one is the following theorem.

\begin{theorem}[Fourier decay for {$\tilde{\mu}_n$}]
\label{thm:decay}
Let $\mu$ be a probability measure on $\SL_d(\Z)$, $d\geq 2$.
Let $\Gamma$ denote the subsemigroup generated by $\mu$, $\Gbb$ the Zariski closure of $\Gamma$ in $\SL_d$ and $E$ the subalgebra of $\Mat_d(\R)$ generated by $\Gbb(\R)$.
Denote $\mu_n = \mu^{*n}$ and $\tilde{\mu}_n = (e^{-\lambda_1 n})_* \mu_n$ where $\lambda_1$ is the top Lyapunov exponent of $\mu$.
Assume that: 
\begin{enumerate}[label=(\alph*)]
\item  The measure $\mu$ has a finite exponential moment; 
\item  The action of $\Gamma$ on $\R^d$ is irreducible;
\item  The algebraic group $\Gbb$ is Zariski connected. % and semisimple.
\end{enumerate} 
Then there exists $\alpha_0 > 0$ such that for every $0 < \alpha < \alpha_0$, there exists $c_0 > 0$ and $n_0 \geq 1$ such that for every $n\geq n_0$,
\[ \forall \xi\in E^*\ \mbox{with}\ e^{\alpha n}\leq \norm{\xi} \leq e^{\alpha_0 n},\quad
|\widehat{\tilde{\mu}_n}(\xi)| \leq e^{-c_0 n}.\]
\end{theorem}

\begin{proof}
We want to apply Theorem~\ref{thm:fourier} to the measure $\tilde{\mu}_n$, and for that, we should check that it is not concentrated near any affine subspace, nor near any translate of the set of non-invertible elements of $E$.
This will follow from Propositions~\ref{pr:NCaffine} and \ref{pr:NCdet}.
Recall that given $\rho>0$, we write $S_E(\rho)$ for the set
\[ S_E(\rho) = \{\, x\in E \mid \abs{\det\nolimits_E(x)} \leq \rho\,\}.\]
Let $D = \dim(E)$. 
Under the assumptions of the theorem, we claim that there exists $\kappa > 0$ depending only $\mu$ such that for every $\omega > 0$, there exists $c =c(\mu, \omega) > 0$ such that for every $n \geq 1$, we can decompose the convolution
\[\tilde\mu_n^{\pp D} \mm \tilde\mu_n^{\pp D} = \eta + \theta\]
into positive Borel measures satisfying the following properties.
\begin{enumerate}
\item $\theta(E) \ll_\mu e^{-cn}$,
\item $\eta(E \setminus B(0, e^{\omega n})) \ll_\mu e^{-cn}$,
\item $\forall x \in E,\ \eta(x + S_E(e^{-\omega n})) \ll_\mu e^{-cn}$,
\item $\forall \rho \geq e^{-n},\ \forall W < E\ \mbox{affine subspace},\ \eta(W^{(\rho)}) \ll_\mu e^{\omega n} \rho^\kappa$.
\end{enumerate}
To justify this claim, let $\eta_1$ be the restriction of $\tilde{\mu}_n$ to $E \setminus B(0, e^{-\omega n})$ and put 
\[
\eta' = \eta_1 \pp \tilde{\mu}_n^{\pp( D - 1)}\quad \text{and} \quad \eta = \eta' \mm \tilde\mu_n^{\pp D}.
\] 
By Theorem~\ref{thm:LargeD}\ref{it:LargeDn}, there is $c = c(\mu,\omega) > 0$ such that for every $n \geq 1$,
\[\tilde{\mu}_n(B(0,e^{-\omega n})) \ll_\mu e^{-cn}
\quad \text{and} \quad 
\tilde{\mu}_n(E \setminus B(0,e^{\omega n})) \ll_\mu e^{-cn}.\]
It follows that
\[\theta(E) \ll_\mu e^{-cn}\]
and
\[
 \eta \bigl( E \setminus B(0, 2D e^{\omega n})\bigr) \ll_\mu e^{-cn}.
\]
For $x \in E$, apply Proposition~\ref{pr:NCdet} to the polynomial function $y \mapsto \det_E(y - e^{\lambda_1 m}x)$. Note that these polynomials all have degree $D = \dim(E)$ and all have the same degree $D$ homogeneous part, namely $\det_E$. We obtain $c > 0$ such that 
\begin{align*}
\eta'\bigl(x + S_E(e^{-\omega n})\bigr) &\leq \tilde\mu_n^{\pp D} \bigl(x + S_E(e^{-\omega n})\bigr)\\
&= \mu_n^{\pp D} \bigl( \bigl\{\, y \in E \mid \abs{\det\nolimits_E(e^{-\lambda_1 n}y- x)} \leq e^{-\omega n} \,\bigr\} \bigr)\\
&= \mu_n^{\pp D} \bigl( \bigl\{\, y \in E \mid \abs{\det\nolimits_E(y- e^{\lambda_1 n}x)} \leq e^{(D \lambda_1 -\omega) m} \,\bigr\} \bigr)\\
&\ll_\mu e^{-cn}.
\end{align*}
Since this property is preserved under additive convolution, the same holds for $\eta$.
Now let $W$ be a proper affine subspace of $E$.
Using the definition of $\eta_1$ and Proposition~\ref{pr:NCaffine}, we find for every $\rho \geq e^{-n}$,
\begin{align*}
\eta_1 (\{\, g \in E \mid d(g,W) \leq \rho \,\})
& \leq \tilde\mu_n \bigl( \bigl\{\, g\in E \mid d(g,W) \leq \rho \text{ and } \norm{g} \geq e^{-\omega n} \,\bigr\} \bigr)\\
&\leq \tilde\mu_n \bigl( \bigl\{ g\in E \mid d(g,W) \leq \rho e^{\omega n} \norm{g} \bigr\} \bigr)\\
&= \mu_n \bigl( \bigl\{\, g\in E \mid d(g, e^{\lambda_1 n}W) \leq \rho e^{\omega n} \norm{g} \,\bigr\} \bigr)\\
&\ll_\mu e^{\omega \kappa n} \rho^\kappa.
\end{align*}
Again this property is preserved under additive convolution, so that $\eta$ satisfies the required conditions.

Let $\eps = \eps(\mu,\kappa) > 0$ and $s = s(\mu,\kappa) \geq 1$ be the constant given by Theorem~\ref{thm:fourier} applied with the parameter $\kappa$.
Set $\omega = \alpha \eps / 2$. Recall that we chose $c = c(\mu,\omega)$. Finally set $\tau = \min\{c/2,\eps \kappa \}$.
With the choice of these parameters, for every $n$ large enough, for every $R \in {[e^{\alpha n }, e^{n}]}$,
\begin{enumerate}
\item $\theta(E) \leq R^{-\tau}$,
\item $\eta(E \setminus B_E(0, R^\eps)) \leq R^{-\tau}$,
\item for all $x \in E$, $\eta(x + S_E(R^{-\eps})) \leq R^{-\tau}$,
\item for all $\rho \geq R^{-1}$ and every proper affine subspace $W \subset E$, $\eta(W^{(\rho)}) \leq R^{\eps} \rho^\kappa$.
\end{enumerate}
In other words, the assumptions of Theorem~\ref{thm:fourier} are satisfied for the measure $\eta$ at the scale $\frac{1}{R}$. Therefore, for all $\xi \in  E^*$ in the range 
\begin{equation*}
e^{\alpha n} \leq \norm{\xi} \leq e^{n},
\end{equation*}
we have
\begin{equation*}
\abs{\widehat{\eta^{* s}}(\xi)} \leq \norm{\xi}^{-\eps \tau}.
\end{equation*}
It follows that
\[ \abs{\widehat{(\tilde\mu_n^{\pp D} \mm \tilde\mu_n^{\pp D} )^{* s}}(\xi)} \leq \norm{\xi}^{-\eps \tau} + O_s(\norm{\xi}^{-\tau}).\]
Applying Lemma~\ref{lm:ordre} to $\tilde\mu_n^{\pp D} \mm \tilde\mu_n^{\pp D}$, we get
\[ 
\abs{\widehat{\tilde\mu_{sn}}(\xi)} \leq \norm{\xi}^{-c_0},
\]
where $c_0 = \frac{\eps \tau}{ 2(2D)^s}$ and, again, assuming that $n \geq n_0(\mu,\alpha)$.
This shows the desired upper bound for $\widehat{\tilde{\mu}_n}(\xi)$ provided that $n$ is a multiple of $s$.

To prove the estimate for general $n$, write $n=sq+r$ with $0\leq r<s$. On the one hand, by~\eqref{eq:hatmuffnu},
\[ \widehat{\tilde{\mu}_n}(\xi) = \int \widehat{\tilde{\mu}_{sq}}(\xi x)\dd\tilde{\mu}_r(x),\]
On the other hand, it follows from the Markov inequality and the fact that $\tilde{\mu}_r$ has bounded exponential moment that, for $x$ outside of a set of exponentially small $\tilde{\mu}_r$-measure,
\[ e^{-\frac{\alpha n}{2}}\norm{\xi} \leq \norm{\xi x} \leq e^{\frac{n}{4s}}\norm{\xi}.\]
This proves the theorem with $\alpha_0 = \frac{1}{4s}$.
\end{proof}

\section{The set of large Fourier coefficients}
\label{sc:fourier}

Starting from Theorem~\ref{thm:decay}, we now use some Fourier analysis to show a first intermediate statement towards Theorem~\ref{thm:main}.

\begin{proposition}[First step: concentration and separation]
\label{pr:step1}
Let $\mu$ be a probability measure on $\SL_d(\Z)$, $d\geq 2$.
Denote by $\Gamma$ the subsemigroup generated by $\mu$, and by $\Gbb$ the Zariski closure of $\Gamma$ in $\SL_d$.
Assume that: 
\begin{enumerate}[label=(\alph*)]
\item  The measure $\mu$ has a finite exponential moment; 
\item  The action of $\Gamma$ on $\R^d$ is irreducible;
\item  The algebraic group $\Gbb$ is Zariski connected. % and semisimple.
\end{enumerate} 
There exist constants $C \geq 0$ and $\sigma > \tau>0$ such that the following holds.\\
Let $\nu$ be a Borel probability measure on $\Tbb^d$.
Let $t_0 \in {(0, 1/2)}$.
Assume that  for some $a_0 \in \Z^d \setminus\{0\}$,
\[\abs{\widehat{\mu^{* n} * \nu}(a_0)} \geq t_0 \quad\text{and}\quad n \geq C \abs{\log t_0}.\]
Then, writing 
$N = e^{\sigma n} \norm{a_0}$ and $M = e^{-\tau n}N$,
there exists a $\frac{1}{M}$-separated set $X\subset\Tbb^d$ such that
\[ \nu\left(\bigcup_{x\in X} B(x,\frac{1}{N})\right)
\geq t_0^C.\]
\end{proposition}

The proof of this statement goes by two steps.
First, following \cite{BFLM}, one applies a Fourier analytic lemma \cite[Proposition~7.5]{BFLM} to translate the concentration of $\nu$ to a statement about its Fourier coefficients.
Then, one uses the Fourier decay of $\tilde{\mu}_n = (e^{-\lambda_1 n})_*\mu^{*n}$ to study the set of large Fourier coefficients
\begin{equation}\label{at}
A(t) = \{a \in \Z^d \mid \abs{\hat\nu(a)} \geq t\},
\end{equation}
and prove the desired statement.

\subsection{Detecting concentration from the Fourier coefficients}

Because it is so elementary, and yet beautiful, we include the Fourier analytic lemma needed for our argument.
The reader is referred to \cite[Proposition~7.5]{BFLM} for its ingenious proof.

\begin{lemma}
Given $d\in\N$,
there exists $c>0$ such that if a measure $\nu$ on $\Tbb^d$ satisfies 
\[
\Ncal(A(t)\cap B(0,N), M) \geq s \left(\frac{N}{M}\right)^d
\]
for some numbers $s,t>0$ and some $M,N\geq 1$ such that $M<cN$, then there exists a $\frac{1}{M}$-separated subset $X\subset\Tbb^d$ such that
\[
\nu\left(\bigcup_{x\in X} B(x,\frac{1}{N})\right) 
\geq c (st)^3.
\]
\end{lemma}

Going back to the statement of Proposition~\ref{pr:step1} above we see that it is enough to show that, under the same assumptions, there exist $C\geq 0$ and $\sigma > \tau>0$ such that, for $N=e^{\sigma n}
\norm{a_0}$ and $M=e^{-\tau n}N$,
\begin{equation}\label{large}
\Ncal\bigl(A(t_0^C) \cap B(0,N) , M\bigr) \geq  t_0^C \Bigl(\frac{N}{M}\Bigr)^d.
\end{equation}
This is the goal of the next paragraph.

\subsection{Fourier decay and large coefficients}

For $a \in \Z^d$ and $x \in \Tbb^d$, we denote by $(a,x) \mapsto \langle a , x \rangle \in \Tbb$ the natural pairing.  
Vectors in $\Z^d = \widehat{\Tbb^d}$ indexing Fourier coefficients are naturally understood as row vectors, so that for any $g \in \SL_d(\Z)$, we have
\[\langle a , gx \rangle = \langle ag , x \rangle.\]

Before we start the proof of \eqref{large}, we record an elementary lemma -- not much more than the Cauchy-Schwarz inequality -- which shows that the set of large Fourier coefficients of a measure has some additive structure.
It will later be combined with the multiplicative properties of $\mu^{*n}$, allowing us to exploit the sum-product phenomenon for the study of the set of large Fourier coefficients.
This approach to Fourier coefficients of multiplicative convolutions of measures goes back to the work of Bourgain and Konyagin \cite{bourgainkonyagin} on exponential sums in finite fields.

We use the symbols $\pp$ and $\mm$ introduced in Section~\ref{sc:sumprod}.

\begin{lemma}[Additive structure of Fourier coefficients]
\label{lm:SpecHolder}
Let $\mu$ be a Borel probability measure on $\SL_d(\Z)$ and $\nu$ a Borel probability measure on $\Tbb^d$.
If 
\[ \abs{\widehat{\mu * \nu}(a_0)} \geq t_0 > 0,\]
then for any integer $k \geq 1$, the set
\[A = \bigl\{ g \in \Mat_d(\Z) \mid \abs{\hat{\nu}(a_0g)} \geq t_0^{2k}/2 \bigr\}\]
satisfies
\[\bigl(\mu^{\pp k} \mm \mu^{\pp k}\bigr) (A) \geq \frac{t_0^{2k}}{2}.\]
\end{lemma}
\begin{proof}
Observe that
\[\widehat{\mu * \nu}(a_0) =  \int_{\Tbb^d}\int_\Gamma e(\langle a_0 , gx \rangle ) \dd \mu (g) \dd \nu(x)  = \int_{\Tbb^d}\int_\Gamma e(\langle a_0g , x \rangle ) \dd \mu (g) \dd \nu(x).\] 
By Hölder's inequality,
\begin{align*}
t_0^{2k} \leq \abs{\widehat{\mu * \nu}(a_0)}^{2k} &\leq  \int_{\Tbb^d} \bigl\lvert \int_\Gamma e(\langle a_0g , x \rangle ) \dd \mu (g) \bigr\rvert^{2k} \dd \nu(x)\\
& = \int_{\Gamma^{2k}} \hat\nu\bigl(a_0 (g_1 + \dotsb + g_k - g_{k+1} - \dotsb - g_{2k}) \bigr) \dd \mu^{\otimes 2k} (g_1, \dotsc, g_{2k})\\
& \leq \int_E \abs{\hat{\nu}(a_0g)}  \dd \bigl(\mu^{\pp k} \mm \mu^{\pp k}\bigr) (g)\\
& \leq \bigl(\mu^{\pp k} \mm \mu^{\pp k}\bigr) (A) + \frac{t_0^{2k}}{2} \bigl(\mu^{\pp k} \mm \mu^{\pp k}\bigr)( E \setminus A),
\end{align*}
which finishes the proof of the lemma.
\end{proof}

The next lemma is a regularity statement we need for measures that have a strong Fourier decay.
It essentially states that if a set in $\R^D$ carries a large proportion of a measure with small Fourier coefficients at all frequencies between $\delta^{-\alpha}$ and $\delta^{-1-\alpha}$, then we can find a ball of radius $\delta^{O(\alpha)}$ in the set on which the measure is comparable to the Lebesgue measure at scale $\delta$.

\begin{lemma}[Regularity from Fourier decay]
\label{lm:dgammafull}
Given $D \geq 1$ and $\alpha > 0$, there exist constants $c = c(D,\alpha) > 0$ and $C_1 = C_1(D,\alpha) > 0$ such that the following holds for all $0 < \delta < c t$.
Let $\mu$ be a Borel measure on $\R^D$, of total mass $\mu(\R^D) \leq 1$. 
Let $A$ be a subset of $\R^D$. Assume
\begin{enumerate}
\item $\Supp(\mu) \subset B(0, \delta^{-\alpha})$,
%\marginpar{\small dans la condition (ii), peut on remplacer $\delta^{-1-\alpha}$ par $\delta^{-1}$ ?}
\item for all $\xi \in \R^D$ with $\delta^{-\alpha} \leq \norm{\xi} \leq \delta^{-1- \alpha}$, $\abs{\hat\mu(\xi)} \leq \norm{\xi}^{-C_1}$,
\item $\mu(A) \geq t$.
\end{enumerate}
Then there exists $x \in \R^D$ such that 
\[\Ncal(A \cap B(x, \delta^\beta), \delta) \geq c t^{D + 1} \Bigl( \frac{\delta^\beta}{\delta} \Bigr)^D,\]
where $\beta = (2D + 1) \alpha $.
\end{lemma}

\begin{proof}
In this proof, the implied constants in the Vinogradov or Landau notation are allowed to depend on the dimension $D$ in addition to the dependency indicated by subscripts.

Let $\phi \colon \R^D \to \R$ be a nonnegative smooth function supported on $B(0,1)$ such that $\int_{\R^D} \phi = 1$.
Set $\phi_\delta(x) = \delta^{-D}\phi(\delta^{-1} x)$, $\forall x \in \R^D$. Note that
\[\forall \xi \in \R^D,\quad  \widehat{\phi_\delta}(\xi) = \hat\phi(\delta \xi).\]
Since $\phi$ is smooth, for any $C_2 > 0$, we have
\begin{equation}
\label{eq:hatphi}
\forall \xi \in \R^D,\quad \abs{\hat\phi(\xi)} \ll_{C_2} (1 + \norm{\xi})^{-C_2}.
\end{equation}

Define $\mu_\delta = \mu \pp \phi_\delta$, viewed either as a measure or as a smooth function on $\R^D$.
%the density of the measure
Clearly,
\[\mu_\delta(A^{(\delta)}) \geq \mu(A) \geq t.\]
Let $c > 0$ be a small constant depending on $D$ and $\alpha$ to be determined later.
Assume $\delta < ct$ and set $\rho = c \delta^\beta t$.
Let $(B_i)_{1\leq i \leq i_{\max}}$ be an essentially disjoint covering of $B(0,\delta^{-\alpha})$ by closed balls of radius $\rho$.
In other words, the intersection multiplicity of the covering is at most $C = O(1)$, so that in particular the number of balls is at most $i_{\max} = O(\delta^{-D \alpha} \rho^{-D})$.
Consider 
\[I = \Bigl\{ 1\leq i \leq i_{\max} \mid \frac{\mu_\delta(A^{(\delta)} \cap B_i)}{\mu_\delta(B_i)} \geq \frac{t}{2C} \Bigr\}.\]
%We have
%\begin{align*}
%t \leq \mu_\delta(A^{(\delta)})  &\leq \sum_{i=1}^{i_{\max}} \frac{\mu_\delta(A^{(\delta)} \cap B_i)}{\mu_\delta(B_i)} \mu_\delta(B_i) \\
%& \leq \sum_{i \in I}\mu_\delta(B_i) +  \frac{t}{2C} \sum_{i \notin I} \mu_\delta(B_i) \\
%& \leq \sum_{i \in I}\mu_\delta(B_i) +  \frac{t}{2}.
%\end{align*}
%Therefore, $\sum_{i \in I}\mu_\delta(B_i) \gg t$. 
Using the finite multiplicity of the covering, we infer that $\sum_{i \in I}\mu_\delta(B_i) \geq t/2$.
Hence there exists $i \in I$ such that 
\[\mu_\delta(B_i) \gg \frac{t}{i_{\max}} \gg t \delta^{D\alpha}\rho^{D}.\]

We fix this $i$ from now on.
Define
\[M = \max_{x \in B_i} \mu_\delta(x) \quad \text{ and } \quad m = \min_{x \in B_i} \mu_\delta(x).\]
We have $\mu_\delta(B_i) \leq M \abs{B_i}$ and hence 
\begin{equation}
\label{eq:MforMax} 
M \gg t \delta^{D\alpha}.
\end{equation}
Let $x_0 \in B_i$ such that $\mu_\delta(x_0) = M$.
By the Plancherel theorem, for any $x \in B_i$,
\begin{equation*}
\mu_\delta(x) = \mu \pp \phi_\delta (x) = \int_{\R^D} e(- \langle\xi , x\rangle) \hat\mu(\xi) \widehat{\phi_\delta}(\xi) \dd \xi
\end{equation*}
Thus,
\begin{align*}
\abs{\mu_\delta(x)- \mu_\delta(x_0)} &\leq \int_{\R^D} \abs{1 - e( \langle\xi , x - x_0\rangle)} \abs{\hat\mu(\xi)} \abs{ \widehat{\phi_\delta}(\xi)} \dd \xi \\
&\leq T_1 + 2 T_2 + 2 T_3,
\end{align*}
where
\begin{align*}
T_1 &= \int_{\norm{\xi} \leq \delta^{-\alpha}} \abs{1 - e( \langle\xi , x - x_0\rangle)} \dd \xi \\
&\ll \int_{\norm{\xi} \leq \delta^{-\alpha}} \norm{\xi} \norm{x - x_0} \dd \xi 
\ll \delta^{-(D + 1)\alpha} \rho,\\
T_2 &= \int_{\delta^{-\alpha} \leq \norm{\xi} \leq \delta^{-1 - \alpha}} \abs{\hat\mu(\xi)} \dd \xi \\
&\ll \int_{\delta^{-\alpha} \leq \norm{\xi} \leq \delta^{-1 - \alpha}} \norm{\xi}^{-C_1} \dd \xi
\ll \frac{\delta^{(C_1 - D)\alpha}}{C_1 - D},\\
T_3 &= \int_{\norm{\xi} \geq \delta^{-1 - \alpha}} \abs{ \hat{\phi}(\delta\xi)} \dd \xi\\
&\ll_{C_2} \int_{\norm{\xi} \geq \delta^{-1 - \alpha}} \delta^{-C_2} \norm{\xi}^{-C_2} \dd \xi
\ll_{C_2} \frac{\delta^{(C_2 - D)\alpha - D}}{C_2 - D}.
\end{align*}
In the last line, we used \eqref{eq:hatphi}.

Picking $\beta = (2D + 1) \alpha$, $C_1 = \frac{D \alpha + 1}{\alpha} + D$ and $C_2 = \frac{D \alpha + D + 1}{\alpha} + D$ and putting these inequalities together, we obtain, remembering \eqref{eq:MforMax} and $\delta < ct$, 
\[M - m \ll_{\alpha} c t \delta^{D\alpha} + \delta^{D\alpha + 1} \ll_{\alpha} c M.\]
This implies $M/m \leq 2$ provided that $c$ is chosen small enough according to $D$ and $\alpha$.
Remembering $i \in I$, we have

\[
t \ll \frac{\mu_\delta(A^{(\delta)} \cap B_i)}{\mu_\delta(B_i)} \leq \frac{M}{m} \frac{\abs{A^{(\delta)} \cap B_i}}{\abs{B_i}}.
\]
Hence 
\[\Ncal(A \cap B_i, \delta) \gg \delta^{-D} \abs{A^{(\delta)} \cap B_i} \gg t \delta^{-D}  \abs{B_i} \gg t \delta^{-D} \rho^D = c^D t^{D+1} \delta^{D\beta-D}.\]
Let $x \in \R^D$ be the center of $B_i$. Then $A \cap B_i \subset A \cap B(x,\delta^\beta)$ and hence
\[\Ncal(A \cap B(x,\delta^\beta), \delta) \gg_{\alpha} t^{D+1} \delta^{D\beta-D}. \qedhere\] 
\end{proof}

Combining the above observations and Theorem~\ref{thm:decay}, we can derive \eqref{large}.

\begin{proof}[Proof of \eqref{large}]
In this proof, the implied constants in the Vinogradov or Landau notation are allowed to depend on $\mu$ and hence on $E$.
As before,  for $n \geq 1$, we let $\tilde{\mu}_n = (e^{- \lambda_1 n})_* \mu_n$ denote the rescaling of $\mu_n = \mu^{*n}$. Denote by $E$ the subalgebra of $\Mat_d(\R)$ generated by $\Gbb(\R)$ and write $D=\dim(E)$.
By Theorem~\ref{thm:decay} there exists constants $\alpha_0 > 0$ and $c_0 > 0$ such that
\[
\forall \xi \in E^* \text{ with } e^{\frac{\alpha_0 n}{8D + 4}} \leq \norm{\xi} \leq e^{\alpha_0 n},\quad \abs{\widehat{\tilde\mu_{n}}(\xi)} \leq \norm{\xi}^{-c_0},
\]
Now fix $\delta = e^{-\frac{\alpha_0 n}{2}}$ and write $\alpha=1/(4D+2)$.
Let $C_1 = C_1(D,\alpha)$ from Lemma~\ref{lm:dgammafull} and set $k = \lceil C_1/c_0 \rceil$ so that the above implies
\[
\forall \xi \in E^* \text{ with }\delta^{-\alpha} \leq \norm{\xi} \leq \delta^{-2},\quad \abs{\widehat{\tilde\mu_{n}^{\pp k} \mm \tilde\mu_{n}^{\pp k}}(\xi)} \leq \norm{\xi}^{-C_1},
\]
This says that the measure $\tilde\mu_{n}^{\pp k} \mm \tilde\mu_{n}^{\pp k}$ is regular at all scales between $\delta^2$ and $\delta^\alpha$.

\smallskip

On the other hand, since $\abs{\widehat{\mu_{n} * \nu}(a_0)} \geq t_0$, it follows from Lemma~\ref{lm:SpecHolder} that
the set
\[A = \bigl\{ g \in E \cap \Mat_d(\Z) \mid \abs{\hat{\nu}(a_0g)} \geq t_1 := t_0^{2k}/2\bigr\}\]
satisfies
\[\bigl(\mu_{n}^{\pp k} \mm \mu_{n}^{\pp k} \bigr) (A) \gg t_0^{2k}.\]
Letting $\tilde A = e^{- \lambda_1 n } \cdot A$ be the rescaling of $A$, we find
\[\bigl(\tilde\mu_{n}^{\pp k} \mm \tilde\mu_{n}^{\pp k} \bigr) (\tilde A) \gg t_0^{2k}.\]
From the large deviation estimate Theorem~\ref{thm:LargeD}\ref{it:LargeDn}, we also have
\[\bigl(\tilde\mu_{n}^{\pp k} \mm \tilde\mu_{n}^{\pp k} \bigr) (E \setminus B(0,\delta^{-\alpha})) \ll e^{-c_2 n}\]
for some $c_2 = c_2(\mu) > 0$.
Assuming $n \geq \frac{4k}{c_2}\abs{\log t_0}$, this implies
\[\bigl(\tilde\mu_{n}^{\pp k} \mm \tilde\mu_{n}^{\pp k} \bigr) (\tilde A \cap B(0,\delta^{-\alpha})) \gg  t_0^{2k}.\]

So we can apply Lemma~\ref{lm:dgammafull} to the restriction of $\tilde\mu_{n}^{\pp k} \mm \tilde\mu_{n}^{\pp k}$ to $B(0,\delta^{-\alpha})$.
Letting $t_1=t_0^{2k}$, we obtain $x \in B(0,\delta^{-\alpha})$ such that
\[\Ncal( \tilde A \cap B(x, \delta^{1/2}), \delta) \gg t_1^{D+1} \delta^{-D/2}.\]
Rescaling back, we find
\begin{equation}
\label{eq:AcapBx}
\Ncal\bigl( A \cap B(e^{\lambda_1 n} x, N_0), M_0 \bigr) \gg t_2 \Bigl(\frac{N_0}{M_0}\Bigr)^D,
\end{equation}
where $t_2 = t_1^{D+1}$, $N_0 = e^{\sigma n}$ and $M_0 = e^{-\tau n} N_0$ with $\sigma = \lambda_1 - \frac{\alpha_0}{4}$ and $\tau = \frac{\alpha_0}{4}$.
In accordance with the statement of the proposition, we put
 \[ N= N_0\norm{a_0}
\quad\text{and}\quad
M = M_0\norm{a_0}.\]

\smallskip

Consider the map $\phi_0 \colon E \to \R^d$, $g \mapsto a_0g$.
Letting $A' = A \cap B(e^{\lambda_1 n} x, N_0)$, we have
\begin{equation}
\label{eq:A'fiber}
\Ncal(A', M_0) \leq \Ncal\bigl(\phi_0(A'), M \bigr) \max_{b \in \phi_0(A')} \Ncal\bigl(A' \cap \phi_0^{-1}(B(b,M)) ,M_0\bigr).
\end{equation}
We claim that
\begin{equation}
\label{eq:A'capW}
\max_{b \in \phi_0(A')} \Ncal\bigl(A' \cap \phi_0^{-1}(B(b,M)) ,M_0\bigr) \ll \Bigl(\frac{N_0}{M_0}\Bigr)^{D - d}
\end{equation}
Evidently $\norm{\phi_0} \asymp_E \norm{a_0}$. 
Let $W_0 = \ker \phi_0$. Since $G$ acts irreducibly on $\R^d$, $\phi_0$ is surjective and hence $\dim(W_0) = D - d$. 
The restriction $\phi_{0\mid W_0^\perp} \colon W_0^\perp \to \R^d$ is bijective. Moreover, by a compactness argument, 
%$\norm{a_0}(\phi_{a_0}|_{W_0^\perp})^{-1}$ est continue sur la sphère unité de $(\R^d)$, décrite par $a_0/\norm{a_0}$.
\[\norm{\phi_{0\mid W_0^\perp}^{-1}} \asymp_E \norm{a_0}^{-1}.\]
Consequently, for any $y \in E$,
\[ \phi_0^{-1}(B(\phi_0(y),M))  \subset y + W_0^{(O(M_0))}.\]
Hence 
\[\Ncal\bigl(A' \cap \phi_0^{-1}(B(\phi_0(y),M)) ,M_0\bigr) \leq \Ncal\bigl(B(0,N_0) \cap W_0^{(O(M_0))} ,M_0\bigr) \ll \Bigl(\frac{N_0}{M_0}\Bigr)^{D - d},\]
which proves the claim \eqref{eq:A'capW}.
From \eqref{eq:AcapBx}, \eqref{eq:A'fiber} and \eqref{eq:A'capW}, we get
\[\Ncal\bigl(\phi_0(A') , M \bigr) \gg t_2 \Bigl(\frac{N}{M}\Bigr)^{d}.\]
By definition of $A'$, we have $\phi_0(A') \subset A(t_1) \cap B(b, \norm{\phi_0}N_0)$, where $b = e^{\lambda_1 n} a_0x$ and $\norm{\phi_0}N_0 \ll N$. 

This is almost what we want, except that the ball $B(b,N)$ is not centered at the origin.
To recenter that ball, we make use once more of the additive properties of the set of large Fourier coefficients.
Choose $b' \in B(b,N)$ to be such that the quantity $\Ncal\bigl(A(t_1) \cap B(b', \frac{N}{2}) , M \bigr)$ is maximal. We get
\[\Ncal\bigl(A(t_1) \cap B(b', \frac{N}{2}) , M \bigr) \gg t_2 \Bigl(\frac{N}{M}\Bigr)^{d}.\]
Choose an $M$-separated subset $A_1 \subset A(t_1) \cap B(b', N/2)$ of cardinality 
\[\abs{A_1}\gg \Ncal\bigl(A(t_1) \cap B(b', \frac{N}{2}) , M \bigr)\]
and such that all Fourier coefficients $\hat\nu(a)$, for $a\in A_1$, fall into the same quadrant of $\C$.
Then
\[ \abs{A_1} \frac{t_1}{4} \leq \bigl\lvert \sum_{a \in A_1} \hat\nu(a) \bigr\rvert = \Bigl\lvert \int_{\Tbb^d} \sum_{a \in A_1} e(\langle a, x\rangle) \dd \nu(x) \Bigr\rvert  \]
By the Cauchy-Schwarz inequality, 
\begin{equation*}
\abs{A_1}^2 \frac{t_1^2}{16} \leq \int_{\Tbb^d} \bigl\lvert \sum_{a \in A_1} e(\langle a, x\rangle)\bigr\rvert^2  \dd \nu(x) = \sum_{a_1,a_2 \in A_1} \abs{\hat\nu(a_1 - a_2)}
\end{equation*}
Thus, there exists $a_2 \in A_1$ such that
\begin{equation*}
\abs{A_1} \frac{t_1^2}{16} \leq \sum_{a_1 \in A_1} \abs{\hat\nu(a_1 - a_2)}.
\end{equation*}
Set $A_2 = (A_1 - a_2) \cap A\bigl(\frac{t_1^2}{32}\bigr)$, we have $\abs{A_2} \geq \frac{t_1^2}{32} \abs{A_1}$ and $A_2 \subset B(0,N)$.
It follows that
\[\Ncal\Bigl(A \Bigl(\frac{t_1^2}{32}\Bigr) \cap B(0,N), M\Bigr) \gg t_1^2 t_2 \Bigl(\frac{N}{M}\Bigr)^{d}.\]
This concludes our proof.
\end{proof}

\section{Concentration near rational points}
\label{sc:rational}

In this section, we finish the proof of Theorem~\ref{thm:main} from the introduction.
We shall in fact prove a slightly more general statement, given as Proposition~\ref{pr:general} below.
%Recall that for parameters $Q\geq 1$ and $\rho>0$, we write $\sW_Q$ for the set of rational points on $\Tbb^d$ with denominator at most $Q$, and $\sW_Q^{(\rho)}$ for its $\rho$-neighborhood.

In all this section, unless stated otherwise, $\mu$ denotes a probability measure on $\SL_d(\Z)$, $d\geq 2$.
The Lyapunov exponents of $\mu$ are denoted by $\lambda_1 \geq \dots \geq \lambda_d$.
The subsemigroup generated by $\mu$ is denoted by $\Gamma$, its Zariski closure by $\Gbb$, and we write $G=\Gbb(\R)$ for the set of real points.
We assume that:
\begin{enumerate}[label=(\alph*)]
\item  The measure $\mu$ has a finite exponential moment; 
\item  The action of $\Gamma$ on $\R^d$ is irreducible;
\item  The algebraic group $\Gbb$ is Zariski connected.
\end{enumerate} 
We also let $E$ be the subalgebra of $\Mat_d(\R)$ generated by $G$.
For $n\in\N$, we write $\mu_n=\mu^{*n}$ for the law of the random walk in $G$ at time $n$.
Moreover, $\nu$ denotes a Borel probability measure $\nu$ on $\Tbb^d$, understood as the starting distribution of a random walk on $\Tbb^d$.
We write $\nu_n = \mu_n * \nu$ for the law of the random walk at time $n$.
Finally, given $Q\geq 1$ and $\rho>0$, we denote by $\sW_Q$ the set of rational points in $\Tbb^d$ with denominator at most $Q$, and by $\sW_Q^{(\rho)}$ its $\rho$-neighborhood.

\begin{proposition}
\label{pr:general}
Let $d \geq 2$. Let $\mu$ be a probability measure on $\SL_d(\Z)$.
%Denote by $\Gamma$ the subsemigroup generated by $\mu$, and by $\Gbb < \SL_d$ the Zariski closure of $\Gamma$.
%Assume 
%\begin{enumerate}[label=(\alph*)]
%\item  The measure $\mu$ has a finite exponential moment; 
%\item  The only subspaces of $\R^d$ preserved by $\Gamma$ are $\{0\}$ and $\R^d$;
%\item  The algebraic group $\Gbb$ is Zariski connected. % and semisimple.
%\end{enumerate} 
Under the assumptions above, the following holds.

%Let $\lambda_1$ denote the top Lyapunov exponent associated to $\mu$.
Given $\lambda \in {(0, \lambda_1)}$, there exists a constant $C = C(\mu, \lambda) > 0$ such that for every Borel probability measure $\nu$ on $\Tbb^d$ and every $t \in {(0,1/2)}$, 
if for some $a \in \Z^d \setminus \{0\}$, 
\[\abs{\widehat{\nu_{n}}(a)} \geq t \quad \text{and} \quad n \geq C \log\frac{\norm{a}}{t},\]
then
\[ \nu \bigl( \sW_{Q}^{(e^{-\lambda n})}\bigr) \geq t^C\]
for some $Q \leq \bigl(\frac{\norm{a}}{t}\bigr)^C$.
\end{proposition}

The argument follows closely the one given in Section~7 of \cite{BFLM}, but some modifications are required since we cannot make use of the proximality assumption.
We shall divide the proof of Theorem~\ref{thm:main} into three parts.
First, one observes that given a Borel probability measure $\nu$ on $\Tbb^d$, the sequence of measures $\nu_n$ satisfies a diophantine property: if it gives much weight to a ball of small radius, then the ball must contain a rational point with small denominator.
Second, starting the separated set $X$ around which, by Proposition~\ref{pr:step1}, $\nu_n$ is concentrated, one goes backwards along the random walk in order to increase the concentration of the measure around the set $X$, until one can apply the diophantine property to conclude that $\nu_{n-m}$ is concentrated near some rational points with bounded denominator.
The last part, concluding the proof, is again going backwards along the random walk, to show that if $\nu_n$ concentrates near the set of rational points of bounded height, then $\nu$ is even more concentrated near that set.

\subsection{An almost diophantine property}

The key to obtain the concentration near \emph{rational} points is the following almost diophantine property of the sequence of measures $\nu_n = \mu_n*\nu$, $n\in\N$.

\begin{proposition}[Almost diophantine property]
\label{pr:dioph}
Let $\mu$ be a probability measure on $\SL_d(\Z)$, $d\geq 2$, with some finite exponential moment.
Assume that $\mu$ acts strongly irreducibly on $\R^d$.

There exist constants $C\geq 0$ and $\eta>0$ depending only on $\mu$, such that for every Borel probability measure $\nu$ on $\Tbb^d$, for every $x\in\Tbb^d$, every $\rho>0$, and every $n\geq C\abs{\log \rho}$,
\[ \nu_n(B(x,\rho)) \geq \rho^\eta
\quad\Longrightarrow\quad
x\in \sW_{\rho^{-1/10}}^{(\rho^{9/10})}.\]
\end{proposition}
\begin{proof}
By Lemma~\ref{lm:notaffine}, the manifold $G$ is not included in any proper affine subspace of $E$ and therefore, writing $D = \dim E$, the set
\[ \{\, g_1+\dots+g_D-h_1-\dots-h_D \mid g_i,h_i\in G \,\}\]
contains a non-empty open set in $E$.
This implies in particular that the map $(g_i,h_i)\mapsto \det(g_1+\dots+g_D-h_1-\dots-h_D)$ is not identically zero on $G^D\times G^D$.
By Proposition~\ref{pr:gapEscape}, we infer that there exists $c>0$ such that for every $m$ large enough,
\[ \mu_m^{\otimes 2d}\bigl( \bigl\{\, (g_i,h_i)_{1\leq i\leq D} \mid \det(\sum g_i - \sum h_i) = 0 \,\bigr\} \bigr) \leq e^{-c m}.\]
Set $\eta=\frac{c}{40Dd\lambda_1}$ and for given $\rho > 0$ choose $m \geq 1$ such that $\rho \asymp e^{-20d\lambda_1 m}$. 
Let $B=B(x,\rho)$ and assume that $\nu_n(B) \geq \rho^\eta$. Then
\begin{align*}
e^{-cm} \asymp \rho^{2D\eta}
& \leq \nu_n(B)^{{2D}}\\
%& = (\sum_{g\in\Gamma} \mu_m(g)\nu_{n-m}(g^{-1}B)}))^{{2D}}\\
& = \Bigl( \int_{\Tbb^d} \sum_{g\in\Gamma} \mu_m(g) \mathbbm{1}_{B}(gx) \dd\nu_{n-m}(x) \Bigr)^{2D}\\
& \leq \int_{\Tbb^d} \Bigl(\sum_{g\in\Gamma} \mu_m(g)\mathbbm{1}_{g^{-1}B}(x) \Bigr)^{2D} \dd\nu_{n-m}(x) \qquad \text{(by Jensen's inequality)}\\
& = \sum_{g_i,h_i} \mu_m(g_1) \dotsm \mu_m(g_D)\mu_m(h_1) \dotsm \mu_m(h_D)
\nu_{n-m}(g_1^{-1}B\cap \dotsb \cap h_D^{-1}B).
\end{align*}
This shows that the $\mu_m^{\otimes 2D}$-measure of ${2D}$-tuples of elements $g_1,\dotsc,h_D$ such that
\[\nu_{n-m}(g_1^{-1}B\cap\dots\cap h_D^{-1}B) \gg e^{-cm}\]
is at least $e^{-cm}$.
In particular, using the large deviation estimate Theorem~\ref{thm:LargeD}\ref{it:LargeDn} and the observation above on the determinant, we may find elements $g_1,\dots,h_d$ in the support of $\mu_m$ satisfying 
\[
\left\{
\begin{array}{l}
\max(\norm{g_i},\norm{h_i})\leq e^{2\lambda_1 m},\\
\det(g_1+\dots+g_D-h_1-\dots-h_D)\neq 0,\\
g_1^{-1}B\cap \dots \cap h_D^{-1}B\neq\emptyset.
\end{array}
\right.
\]
If $y \in \R^d$ represents a point in that intersection, then, writing $M = g_1+\dots+g_D-h_1-\dots -h_D$, there exists $v\in\Z^d$ such that
\[ My \in v +B(0,2D\rho),\]
whence
\begin{equation}\label{ygv}
y  \in  M^{-1}v + B(0,2D\norm{M}^{-1}\rho).
\end{equation}
Now, the matrix $M$ has integer entries, and its determinant is bounded above by $e^{2d\lambda_1 m}$, so that the entries of $M^{-1}$ are rational numbers with denominator bounded above by $e^{2d\lambda_1 m}\leq\rho^{-\frac{1}{10}}$.
Moreover, 
\[\norm{M^{-1}} \leq \norm{M}^{d-1} \det(M)^{-1} \leq e^{2(d-1)\lambda_1 m}.\]
Equality \eqref{ygv} above shows that $x=g_1y\mod\Z^d$ is at distance at most $\rho^{\frac{9}{10}}$ from a rational point with denominator at most $\rho^{-\frac{1}{10}}$.
This finishes the proof.
\end{proof}

\subsection{Bootstrapping concentration}

We now wish to combine the diophantine property of $\nu_n = \mu_n*\nu$ with the concentration statement given by Proposition~\ref{pr:step1} to obtain some concentration near rational points.
To help the reader follow our progress towards Proposition~\ref{pr:general}, we formulate another intermediate step, which is the goal of this paragraph.

\begin{proposition}[Second step: concentration around rational points]
\label{pr:step2}
Under the assumptions recalled at the beginning of this section, there exists a constant $C$ depending only on $\mu$ such that the following holds.\\
Let $t_0 \in {(0, 1/2)}$.
Assume that  for some $a_0 \in \Z^d \setminus\{0\}$,
\[\abs{\widehat{\nu_n}(a_0)} \geq t_0 \quad\text{and}\quad n \geq C \log\frac{\norm{a_0}}{t_0}.\]
Then, for every integer $m$ such that $m\geq  C \log\frac{\norm{a_0}}{t_0}$ and $n-m\geq Cm$,
\[ \nu_{n-m} \bigl(\sW_{Q}^{(Q^{-8})}\bigr) \geq t_0^C,\]
for some $Q\in[e^{\frac{m}{C}},e^{Cm}]$.
\end{proposition}

%One can reduce the interval for $Q$ to $[e^{\frac{n}{C}},e^{n/A}]$, for arbitrary $A$ (of course $C$ will depend on $A$).
%To see this, simply take $m'$ larger in the proof.

The concentration statement given by Proposition~\ref{pr:step1} is not strong enough for a direct application of the diophantine property.
We first need Lemma~\ref{lm:rewind} below to bootstrap concentration.
It is exactly the same statement as \cite[Proposition~7.2]{BFLM}, and the proof is also the same, with some minor modifications to avoid the use of the proximality assumption; we include it nonetheless, for readability.

Given a subset $X\subset\Tbb^d$ and a small parameter $\rho>0$, we shall write $X^{(\rho)}$ for the $\rho$-neighborhood of $X$.
\begin{lemma}
\label{lm:rewind}
Given $\eps>0$, there exist $c>0$ and $m_0\in\N$ so that for $m\geq m_0$, the following holds for every Borel probability measure $\nu$ on $\Tbb^d$.
Given scales $r,\rho>0$ such that $e^{2d\lambda_1 m}\rho<r$, there are scales
$r_1=e^{-(\lambda_1+\eps)m} r$ and $\rho_1=e^{-(\lambda_1-\eps)m}\rho$,
so that for every $r$-separated set $X\subset\Tbb^d$, one can construct an $r_1$-separated set $X_1\subset\Tbb^d$ of cardinality $\abs{X_1} \leq \abs{X}$ and with the property
\[ \nu(X_1^{(\rho_1)}) \geq \nu_m(X^{(\rho)})^d - e^{-c m}.\]
\end{lemma}

In the proof of \cite[Proposition~7.2]{BFLM}, the following large deviation estimate in the proximal case is used. We extend it to the non-proximal case.
\begin{lemma}\label{lm:place}
Let $\mu$ be a Borel probability measure on $\SL_d(\R)$ with some finite exponential moment, and assume that the semigroup $\Gamma$ generated by $\mu$ acts strongly irreducibly on $\R^d$.
Then, for every $\eps > 0$, there exists $c > 0$ such that for every large enough $m\in\N$,
\[\mu_{m}^{\otimes d}\Bigl\{ (g_1,\dotsc,g_d)\mid \forall v \in \R^d \setminus \{0\}, \max_i \frac{\norm{g_i v}}{\norm{v}} \geq e^{(\lambda_1 - \eps)m} \Bigr\} \geq 1 - e^{-cm}.\]
\end{lemma}

\begin{remark}
If one assumes that $\mu$ is supported on $\SL_d(\Z)$, and replaces $d$ by $d^2$, then this lemma follows directly from Proposition~\ref{pr:NCdet}.
This particular case would be sufficient for our purposes.
\end{remark}

\begin{proof}
In this proof, $c$ denotes a small positive constant, depending on $\eps$, and whose value may vary from one line to another.
Let $r$ denote the proximality dimension of $\Gamma$.
We shall use the notation introduced in the proof of Proposition~\ref{pr:LDrhokappa}. Recall that for $g \in \Gamma$, we consider its Cartan decomposition $g = k\diag(\sigma_1(g),\dotsc,\sigma_d(g))l$, where $k$ and $l$ are orthogonal matrices and $\sigma_1(g) \geq \dots \geq \sigma_d(g)$ are the singular values of $g$.
We defined 
\[
%V^+_g = k \Span(e_1,\dotsc,e_r) \; \text{ and } \;
W^-_g = l^{-1}\Span(e_{r+1},\dotsc, e_d)\]
where $(e_1,\dotsc,e_d)$ is the standard basis of $\R^d$, so that for every non-zero $v \in \R^d$,
\begin{equation}
\label{eq:sigmaV-}
\frac{\norm{gv}}{\norm{v}} \geq \dang(\R v , W^-_g) \norm{g},
\end{equation}
By the large deviation estimate Theorem~\ref{thm:LargeD}\ref{it:LargeDn} if $g_1,\dots,g_d$ are independent random variables with law $\mu_m$, then with probability at least $1-e^{-cm}$,
\begin{equation*}
\forall i\in\{1,\dots,d\},\quad
\norm{g_i} \geq e^{(\lambda_1-\eps)m}.
\end{equation*}

For a subspace $W\leq\R^d$, we let $\Nbd(W,\rho)$ denote the $\rho$-neighborhood of $W$ in $\R^d$.
It follows from the above that the lemma will be proved -- with $4d\eps$ instead of $\eps$ -- if we can show that with probability at least $1-e^{-cm}$, the intersection
\[\bigcap_{i = 1}^d \Nbd(W^-_{g_i}, e^{- 4d\eps m})\]
reduces to a ball of radius $\frac{1}{2}$.
%because then, no unit vector $v$ can lie in this intersection.

For that, we construct inductively for $k=1,\dots, d - r + 1$ a linear subspace $W_k$ of dimension $d - r  + 1 - k$, depending on $g_1,\dots,g_k$, such that
\[\bigcap_{i = 1}^k \Nbd(W^-_{g_i}, e^{- 4d\eps m}) \subset \Nbd(W_k,  e^{- 4(d+1-k) \eps m}).\]
At each step, $W_{k+1}$ is constructed in terms of $W_k$ and $g_{k+1}$ and the construction is possible with probability $1 - e^{-cm}$.

%On conclura avec l'indépendance des $(g_i)$.
For $k=1$, one may simply take $W_1=W^-_{g_1}$.
Then, suppose $W_k$ has been constructed, and let $\R w\subset W_k$ be any line.
By Theorem~\ref{thm:LargeD}\ref{it:LargeDnc}, with probability at least $1 - e^{-cm}$
\[
\frac{\norm{g_{k+1} w}}{\norm{w}} \geq e^{(\lambda_1 - \eps)m}.
\]
By Theorem~\ref{thm:LargeD}\ref{it:LargeDsv}, with probability $1-e^{-cm}$,
\begin{equation}
\label{eq:sigmalambda}
\forall j\in\{1,\dots,d\},\quad \abs{\frac{1}{m} \log \sigma_j(g_{k+1}) - \lambda_j} \leq \eps,
\end{equation}
and by a straightforward generalization of \cite[Lemma~4.1(2)]{BFLM},
\[
\frac{\norm{gw}}{\norm{w}} \leq \norm{g} d(\R w,W^-_{g_{k+1}}) + \sigma_{r+1}(g_{k+1}).
\]
Since by a theorem of Guivarc'h and Raugi \cite{GuivarchRaugi}, $\lambda_r>\lambda_{r+1}$, we deduce from the above that, provided $\eps>0$ is small enough, 
\[d(\R w, W^-_{g_{k+1}}) \geq e^{-3\eps m}.\]
This implies
%see for instance Lemma 2.9 in [Saxcé,product theorem]
that there exists a proper subspace $W_{k+1}<W_k$ such that
\[\Nbd(W_k, e^{-4(d-k+1)\eps m}) \cap \Nbd(W^-_{g_{k+1}}, e^{-4d\eps m})
\subset \Nbd(W_{k+1}, e^{-4(d-k)\eps m}).\]
This proves what we want.
\end{proof}

Before proving Lemma~\ref{lm:rewind}, observe that, by the definition of Lyapunov exponents, 
\[\lambda_1 + \dots + \lambda_d = 0.\]
Since $\lambda_1 \geq \dots \geq \lambda_d$, it follows that
\[-\lambda_d \geq -(d-1)\lambda_1.\]
Recall also that by~\cite{Furstenberg1963}, $\lambda_1 > 0$, hence
\[\lambda_1 - \lambda_d \leq d\lambda_1 < 2d \lambda_1.\]

\begin{proof}[Proof of Lemma~\ref{lm:rewind}]
First, by Jensen's inequality used in the same way as in the proof of Proposition~\ref{pr:dioph},
\begin{equation*}
\nu_m(X^{(\rho)})^d
% & = (\sum_{g\in\Gamma} \mu_m(g)\nu(g^{-1}X^{(\rho)}))^d\\
%& = (\int\sum_{g\in\Gamma} \mu_m(g)\mathbbm{1}_{g^{-1}X^{(\rho)}}(x)\dd\nu(x))^d\\
%& \leq \int(\sum_{g\in\Gamma} \mu_m(g)\mathbbm{1}_{g^{-1}X^{(\rho)}}(x))^d\dd\nu(x)\\
 \leq \sum_{g_1,\dots,g_d\in\Gamma} \mu_m(g_1)\dots\mu_m(g_d) \nu(g_1^{-1}X^{(\rho)}\cap\dots\cap g_d^{-1}X^{(\rho)}).
\end{equation*}
This implies that the set of $d$-tuples $(g_i)_{1\leq i\leq d}$ such that
\begin{equation}\label{inter}
\nu(g_1^{-1}X^{(\rho)}\cap\dots\cap g_d^{-1}X^{(\rho)})\geq \nu_m(X^{(\rho)})^d-e^{-c m}
\end{equation}
 has $\mu_m^{\otimes d}$-measure at least $e^{-c m}$.
By Theorem~\ref{thm:LargeD}\ref{it:LargeDsv} and Lemma~\ref{lm:place}, if $c$ is chosen small enough, there must exist $(g_1,\dots,g_d)$ satisfying this inequality, and moreover
\begin{equation}\label{normgi}
\forall i=1,\dots,d,\quad
\norm{g_i} \leq e^{(\lambda_1+\eps)m}
\quad\text{and}\quad
\norm{g_i^{-1}} \leq e^{(-\lambda_d+\eps)m}
\end{equation}
and
\begin{equation}\label{placegi}
\forall v \in \R^d \setminus \{0\},\quad
\max_i \frac{\norm{g_i v}}{\norm{v}} \geq e^{(\lambda_1 - \eps)m}.
\end{equation}
We fix such elements $g_1,\dots,g_d$ for the rest of the proof.

Without loss of generality, we may assume that $\eps>0$ is so small that
\[ \lambda_1-\lambda_d+3\eps < 2 d \lambda_1.\]
We claim then that the set $g_1^{-1}X^{(\rho)}\cap\dots\cap g_d^{-1}X^{(\rho)}$ is included in a union of at most $\abs{X}$ balls of radius $\rho_1=e^{-(\lambda_1-\eps)m}\rho$.
Indeed, from \eqref{normgi} one finds -- drawing a picture of $X^{(\rho)}$ and $g_i^{-1}X^{(\rho)}$ -- that given $x\in X$ and $i\geq 1$, the set $g_1^{-1}B(x,\rho)$ meets at most one component $g_i^{-1}B(y,\rho)$, $y\in X$.
Therefore, there are at most $\abs{X}$ non-empty intersections $g_1^{-1}B(x_1,\rho)\cap\dotsb \cap g_d^{-1}B(x_d,\rho)$, for $x_1,\dotsc,x_d\in X$.

If $x,y$ lie inside such an intersection, then, for each $i$, $\norm{g_i(x-y)}\leq\rho$, and \eqref{placegi} implies that $\norm{x-y}\leq e^{-(\lambda_1-\eps)m}\rho=\rho_1$.
Thus, each intersection $g_1^{-1}B(x_1,\rho)\cap \dots \cap g_d^{-1}B(x_d,\rho)$ is included in a ball of radius $\rho_1$.

Finally, using~\eqref{normgi} again, we see that these intersections are separeated by at least $r_1=e^{-(\lambda_1+\eps)m} r$, and the proposition follows.
\end{proof}

To prove Proposition~\ref{pr:step2}, we proceed as follows.
Applying first Proposition~\ref{pr:step1}, we shall obtain $m_0\in\N$ and scales $\rho_0$ and $r_0$, together with an $r_0$-separated set $X_0\subset\Tbb^d$ such that
\[ \nu_{n-m_0}(X_0^{(\rho_0)}) \geq t_0^C.\]
The idea is then to reduce the radius of the balls, by an iterated application of Lemma~\ref{lm:rewind}.
We shall thus obtain an increasing sequence of integers $m_k$ and a decreasing sequence of scales $\rho_k$ and $r_k$, $k=1,2,\dots$ together with an $r_k$-separated set $X_k\subset\Tbb^d$ such that $\abs{X_k}\leq \abs{X_0}$ and
\[ \nu_{n-m_k}(X_k^{(\rho_k)}) \geq t_0^{C_k}.\]
Once we arrive at a scale $\rho_k$ such that
\[ \abs{X_k} \leq \abs{X_0} \leq \rho_k^{-\frac{\eta}{2}},\]
we shall be able to use the diophantine property of the random walk to conclude.
Now let us turn to the detailed proof.

\begin{proof}[Proof of Proposition~\ref{pr:step2}]
Let $C_0$ and $\sigma > \tau>0$ be the constants given by Proposition~\ref{pr:step1}.
 and $C'$ and $\eta>0$ the ones given by Proposition~\ref{pr:dioph}.
Then write
\[ m = m_0 + km_+,\]
%\[ n = m_0 + km_+ + m',\]
where
\[ m_0\geq C_0|\log t_0|,
\qquad
\frac{\tau m_0}{8d\lambda_1}
\leq m_+ \leq
\frac{\tau m_0}{4d\lambda_1} \]
and
\[ k = \big\lceil \frac{32d^2\sigma}{\eta\tau}\big\rceil =O_\mu(1).\]
%and
%\[ C'(m_0+km_+)\lambda_1 \leq m' \leq 2C'(m_0+km_+)\lambda_1.\]
This is feasible provided $m\geq C\abs{\log t_0}$, where $C\geq 0$ depends on $\mu$ via the constants $C_0$, $\tau$, etc.
Note that within constants depending only on $\mu$,
\[ m \asymp m_0 \asymp m_+.\]

By Proposition~\ref{pr:step1} applied to
\[ \nu_n = \mu_{m_0}*\nu_{n-m_0},\]
there exist scales
\[ \rho_0=e^{-\sigma m_0} \norm{a_0}^{-1}
\quad\text{and}\quad
r_0 = e^{\tau m_0}\rho_0\]
together with an $r_0$-separated subset $X_0\subset\Tbb^d$ such that
\[ \nu_{n-m_0}(X_0^{(\rho_0)}) \geq t_0^{C_0}.\]
Note that, since $X_0$ is $r_0$-separated, 
\[ \abs{X_0} \ll_d r_0^{-d} \leq e^{d(\sigma - \tau)m_0} \norm{a_0}^d.\]
Thus if $C$ was chosen large enough, we have
\[ \abs{X_0} \leq e^{d\sigma m_0}.\]

Choose $\eps>0$ such that $k\eps < d\lambda_1$ so that
\[2 k\eps m_+ < \frac{\tau m_0}{2}\]
and apply Lemma~\ref{lm:rewind} to
\[ \nu_{n - m_0} = \mu_{m_+}*\nu_{n - m_0 - m_+}.\]
This is allowed since by our choice of parameters
\[ e^{2d\lambda_1 m_+}\rho_0 \leq e^{\frac{\tau m_0}{2}}\rho_0 < r_0.\]
This yields scales
\[ \rho_1 = e^{-(\lambda_1-\eps)m_+} \rho_0
\quad\text{and}\quad
r_1 = e^{-(\lambda_1+\eps)m_+}r_0\]
together with an $r_1$-separated subset $X_1$ such that $\abs{X_1} \leq \abs{X_0}$ and
\[ \nu_{n-m_0-m_+}(X_1^{(\rho_1)}) \geq t_0^{C_1},\]
provided $m$
% and hence $m_+$,
is large enough to ensure that  $e^{-cm_+}<t_0^{C_1}$.
We may repeat this procedure at least $k$ times, and therefore obtain a sequence of scales defined inductively by
\[ \rho_{i+1}=e^{-(\lambda_1-\eps)m_+}\rho_i
\quad\text{and}\quad
r_{i+1} = e^{-(\lambda_1+\eps)m_+}r_i.\]
Indeed, our choice of $\eps$ ensures that for every $i \leq k$,
\[ e^{2 d\lambda_1 m_+} \rho_{i} \leq e^{2 d\lambda_1 m_+ + 2i \eps m_+ - \tau m_0}r_{i} < r_{i}.\]
In the end, we obtain scales $\rho_k$ and $r_k$, and a set $X_k$ with
\[ \abs{X_k} \leq \abs{X_0} \leq e^{d\sigma m_0}\]
such that
\begin{equation}\label{xklarge}
 \nu_{n-m}(X_k^{(\rho_k)}) \geq t_0^{C_k}.
\end{equation}
Moreover, 
\[ \rho_k = e^{-k(\lambda_1-\eps)m_+}\rho_0 \leq e^{- \frac{k \lambda_1 m_+}{2}} \leq e^{-\frac{2d\sigma m_0}{\eta}}\]
so that
\[ |X_k| \leq \rho_k^{-\frac{\eta}{2}}.\]
Therefore, adjusting slightly the values of the constants, we may restrict $X_k$ to the points $x \in X_k$ satisfying
\[ \nu_{n-m}(B(x,\rho_k)) \geq \rho_k^{\eta},\]
while preserving \eqref{xklarge}.

Note that we also have $\rho_k^{-1} \leq e^{(\sigma + \lambda_1) m} \norm{a_0}$. 
Thus, if $C$ was chosen large enough, then $n-m\geq C m\geq C'\abs{\log\rho_k}$, 
and we may conclude by Proposition~\ref{pr:dioph} that
\[ X_k^{(\rho_k)} \subset \sW_{\rho_k^{-1/10}}^{(\rho_k^{8/10})}.\]
This proves the proposition with $Q = \rho_k^{-1/10}$ in the desired range.
\end{proof}

\subsection{End of the proof of Proposition~\ref{pr:general}: near rational points}

The end of the proof of Proposition~\ref{pr:general} is based on an argument similar in spirit to the one used in Lemma~\ref{lm:rewind}, to bootstrap concentration.
The proposition we shall need is again taken from \cite{BFLM}, where it appears as \cite[Proposition~7.4]{BFLM}.
The proof we present follows closely the one given in \cite{BFLM}, but the key Lemma~\ref{lm:nombre} below, analogous to \cite[Lemma~7.10]{BFLM}, is proved using a new argument, which avoids using a regularity property of the $\mu$-stationary measure on the projective space, only available with a proximality assumption.

\begin{proposition}
\label{pr:rewindbis}
Let $\mu$ be a probability measure on $\SL_d(\Z)$ with some finite exponential moment and acting strongly irreducibly on $\R^d$.
Given $\eps>0$, there exist $m_*$ and $\omega>0$ such that if $\rho>0$, $Q\geq 1$ and $m\geq m_*$ satisfy
\[ e^{d\lambda_1 m}\rho < Q^{-2},\]
and $\nu$ is any Borel probability measure on $\Tbb^d$, then
\[ \nu \Bigl(\sW_Q^{(e^{-(\lambda_1-\eps) m}\rho)}\Bigr) \geq \nu_{m} \bigl(\sW_Q^{(\rho)}\bigr) - e^{-\omega m}.\]
\end{proposition}

The proof of this proposition is based on the following lemma.

\begin{lemma}
\label{lm:nombre}
Let $\mu$ be a Borel probability measure on $\SL_d(\R)$ with some finite exponential moment and whose support generates a subsemigroup acting strongly irreducibly on $\R^d$.
Given $\eps>0$, there exists $\theta>0$ such that the following holds for every integer $m$ sufficiently large.

Let $A$ be a subset of $\SL_d(\R)$ such that $\mu_m(A)\geq e^{-\theta m}$.
There exists a subset $\cG=\{g_i\}_{1\leq i\leq k}$ of cardinality $k\geq e^{\theta m}$ in $A$ such that for every subset $\{g_{i_1},\dots,g_{i_d}\}$ of $d$ elements of $\cG$, for every $v$ in $\R^d$,
\[ \max_{1\leq j\leq d} \norm{g_{i_j} v} \geq e^{(\lambda_1-\eps)m} \norm{v}.\]
\end{lemma}
\begin{proof}
Having fixed $\eps>0$, let
\[ \cT = \bigl\{\, (h_1,\dots,h_d) \mid \forall v \in \R^d \setminus \{0\}, \max_i \frac{\norm{h_i v}}{\norm{v}}< e^{(\lambda_1 - \eps)m} \,\bigr\}.\]
By Lemma~\ref{lm:place}, there exists $c > 0$ such that for every large enough $m$,
\[\mu_{m}^{\otimes d}(\cT) \leq e^{-cm}.\]
We shall prove that the lemma holds with $\theta=\frac{c}{d2^{d+1}}$.
Let
\[ A_1=\bigl\{\, g\in A \mid 
\mu_m^{\otimes d-1}(\{\,(h_2,\dots,h_{d}) \mid (g,h_2,\dots,h_d)\in\cT\,\}) \geq e^{-cm/2} \,\bigr\}.\]
Then
\[ e^{-cm} \geq \mu_m^{\otimes d}(\cT) \geq e^{-cm/2}\mu_m(A_1),\]
and therefore
\[ \mu_m(A_1)\leq e^{-cm/2}.\]
To construct $\cG$, we first choose $g_1\in A\setminus A_1$; this is possible because $\theta< c/2$.
Let
\[ A_2(g_1) = \bigl\{\, g\in A \mid 
\mu_m^{\otimes d-2}(\{\, (h_3,\dots,h_{d}) \mid (g_1,g,h_3\dots,h_d)\in\cT\,\})\geq e^{-cm/4} \,\bigr\}.\]
Since $g_1\not\in A_1$, we have
\begin{align*}
e^{-cm/2} & \geq \mu_m^{\otimes d-1}(\{\, (h_2,\dots,h_d) \mid (g_1,h_2,\dots,h_d)\in\cT \,\})\\
& \geq e^{-cm/4}\mu_m(A_2(g_1)),
\end{align*}
whence
\[ \mu_m(A_2(g_1)) \leq e^{-cm/4}.\]
We may therefore pick an element $g_2\in A$ such that $g_2\not\in A_1\cup A_2(g_1)$.
Then set
\[ A_3(g_1,g_2)= \bigl\{\, g \mid \mu_m^{\otimes d-3}(\{\,(h_4,\dots,h_d) \mid (g_1,g_2,g,h_4,\dots,h_d)\in\cT\,\}) \geq e^{-cm/8}\,\bigr\}\]
for which it is readily checked, using the fact $g_2\not\in A_1(g_1)$, that
\[ \mu_m(A_3(g_1,g_2))\leq e^{-cm/8}.\]
This allows us to pick $g_3\in A$ such that $g_3\not\in A_1\cup A_2(g_1)\cup A_2(g_2)\cup A_3(g_1,g_2)$.
Following this procedure, the elements  $g_1,g_2,g_3,\dots$ of $\cG$ are constructed inductively.
Once $g_1,\dots,g_k$ have been chosen, one picks $g_{k+1}\in A$ outside the union of all subsets
\begin{multline*}
A_r (g_{i_1},\dots,g_{i_{r-1}}) = \\
\bigl\{\, g \mid
\mu_m^{\otimes d-r}(\{\, (h_{r+1},\dots,h_d) \mid (g_{i_1},\dots,g_{i_{r-1}},g,h_{r+1},\dotsc,h_d)\in\cT \,\})\geq e^{-cm/2^r}\,\bigr\},
\end{multline*}
where $(g_{i_1},\dots,g_{i_{r-1}})$ can be any subset of $(g_1,\dots,g_k)$ with at most $d$ elements.
By convention, for $r=d$, write
\[ A_d(g_{i_1},\dots,g_{i_{d-1}})=\{\, g \mid (g_{i_1},\dots,g_{i_{d-1}},g)\in\cT \,\}.\]
Just as above, one checks by induction, using $g_{i_{r-1}}\not\in A_{r-1}(g_{i_1},\dots,g_{i_{r-2}})$, that
\[ \mu_m(A_r(g_{i_1},\dots,g_{i_{r-1}})) \leq e^{-cm/2^r}\leq e^{-cm/2^d}.\]
Thus, at step $k$, the union of all subsets $A_r(g_{i_1},\dots,g_{i_r})$ to be avoided has measure at most
\[ \Bigl(1 + \binom{k}{1} + \dotsb + \binom{k}{d}\Bigr) e^{-cm/2^d} \ll k^d e^{-cm/2^d}.\]
So the procedure can go on as long as $k^d e^{-cm/2^d} < \mu_m(A)$.
Since $\mu_m(A)\geq e^{-\theta m} = e^{-cm/2^{d+1}}$, one can at least reach some $k\geq e^{\frac{cm}{d2^{d+1}}}$, which proves the lemma.
\end{proof}

The rest of the proof of Proposition~\ref{pr:rewindbis} is exactly as in \cite[\S 7.D.]{BFLM}; we include it for completeness.

\begin{proof}[Proof of Proposition~\ref{pr:rewindbis}]
Once more, write
\[ \nu_m(\sW_Q^{(\rho)}) = \sum_g \mu_m(g)\nu(g^{-1}\sW_Q^{(\rho)}),\]
to observe that
\[ A = \Bigl\{\, g \in \SL_d(\R) \mid \nu \bigl( g^{-1}\sW_Q^{(\rho)} \bigr)\geq  \nu_m\bigl(\sW_Q^{(\rho)}\bigr)  - e^{-\theta m} \,\Bigr\}\]
satisfies
\[ \mu_m(A)\geq e^{-\theta m}.\]
Using the large deviation estimate for $\norm{g^{-1}}$, we may reduce $A$ without any significant loss of $\mu_m$-measure so that for every $g$ in $A$,
\[ \norm{g^{-1}} \leq e^{(-\lambda_d + \eps)m} \leq \frac{1}{2}e^{d\lambda_1 m}.\]
By Lemma~\ref{lm:nombre}, there exists a subset $\cG\subset A$ of cardinality at least $e^{\theta m}$ such that for any distinct elements $g_1,\dots,g_d$ in $\cG$, for every $v\in\R^d$,
\[ \max_{1\leq i\leq d} \norm{g_i v} \geq e^{(\lambda_1-\eps)m}\norm{v}.\]
For such elements $g_1,\dots,g_d$,
\[ g_1^{-1}\sW_Q^{(\rho)}\cap\dots g_d^{-1}\sW_Q^{(\rho)}
= \bigcup_{x_1,\dots,x_d\in \sW_Q} g_1^{-1}B(x_1,\rho)\cap\dots\cap g_d^{-1}B(x_d,\rho).\]
Now, if $u\in\R^d$ represents an element of $g_1^{-1}B(x_1,\rho)\cap\dots\cap g_d^{-1}B(x_d,\rho)$, then, for some vectors $v_i\in\Z^d$,
\begin{equation}\label{ugi}
g_iu=x_i + v_i + O(\rho),\quad i=1,\dots,d
\end{equation}
i.e.
\[
 u = g_i^{-1}(x_i+v_i) + O(e^{d\lambda_1 m}\rho).
\]
But the points $g_i^{-1}(x_i+v_i)$ are rational with denominator at most $Q$, so that they are at least $Q^{-2}$ away from one another.
Since $e^{d\lambda_1 m}\rho<Q^{-2}$, this shows that there exists $u_0\in\R^d$, rational with denominator at most $Q$ such that for each $i$,
$g_iu_0 = x_i+v_i$.
Coming back to \eqref{ugi} above, we find
\[ \norm{g_i(u-u_0)} \leq \rho,\quad i = 1,\dotsc,d\]
and by definition of the subset $\cG$,
\[ \norm{u-u_0} \leq e^{-m(\lambda_1-\eps)}\rho.\]
This shows that
\[ g_1^{-1}\sW_Q^{(\rho)}\cap\dots \cap g_d^{-1}\sW_Q^{(\rho)} \subset \sW_Q^{(e^{-(\lambda_1-\eps)m}\rho)}.\]
In other words, the family of subsets
\[\bigl\{\, g^{-1}\sW_Q^{(\rho)}\setminus \sW_Q^{(e^{-m(\lambda_1-\eps)}\rho)} \mid g\in\cG \,\bigr\}\]
has intersection multiplicity less than $d$.
Therefore,
\[ \sum_{g\in\cG} \nu_m\bigl(g^{-1}\sW_Q^{(\rho)}\setminus \sW_Q^{(e^{-m(\lambda_1-\eps)}\rho)}\bigr) \leq d,\]
and as $\abs{\cG} \geq e^{\theta m}$, there must exist $g$ in $\cG$ such that 
\[ \nu_m\bigl( g^{-1}\sW_Q^{(\rho)} \setminus \sW_Q^{(e^{-m(\lambda_1-\eps)}\rho)}\bigr) \leq de^{-\theta m}.\]
Then,
\begin{align*}
 \nu_m\bigl(\sW_Q^{(e^{-m(\lambda_1-\eps)}\rho)}\bigr) & \geq \nu_m\bigl(g^{-1}\sW_Q^{(\rho)}\bigr)-d e^{-\theta m}\\
& \geq \nu\bigl(\sW_Q^{(\rho)}\bigr) - (d+1) e^{-\theta m}\\
& \geq \nu\bigl(\sW_Q^{(\rho)}\bigr) - e^{-\omega m}.
\end{align*}
\end{proof}

We are finally ready to prove Proposition~\ref{pr:general}.

\begin{proof}[Proof of Proposition~\ref{pr:general}]
Recall the shorthand $\nu_n=\mu_n*\nu$, $n \in \N$.
By Proposition~\ref{pr:step2}, there is a constant $C_0 > 0$ depending only on $\mu$ such that for $m_0 = C_1 \log\frac{\norm{a}}{t}$ with $C_1 \geq C_0$, there exists $Q\in[e^{m_0/C_0},e^{C_0 m_0}]$ such that
\[ \nu_{n-m_0}\bigl(\sW_Q^{(Q^{-8})}\bigr) \geq t^C_0.\]

Set $\rho_0=Q^{-8}$, choose $m_1$ maximal so that $e^{d\lambda_1 m_1}\rho_0 < Q^{-2}$.
Then $e^{d\lambda_1 m_1} \asymp_\mu Q^{-2} \rho_0^{-1} = Q^6$ and hence
\begin{equation}
\label{eq:m1geqC1}
m_1 \geq \frac{6}{d\lambda_1} \log Q - O_\mu(1) \geq \frac{6 C_1}{d\lambda_1 C_0}\log\frac{\norm{a}}{t} - O_\mu(1).
\end{equation}
Thus by picking $C_1$ sufficiently large, we can make $m_1 \geq m_*$ where $m_*$ is the constant given by Proposition~\ref{pr:rewindbis} applied to $\eps := (\lambda_1 - \lambda)/2$.

It is easy to see that if $C = C(\mu,\lambda)$ is chosen large enough, every integer $n \geq C \log\frac{\norm{a}}{t}$ can be written as
\[n = m_0 + m_1 + \dotsb + m_k\]
for some $k \geq 2$ and some integers $m_2, \dotsc, m_k$ satisfying 
\begin{equation}
\label{eq:relmj+1}
\forall j = 1, \dotsc, k-1, \quad  m_j < m_{j+1} < \bigl(1 + \frac{\lambda_1 - \eps}{d \lambda_1}\bigr) m_j
\end{equation}
Define recursively for $j = 1, \dotsc, k$, 
\[\rho_j = e^{-(\lambda_1 - \eps) m_j} \rho_{j-1}.\]
Then \eqref{eq:relmj+1} implies, by a simple induction, that
\[\forall j = 1, \dotsc, k, \quad  e^{d\lambda_1 m_j} \rho_{j-1} < Q^{-2}.\]

Therefore we can apply repeatedly Proposition~\ref{pr:rewindbis} to get
\[ \nu\bigl(\sW_Q^{(\rho_k)}\bigr) \geq t^{C_0} - \sum_{j = 1}^k e^{- \omega m_j} \]
where $\omega > 0$ is a constant depending only on $\mu$ and $\lambda$.
Observe that, first, 
\[ \rho_k = e^{-(\lambda + \eps) (n - m_0)} \rho_0 \leq e^{-\lambda n} \]
provided that $C \geq \frac{\lambda + \eps}{\eps} C_1$.
Secondly, by \eqref{eq:relmj+1},
\[\sum_{j = 1}^k e^{- \omega m_j} \leq e^{-\omega m_1} \sum_{i \geq 1} e^{-\omega i} \ll_\omega e^{-\omega m_1} \]
is smaller than $t^{C_0}/2$ provided that $C_1/C_0$ is chosen large enough (recall \eqref{eq:m1geqC1}).
This finishes the proof.\qedhere

%After $i$ steps, having obtained integers $m_1,\dots, m_{i-1}$ and $\rho_i>0$ such that 
%\[ \nu_{n -\sum_{j<i} m_j}(\sW_Q^{(\rho_i)}) \geq t^C - \sum_{j<i} e^{-\omega m_j},\]
%choose $m_i$ maximal so that $e^{d\lambda_1 m_i}\rho_i < Q^{-2}$, and set $\rho_{i+1}=e^{-\lambda m_i}\rho_i$, so that by Proposition~\ref{pr:rewindbis},
%\[ \nu_{n-\sum_{j\leq i} m_j}(\sW_Q^{(\rho_{i+1})}) \geq t^C -\sum_{j\leq i}e^{-\omega m_j}.\]
%The procedure stops when $n-\sum_j m_j<m_*$, the constant given by Proposition~\ref{pr:rewindbis}.
%Then, we find, for some $r<m_*= O_\mu(1)$,
%\[ \nu_r( \sW_Q^{(\rho_0e^{-\lambda (n-m_0-r)})}) \geq t^C - \sum_j e^{-\omega m_j},\]
%and therefore, for every $\eps>0$, for $n$ large enough
%\begin{align*}
%\nu(\sW_Q^{(e^{-(\lambda-2\eps)n}})) 
%& \geq \nu_r( \sW_Q^{(e^{-(\lambda-\eps)n})}) - e^{-cn}\\
%& \geq t^C - e^{-cn} -\sum_j e^{-\omega m_j}.
%\end{align*}

%To conclude, it suffices to show that $\sum_j e^{-\omega m_j}$ is bounded above by $t^C/2$ provided $m_0$ (or $Q$) is chosen large enough.
%Now the recurrence relation defining the numbers $m_i$ shows that
%\[ \rho_{i+1}= e^{-\lambda m_i}\rho_i \quad\text{and}\quad e^{d\lambda_1 m_i} \rho_i \asymp Q^{-2}.\]
%This implies
%\[ e^{d\lambda_1 m_{i+1}} \asymp Q^{-2}\rho_{i+1}^{-1}
%= Q^{-2}e^{\lambda m_i}\rho_{i}^{-1}
%= Q^{-2} e^{\lambda(m_i+\dots+m_1)}\rho_0^{-1}
%\asymp e^{\lambda m_i}e^{d\lambda_1 m_i}\]
%and therefore
%\[ m_{i+1} = \frac{d\lambda_1+\lambda}{d\lambda_1}m_i + O(1).\]
%Since $\frac{d\lambda_1+\lambda}{d\lambda_1}>1$, this shows that the sum $\sum e^{-\omega m_i}$ converges, which is what we wanted to show.
\end{proof}

\section{Conclusion}
\label{sc:conclusion}

To conclude this paper, we mention one application of our main theorem, and then give some possible further directions of research, some of which we hope to address in publications to come.

\subsection{Expansion in simple groups modulo arbitrary integers}
\label{ss:exp}

In Section~\ref{sc:noncon}, we made use of the result of Salehi Golsefidy and Varjú \cite{SGV} about expansion in semisimple groups modulo prime -- or square-free -- numbers.
In a reverse direction, it was observed by Bourgain and Varjú \cite{bv} that the quantitative equidistibution of linear random walks on the torus of Bourgain, Furman, Lindenstrauss and Mozes \cite{BFLM} could be used to derive some expansion results in $\SL_d(\Z/q\Z)$, where $q$ runs over \emph{all} natural integers.
Because of the proximality assumption required by \cite{BFLM}, their argument could only apply to \emph{$\R$-split} simple $\Q$-groups, such as $\SL_d$.
With Theorem~\ref{thm:main} at hand, we can now generalize their result to any simple $\Q$-group.

\begin{theorem}[Super-approximation in simple groups modulo arbitrary integers]
\label{sa}
Let $S$ be a finite subset of $\GL_d(\Z)$, and $\Gamma$ the subgroup generated by $S$.
If the identity component of the Zariski closure of $\Gamma$ is an absolutely simple algebraic group, then the family of Cayley graphs $\cG(\pi_q(\Gamma),\pi_q(S))_{q\in\N}$ is a family of expanders .
\end{theorem}

One motivation for trying to prove such expansion results with no moduli restriction comes from the work of Bourgain and Kontorovich \cite{bk_zaremba,BK} where the circle method is combined with super-approximation to derive spectacular number theory results.

Most ideas for the proof of Theorem~\ref{sa} are already present in \cite{bv}, but that proof is rather involved, and only uses Theorem~\ref{thm:main} as a black box.
For this reason, we do not include it here, but rather refer the reader to the expository article~\cite{HeSaxce_expansion} for a detailed proof.
Other recent results on super-approximation, similar to the theorem above, but with different assumptions, are \cite[Appendix]{BK}, \cite{SG_II}, and \cite{FSZ}.
The survey~\cite{Breuillard_survey} also provides a very nice overview of the subject.

As observed by Salehi Golsefidy and Varjú~\cite[Question 2]{SGV}, one should expect the theorem to hold with the weaker assumption that the Zariski closure of $\Gamma$ is perfect.
To prove such a result, if one wants to exploit some equidistribution result on the torus similar to Theorem~\ref{thm:main}, one should relax the irreducibility assumption, which leads us to the second point of this conclusion.

\subsection{Without irreducibility}
The only obvious obstruction to equidistribution is when the random walk is trapped in a rational coset of a subtorus that is obtained as the image of a $\Gamma$-invariant rational subspace via the projection $\R^d \to \Tbb^d$.
Thus, in order to prove equidistribution of a linear random walk on the torus, it may be more natural only to assume the action of $\Gamma$ to be irreducible on $\Q^d$, rather than $\R^d$.

Indeed, for example, assume that the group generated by the random walk has semisimple Zariski closure without compact factor and acts strongly irreducibly on $\Q^d$.
Then Guivarc'h and Starkov~\cite{GuivarchStarkov} and independently Muchnik~\cite{Muchnik} showed that every proper closed invariant subset is a finite set of rational points. 
Moreover, under the same assumption, Benoist and Quint~\cite[Corollary~1.4]{bq2} showed that the only non-atomic stationary measure on $\Tbb^d$ is the Haar measure. See also Benoist-Quint~\cite{bq3} for a result on equidistribution of trajectories. 

Similarly, Theorem~\ref{thm:main} should remain valid if one only assumes the irreducibility of the action of $\Gamma$ on $\Q^d$, as long as the Zariski closure of $\Gamma$ is semisimple without compact factor.

The general approach used here should work in this setting, but there is one important difference: the algebra $E$ generated by $\Gamma$ will no longer be simple, but only semisimple.
In particular, the rescaled measure $\tilde{\mu}_n$ studied in Section~\ref{sc:noncon} may very well be concentrated on a proper ideal of $E$.
One therefore needs to modify several of our arguments to adapt the proof to this more general setting.

%We hope to address this issue in forthcoming work in collaboration with Elon Lindenstrauss.
%\comm{En fait, on pourrait même espérer un résultat dans l'esprit de notre article sur le somme-produit dans les représentations des groupes de Lie : si on part d'un point hors de toute classe à gauche d'un tore préservé par l'action de $\Gamma$, alors, la marche aléatoire s'équidistribue, si $G$ est semisimple, ou parfait, ou même engendré par des unipotents.}

Furthermore, the question of equidistribution is still interesting even without any irreducibility assumption. For example, the above-mentioned work of Benoist and Quint gives a classification of orbit closures and stationary measures under the assumption that the Zariski closure of the group is semisimple.
In another direction, Bekka and Guivarc'h~\cite{BekkaGuivarch} showed that the measure preserving action of a subgroup $\Gamma < \SL_d(\Z)$ on $\Tbb^d$ has a spectral gap if and only if there is no nontrivial $\Gamma$-invariant torus factor on which $\Gamma$ acts as a virtually abelian group.

\subsection{The two other assumptions}

First, we believe that the Theorem~\ref{thm:main} is still valid even if one does not require the group $\Gbb$ to be Zariski connected.
In fact, many arguments in our proof still works without this assumption, but, as is the case without the irreducibility assumption, the rescaled measures $\tilde{\mu}_n$ may concentrate near a proper subspace of $E$: the algebra generated by the connected component of $G$.
This leads to several technical difficulties when trying to prove a flattening statement.

\smallskip

Second, it would be an interesting problem to determine what moment conditions are really necessary in order to have the convergence statement of Theorems~\ref{thm:easy} and \ref{thm:main}.
It seems plausible for example that Theorem~\ref{thm:easy} holds with the weaker assumption of a moment of order $1$:
$ \int \log\norm{g}\dd\mu(g) < \infty$.
Even a counter-example to Theorem~\ref{thm:easy} without any moment condition would be interesting.

\subsection{Spaces of lattices}

Given the results of Benoist and Quint \cite{bq2} classifying stationary measures on the space of lattices $\SL_d(\R)/\SL_d(\Z)$, it is very natural to ask whether one can obtain an analog of Theorem~\ref{thm:main} in this setting.
Even the following qualitative equidistribution problem is still open \cite[\S 5.4. Question~3]{bqintro}.

\smallskip

Let $\mu$ be a measure on $\SL_d(\R)$ generating a Zariski dense subgroup $\Gamma$, and $x$ a point in $\SL_d(\R)/\SL_d(\Z)$ with infinite $\Gamma$-orbit.
Show that the sequence of measures $(\mu_n*\delta_x)_{n\geq 1}$ converges to the Haar measure as $n$ goes to infinity.

\smallskip

\subsection*{Acknowledgements.}
We are indebted to Emmanuel Breuillard for several ideas used in \S\S 3.2 and 3.3 and for sharing his unpublished note \cite{Breuillard} on non-concentration estimates for random matrix products.
It is a pleasure to thank him, as well as Richard Aoun, Yves Benoist, Elon Lindenstrauss, Jean-François Quint and Péter Varjú, for useful and motivating discussions.
We also thank the (unfortunately) anonymous referees for their careful proofreading of the manuscript and their helpful suggestions.
\bibliographystyle{abbrv} %alpha-fr abbrv-fr plain-fr acm amsplain aomplain
\bibliography{bib_nonproximal}

\begin{thebibliography}{10}

\bibitem{aoun}
R.~Aoun.
\newblock Transience of algebraic varieties in linear groups---applications to
  generic {Z}ariski density.
\newblock {\em Ann. Inst. Fourier (Grenoble)}, 63(5):2049--2080, 2013.

\bibitem{BekkaGuivarch}
B.~Bekka and Y.~Guivarc'h.
\newblock On the spectral theory of groups of affine transformations of compact
  nilmanifolds.
\newblock {\em Ann. Sci. \'{E}c. Norm. Sup\'{e}r. (4)}, 48(3):607--645, 2015.

\bibitem{bq1}
Y.~Benoist and J.-F. Quint.
\newblock Mesures stationnaires et ferm\'{e}s invariants des espaces
  homog\`enes.
\newblock {\em Ann. of Math. (2)}, 174(2):1111--1162, 2011.

\bibitem{bqintro}
Y.~Benoist and J.-F. Quint.
\newblock Introduction to random walks on homogeneous spaces.
\newblock {\em Jpn. J. Math.}, 7(2):135--166, 2012.

\bibitem{bq2}
Y.~Benoist and J.-F. Quint.
\newblock Stationary measures and invariant subsets of homogeneous spaces
  ({II}).
\newblock {\em J. Amer. Math. Soc.}, 26(3):659--734, 2013.

\bibitem{bq3}
Y.~Benoist and J.-F. Quint.
\newblock Stationary measures and invariant subsets of homogeneous spaces
  ({III}).
\newblock {\em Ann. of Math. (2)}, 178(3):1017--1059, 2013.

\bibitem{BenoistQuint}
Y.~Benoist and J.-F. Quint.
\newblock {\em Random walks on reductive groups}, volume~62 of {\em Ergebnisse
  der Mathematik und ihrer Grenzgebiete. 3. Folge. A Series of Modern Surveys
  in Mathematics}.
\newblock Springer, Cham, 2016.

\bibitem{Borel}
A.~Borel.
\newblock {\em Linear algebraic groups}, volume 126 of {\em Graduate Texts in
  Mathematics}.
\newblock Springer-Verlag, New York, second edition, 1991.

\bibitem{BougerolLacroix}
P.~Bougerol and J.~Lacroix.
\newblock {\em Products of random matrices with applications to
  {S}chr\"{o}dinger operators}, volume~8 of {\em Progress in Probability and
  Statistics}.
\newblock Birkh\"{a}user Boston, Inc., Boston, MA, 1985.

\bibitem{Bourgain2009}
J.~Bourgain.
\newblock Multilinear exponential sums in prime fields under optimal entropy
  condition on the sources.
\newblock {\em Geom. Funct. Anal.}, 18(5):1477--1502, 2009.

\bibitem{Bourgain2010}
J.~Bourgain.
\newblock The discretized sum-product and projection theorems.
\newblock {\em J. Anal. Math.}, 112:193--236, 2010.

\bibitem{BFLM}
J.~Bourgain, A.~Furman, E.~Lindenstrauss, and S.~Mozes.
\newblock Stationary measures and equidistribution for orbits of nonabelian
  semigroups on the torus.
\newblock {\em J. Amer. Math. Soc.}, 24(1):231--280, 2011.

\bibitem{bourgaingamburd_sl2p}
J.~Bourgain and A.~Gamburd.
\newblock Uniform expansion bounds for {C}ayley graphs of {${\rm
  SL}_2(\mathbb{F}_p)$}.
\newblock {\em Ann. of Math. (2)}, 167(2):625--642, 2008.

\bibitem{BourgainGamburd_modpnII}
J.~Bourgain and A.~Gamburd.
\newblock Expansion and random walks in {$\SL_d(\Z / p^n\Z)$}. {II}.
\newblock {\em J. Eur. Math. Soc. (JEMS)}, 11(5):1057--1103, 2009.
\newblock With an appendix by Bourgain.

\bibitem{BK}
J.~Bourgain and A.~Kontorovich.
\newblock On the local-global conjecture for integral {A}pollonian gaskets.
\newblock {\em Invent. Math.}, 196(3):589--650, 2014.
\newblock With an appendix by P\'{e}ter P. Varj\'{u}.

\bibitem{bk_zaremba}
J.~{Bourgain} and A.~{Kontorovich}.
\newblock {On Zaremba's conjecture}.
\newblock {\em {Ann. Math. (2)}}, 180(1):137--196, 2014.

\bibitem{bourgainkonyagin}
J.~Bourgain and S.~V. Konyagin.
\newblock Estimates for the number of sums and products and for exponential
  sums over subgroups in fields of prime order.
\newblock {\em C. R. Math. Acad. Sci. Paris}, 337(2):75--80, 2003.

\bibitem{bv}
J.~Bourgain and P.~P. Varj\'{u}.
\newblock Expansion in {$SL_d({\bf Z}/q{\bf Z}),\,q$} arbitrary.
\newblock {\em Invent. Math.}, 188(1):151--173, 2012.

\bibitem{Breuillard}
E.~Breuillard.
\newblock A non concentration estimate for random matrix products.
\newblock Unpublished notes available at
  \url{https://www.math.u-psud.fr/~breuilla/RandomProducts2.pdf}.

\bibitem{Breuillard_survey}
E.~Breuillard.
\newblock Approximate subgroups and super-strong approximation.
\newblock In {\em Groups {S}t {A}ndrews 2013}, volume 422 of {\em London Math.
  Soc. Lecture Note Ser.}, pages 1--50. Cambridge Univ. Press, Cambridge, 2015.

\bibitem{bgt_expansionlinear}
E.~Breuillard, B.~Green, and T.~Tao.
\newblock Approximate subgroups of linear groups.
\newblock {\em Geom. Funct. Anal.}, 21(4):774--819, 2011.

\bibitem{CoxLittleOShea}
D.~A. Cox, J.~Little, and D.~O'Shea.
\newblock {\em Ideals, varieties, and algorithms}.
\newblock Undergraduate Texts in Mathematics. Springer, Cham, fourth edition,
  2015.
\newblock An introduction to computational algebraic geometry and commutative
  algebra.

\bibitem{saxce_producttheorem}
N.~de~Saxc\'{e}.
\newblock A product theorem in simple {L}ie groups.
\newblock {\em Geom. Funct. Anal.}, 25(3):915--941, 2015.

\bibitem{FriedMoshe}
M.~D. Fried and M.~Jarden.
\newblock {\em Field arithmetic}, volume~11 of {\em Ergebnisse der Mathematik
  und ihrer Grenzgebiete. 3. Folge. A Series of Modern Surveys in Mathematics
  [Results in Mathematics and Related Areas. 3rd Series. A Series of Modern
  Surveys in Mathematics]}.
\newblock Springer-Verlag, Berlin, second edition, 2005.

\bibitem{FSZ}
E.~Fuchs, K.~E. Stange, and X.~Zhang.
\newblock Local-global principles in circle packings.
\newblock {\em Compos. Math.}, 155(6):1118--1170, 2019.

\bibitem{Furstenberg1963}
H.~Furstenberg.
\newblock Noncommuting random products.
\newblock {\em Trans. Amer. Math. Soc.}, 108:377--428, 1963.

\bibitem{guivarch}
Y.~{Guivarc'h}.
\newblock {Produits de matrices al\'eatoires et applications aux propri\'et\'es
  g\'eom\'etriques des sous-groupes du groupe lin\'eaire. (Products of random
  matrices and applications to geometric properties of subgroups of linear
  groups)}.
\newblock {\em {Ergodic Theory Dyn. Syst.}}, 10(3):483--512, 1990.

\bibitem{GuivarchRaugi}
Y.~Guivarc'h and A.~Raugi.
\newblock Frontière de {F}urstenberg, propriétés de contraction et
  théorèmes de convergence.
\newblock {\em Z. Wahrsch. Verw. Gebiete}, 69(2):187--242, 1985.

\bibitem{GuivarchStarkov}
Y.~Guivarc'h and A.~N. Starkov.
\newblock Orbits of linear group actions, random walks on homogeneous spaces
  and toral automorphisms.
\newblock {\em Ergodic Theory Dynam. Systems}, 24(3):767--802, 2004.

\bibitem{He2016}
W.~He.
\newblock Discretized sum-product estimates in matrix algebras.
\newblock {\em J. Anal. Math.}, 139(2):637--676, 2019.

\bibitem{he_projection}
W.~He.
\newblock Orthogonal projections of discretized sets.
\newblock {\em J. Fractal Geom.}, 7(3):271--317, 2020.

\bibitem{he_schubert}
W.~He.
\newblock Random walks on linear groups satisfying a {S}chubert condition.
\newblock {\em Israel J. Math.}, 238(2):593--627, 2020.

\bibitem{HeSaxce_expansion}
W.~He and N.~de~Saxc\'{e}.
\newblock Trou spectral dans les groupes simples.
\newblock Available at
  \url{https://www.math.univ-paris13.fr/~desaxce/publications/sg_qarbitraire.pdf}.

\bibitem{helfgott_sl2}
H.~A. Helfgott.
\newblock Growth and generation in {${\rm SL}_2(\mathbb{Z}/p\mathbb{Z})$}.
\newblock {\em Ann. of Math. (2)}, 167(2):601--623, 2008.

\bibitem{HLW}
S.~Hoory, N.~Linial, and A.~Wigderson.
\newblock Expander graphs and their applications.
\newblock {\em Bull. Amer. Math. Soc. (N.S.)}, 43(4):439--561, 2006.

\bibitem{LangWeil}
S.~Lang and A.~Weil.
\newblock Number of points of varieties in finite fields.
\newblock {\em Amer. J. Math.}, 76:819--827, 1954.

\bibitem{LePage}
E.~Le~Page.
\newblock Th\'{e}or\`emes limites pour les produits de matrices al\'{e}atoires.
\newblock In {\em Probability measures on groups ({O}berwolfach, 1981)}, volume
  928 of {\em Lecture Notes in Math.}, pages 258--303. Springer, Berlin-New
  York, 1982.

\bibitem{Li2018}
J.~{Li}.
\newblock {Discretized Sum-product and Fourier decay in $\mathbb{R}^n$}.
\newblock {\em arXiv e-prints}, page arXiv:1811.06852, Nov 2018.
\newblock arXiv:1811.06852, to appear in Journal d'Analyse Mathématique.

\bibitem{Muchnik}
R.~Muchnik.
\newblock Semigroup actions on {$\mathbb{T}^n$}.
\newblock {\em Geom. Dedicata}, 110:1--47, 2005.

\bibitem{nori}
M.~V. Nori.
\newblock On subgroups of {${\rm GL}_n({\bf F}_p)$}.
\newblock {\em Invent. Math.}, 88(2):257--275, 1987.

\bibitem{pyberszabo}
L.~Pyber and E.~Szab\'{o}.
\newblock Growth in finite simple groups of {L}ie type.
\newblock {\em J. Amer. Math. Soc.}, 29(1):95--146, 2016.

\bibitem{raghunathan}
M.~S. Raghunathan.
\newblock {\em Discrete subgroups of {L}ie groups}.
\newblock Springer-Verlag, New York-Heidelberg, 1972.
\newblock Ergebnisse der Mathematik und ihrer Grenzgebiete, Band 68.

\bibitem{SG_II}
A.~Salehi~Golsefidy.
\newblock Super-approximation, {II}: the {$p$}-adic case and the case of
  bounded powers of square-free integers.
\newblock {\em J. Eur. Math. Soc. (JEMS)}, 21(7):2163--2232, 2019.

\bibitem{SGV}
A.~Salehi~Golsefidy and P.~P. Varj\'{u}.
\newblock Expansion in perfect groups.
\newblock {\em Geom. Funct. Anal.}, 22(6):1832--1891, 2012.

\bibitem{Tao2008}
T.~Tao.
\newblock Product set estimates for non-commutative groups.
\newblock {\em Combinatorica}, 28(5):547--594, 2008.

\bibitem{TaoVu}
T.~Tao and V.~H. Vu.
\newblock {\em Additive combinatorics}, volume 105 of {\em Cambridge Studies in
  Advanced Mathematics}.
\newblock Cambridge University Press, Cambridge, 2010.

\bibitem{VanderWaerden}
B.~L. {Van der Waerden}.
\newblock {\em Modern Algebra. Volume II. Based in part on lectures by E. Artin
  and E. Noether.}
\newblock Frederick Ungar Publishing Co., New York, transl. from the 3rd german
  edition, 1950.

\end{thebibliography}

\end{document}